\documentclass[11pt, reqno]{amsart}

\usepackage[noend]{algpseudocode}
\usepackage{algorithmicx,algorithm}
\usepackage{fullpage}

\usepackage[colorlinks,
            linkcolor=black,
            anchorcolor=blue,
            citecolor=green
            ]{hyperref}
\usepackage{cite}

\usepackage{bm}
\usepackage{amsmath}
\usepackage{amssymb}
\usepackage{amsfonts}
\usepackage{amsthm}
\usepackage{mathrsfs}
  \theoremstyle{plain} \newtheorem{thm}{Theorem}[section]
  \theoremstyle{plain} \newtheorem{lem}[thm]{Lemma}
  \theoremstyle{plain} \newtheorem{prop}[thm]{Proposition}
  \theoremstyle{plain} \newtheorem{cor}[thm]{Corollary}
  \theoremstyle{plain} \newtheorem{defi}[thm]{Definition}
  \theoremstyle{plain} 
    \theoremstyle{plain} \newtheorem{aum}[thm]{Assumption}
  \theoremstyle{remark} \newtheorem{rmk}[thm]{Remark}
  \theoremstyle{remark} 
  \theoremstyle{remark}

\usepackage{paralist}

\newcommand{\beq}{\begin{equation}}
\newcommand{\eeq}{\end{equation}}
\newcommand{\bals}{\begin{align*}}  
\newcommand{\bal}{\begin{align}}

\usepackage{epsfig}

\usepackage{multirow}
\usepackage{subfigure}
\usepackage{graphicx}
\usepackage{float}
\usepackage{booktabs}
\usepackage{enumerate}

\numberwithin{equation}{section}

\usepackage{fancyhdr}
\title[On Threshold Solutions of equivariant CSS]{Threshold solutions for the focusing generalized Hartree equations}

\author{Tao Zhou}
\address{School of Mathematical Sciences\\
Peking University\\
Beijing\\ China}
\email{zhoutao@pku.edu.cn}

\begin{document}
\maketitle

\begin{abstract}
	We study the global behavior of solutions to the focusing generalized Hartree equation with $H^1$ data at mass-energy threshold in the inter-range case. In the earlier works of Arora-Roudenko \cite{arora2019global}, the behavior of solutions below the mass-energy threshold was classified. In this paper, we first exhibit three special solutions: $e^{it} Q$, $Q^\pm$, where $Q^\pm$ exponentially approach to the $e^{it} Q$ in the positive time direction, $Q^+$ blows up and $Q^-$ scatters in the negative time direction. Then we classify solutions at this threshold, showing that they behave exactly as the above three special solutions up to symmetries, or scatter or blow up in both time directions. The argument relies on the uniqueness and non-degeneracy of ground state, which we regard as an assumption for the general case.
\end{abstract}

\tableofcontents

	\section{Introduction}
In this paper, we consider the Cauchy problem
\begin{equation}
	\begin{cases}
	i \partial_t u + \Delta u + \left( |\cdot|^{-(N-\gamma)} * |u|^p \right) |u|^{p-2} u =0, \; (t,x) \in \mathbb{R} \times \mathbb{R}^N, \\
	u(0,x)  = u_0(x) \in H^1(\mathbb{R}^N),
	\end{cases}
	\label{Hartree equation}
\end{equation}
where $\gamma \in (0,N)$ and $ p \ge 2$.

The equation (\ref{Hartree equation})  is a generalization of the standard Hartree equation with $p=2$, 
\[
i \partial_t u + \Delta u + \left( |\cdot| ^{-(N-\gamma)}  * |u|^2 \right) u =0, \quad (t,x) \in \mathbb{R} \times  \mathbb{R}^N,
\]
which can be considered as a classical limit of a field equation describing a quantum mechanical non-relativistic many-boson system interacting through a two-body potential $V(x)= \frac{1}{|x|^{N-\gamma}}$, see \cite{ginibre1980class}. How it arises as an effective evolution equation in the mean-field limit of many-body quantum systems can be traced to \cite{hepp1974classical}. Lieb \& Yau \cite{lieb1987chandrasekhar} mentioned it in a context of developing theory for stellar collapse, and in particular, in the boson particles setting. A special case of the convolution with $\frac{1}{|x|}$ in $\mathbb{R}^3$ is referred as the Coulomb potential, which goes back to the work of Lieb \cite{lieb1977existence} and has been intensively studied since then, see reviews \cite{frohlich2002point} and \cite{frohlich2010classical}.

The equation (\ref{Hartree equation}) can be written as the Schrodinger-Poisson system of the form
\[
\begin{cases}
	i \partial_t u + \Delta u + V |u|^{p-2} u=0, \\
	(-\Delta)^\frac{\gamma}{2} V = (N-2) |\mathbb{S}|^{N-1} |u|^p.
\end{cases}
\]
This can be thought of as an electrostatic version of the Maxwell-Schr\"{o}dinger system, describing the interaction between the electromagnetic field and the wave function related to a quantum nonrelativistic charged particle, see \cite{lieb2003thestability} for examples.

The Cauchy problem (\ref{Hartree equation}) is locally wellposed in $H^1(\mathbb{R}^N)$ for
\begin{equation*}
	\begin{cases}
		2 \le p \le 1+ \frac{\gamma+2}{N-2}, & \text{if } N \ge 3, \\
		2 \le p < \infty, & \text{if } N=1,2,
	\end{cases}
\end{equation*}see \cite{arora2019global} for details. We denote the forward lifespan by $[0,T_+)$ and the backward by $(T_-, 0]$. If $T_+(u) < +\infty$, then $\| u(t) \|_{H^1} \to \infty$, and it is said that the solution blows up in finite time. And the same as $T_-(u)> -\infty$. 

When in its lifespan, the solutions to (\ref{Hartree equation}) satisfy mass, energy and momentum conservations:
\begin{align*}
	& M[u(t)] \triangleq \int |u(x,t)|^2 dx = M[u_0], \\
	& E[u(t)] \triangleq \frac{1}{2} \int |\nabla u(x,t)|^2 dx - \frac{1}{2p} \int \left( |\cdot|^{-(N-\gamma)} * |u(\cdot,t)|^p \right) (x) |u(x,t)|^p dx = E[u_0], \\
	& P[u](t) = \Im \int \bar{u} (x,t) \nabla u(x,t) dx = P[u_0].
\end{align*}

The equation (\ref{Hartree equation}) has several invariances: if $u(x,t)$ is a solution to (\ref{Hartree equation}), then\\
$\bullet$ by \textit{invariance under scaling}, so is $\lambda^\frac{\gamma+2}{2(p-1)} u(\lambda x, \lambda^2 t), \; \lambda >0$; \\
$\bullet$ by \textit{invariance under spatial translation}, so is $u(x+x_0,t)$, $x_0 \in \mathbb{R}^N$; \\
$\bullet$ by \textit{invariance under phase rotation}, so is $e^{i \theta_0 } u$, $\theta_0 \in \mathbb{R}$; \\
$\bullet$ by \textit{invariance under Galilean transformation}, so is $e^{i x \cdot \xi_0 } e^{-it |\xi_0|^2} u (x-2\xi_0 t,t ), \; \xi_0 \in \mathbb{R}^N$; \\
$\bullet$ by \textit{invariance under time translation}, so is $u(x,t+t_0)$, $t_0 \in \mathbb{R}$; \\
$\bullet$ by \textit{invariance under time reversal symmetry}, so is $\overline{u(x,-t)}$. 

By scaling invariance,
\[
\| \lambda^\frac{\gamma+2}{2(p-1)} u_0(\lambda x) \|_{\dot{H}^{s_c} (\mathbb{R}^N)} = \| u_0 \|_{\dot{H}^{s_c}(\mathbb{R}^N)},
\]
where 
\begin{equation}
	s_c=\frac{N}{2}- \frac{\gamma+2}{2(p-1)},
	\label{critical parameter sc}
\end{equation}
and here we call the problem (\ref{Hartree equation}) is $\dot{H}^{s_c}$-critical. In particular, when $s_c \in (0,1)$, we say the problem (\ref{Hartree equation}) is in the inter-range case.

Furthermore, as shown in \cite[Section 4]{arora2019global}, the minimization problem
\begin{equation}
\inf_{0 \not = u \in H^1 (\mathbb{R}^N)} \frac{\| u \|_2^{(N+\gamma)- (N-2)p} \| \nabla u \|_2^{Np -(N+\gamma)}}{\int_{\mathbb{R}^N} \left( |\cdot|^{-(N-\gamma)} * |u|^p \right) |u|^p dx}
\label{Weinstein-type minimization problem}
\end{equation}
has a minimizer $Q$, and we call it the ground state of problem (\ref{Weinstein-type minimization problem}). As shown in \cite[Section 4]{arora2019global}, such minimizer $Q$ is a strictly positive, radially symmetric and decreasing function. Moreover, after scaling if necessary, $Q$ satisfies the following Euler-Lagrange equation
\begin{equation}
	-\Delta Q + Q - \left( |\cdot|^{-(N-\gamma)} * Q^p \right) Q^{p-1}=0.
	\label{ground state: equation}
\end{equation}

As for the uniqueness and nondegeneracy of the ground state, where the nondegeneracy means that
\[
\ker L_+ = \text{span} \{ \partial_{x_1} Q,..., \partial_{x_N} Q \} \text{ and } L_+ \text{ is defined in } (\ref{linearized operator L+}),
\]
it is an intricate issue for general $(\gamma, p ,N)$ and  is still an open question at least to the author's knowledge. As for the uniqueness of ground state to (\ref{ground state: equation}) for $(\gamma,p,N)=(2,2,3)$, it dates back to \cite{lieb1977existence}, later it was extended to dimension $N=4$ by Krieger, Lenzmann and Rapha\"{e}l in \cite{krieger2008stability}, and Arora and Roudenko \cite{arora2019global} generalized to $3 \le N \le 5$ with $(\gamma,p)=(2,2)$. In addition, Lenzmann dealt with the non-degeneracy of the ground state when $(\gamma,p,N)=(2,2,3)$ in \cite{lenzmann2009uniqueness}. With the satisfying result for $(\gamma,p)=(2,2)$ as a preliminary, there are also some related results for $(\gamma,p)$ sufficiently closed to $(2,2)$ by the perturbation method. As for $(\gamma,p,N)=(2,2+\varepsilon,3)$ with $0< \varepsilon \ll 1$, Xiang \cite{xiang2016uniqueness} proved the uniqueness and the nondegeneracy of the ground state. In recent work by Li \cite{zexing2022priori}, he extended the results in \cite{xiang2016uniqueness} and proved the uniqueness and nondegeneracy of ground state for $(\gamma,p)$ close to $(2,2)$ when $N \in \{ 3,4,5\}$. Furthermore, when $N \ge 3$, Seok \cite{seok2019limit} checked the validity of the uniqueness and non-degeneracy of the ground state for $p \in \left(2,\frac{2N}{N-2} \right)$ and $\gamma$ sufficiently close to $0$ or $p \in \Big[ 2, \frac{2N}{N-2} \Big)$ and $\gamma$ sufficiently close to $N$.

Next, as in \cite{arora2019global}, Arora and Roudenko consider the global behavior of solutions to  (\ref{Hartree equation}) below the mass-energy threshold. Precisely, if we define

\noindent $\bullet$ \textit{renormalized mass-energy:} $\mathcal{M} \mathcal{E}[u]= \frac{M[u]^\theta E[u]}{M[Q]^\theta E[Q]}$, \\

\noindent $\bullet$ \textit{remormalized gradient}: $\mathcal{G}[u]= \frac{ \| u \|_2^\theta \| \nabla u  \|_2}{ \| Q \|_2^\theta \| \nabla Q \|_2}$, \\

\noindent $\bullet$ \textit{renormalized momentum:} $\mathcal{P}[u]=\frac{\| u \|_2^{\theta-1} P[u]}{\| Q \|_2^\theta \| \nabla Q \|_2}$, \\
where $\theta=\frac{1-s_c}{s_c}$, then
\begin{thm} [\cite{arora2019global}. classification of solutions to (\ref{Hartree equation}) below mass-energy threshold] Assume $\gamma \in (0,N)$, $p \ge 2$ and $s_c \in (0,1)$.  Let $u_0 \in H^1(\mathbb{R}^N)$ with $P[u_0]=0$ and let $u(t)$ be the corresponding solution to (\ref{Hartree equation}) with the maximal time interval of existence $(T_-,T_+)$. Suppose that $\mathcal{M} \mathcal{E}[u_0] <1$.\\
	\noindent $(1)$ If $\mathcal{G}[u_0] < 1$, then the solution exists globally in time with $\mathcal{G}[u(t)]<1$ for all $t \in \mathbb{R}$, and $u$ scatters in $H^1(\mathbb{R}^N)$; \\
	\noindent $(2)$ If $\mathcal{G} [u_0] >1$, then $\mathcal{G}[u(t)] >1$, for all $t \in (T_-,T_+)$. Moreover, if $u_0$ is of finite variance or $u_0 \in L^2(\mathbb{R}^N)$ is radial, then the solution blows up in finite time. If $u_0$ is of infinite variance and nonradial, then either the solution blows up in finite time or there exists a time sequence $t_n\to \infty$ (or $t_n \to -\infty$) such that $\| \nabla 
	u(t_n) \|_{L^2(\mathbb{R}^N)} \to \infty$.
	\label{classification of solution below threshold}
\end{thm}

When it comes to the dynamics of threshold solutions, Duyckaerts and Merle dealt with the energy-critical NLW and NLS in \cite{duyckaerts2008dynamics} and \cite{duyckaerts2009dynamic} respectively for $N \in \{ 3,4,5 \}$, with high-dimensional generalizations in \cite{li2011dynamics} and \cite{campos2020threshold} respectively. Moreover, Miao, Wu and Xu considered the energy-critical Hartree equations in \cite{miao2015dynamics}. As for the inter-range case, Duyckaerts and Roudenko solved the classification of threshold solutions to cubic NLS in $\mathbb{R}^3$ in  \cite{duyckaerts2010threshold}, and later Campos, Farah and Roudenko extended the results for NLS to all the inter-range cases in \cite{campos2020threshold}. 

The dynamics above the mass-energy threshold is mostly open. As for NLS, \cite{schlag2009stable} constructed a stable manifold near the soliton family  with an improvement to optimal topology in \cite{beceanu2009critical}, and \cite{nakanishi2012globalfornls} clarified the global dynamics of radial solutions slightly above the threshold. When it is above the threshold, after \cite{holmer2010blow} gave two blow-up criteria, \cite{duyckaerts2015going}  showed a scattering versus blow-up dichotomy for finite variance solutions. Moreover, the reader can refer to \cite{beceanu2014center}, \cite{jia2020global}, \cite{krieger2015center} and  to \cite{nakanishi2011global}, \cite{nakanishi2012global} , \cite{nakanishi2012invariant} for the dynamics of NLW and NLKG respectively.

In this paper, we mainly consider the threshold solutions to (\ref{Hartree equation}) for $s_c \in (0,1)$ and $p \ge 2$. The argument needs the uniqueness of ground state, and we ragard it as an assumption for general case:

\begin{aum} (uniqueness of the ground state) Assume $\gamma \in (0,N), p \ge 2$ and $s_c \in (0,1)$, then the minimizer of (\ref{Weinstein-type minimization problem}) is unique up to symmetries.
	\label{assumption: uniqueness of the ground state}
\end{aum}

In addition, the argument in this paper needs the spectral properties of linearized operator $\mathcal{L}$ and its coercivity, requring us figure out the null space of $\mathcal{L}$, which we also have to put as an assumption for general case.
 
\begin{aum}(nondegeneracy of the ground state)
	As for $(\gamma, p, N)$ shown above, the kernel of linearized operator $L_+$ is
	\begin{equation}
		\ker L_+ =\text{span} \Big\{ \partial_{x_1} Q , \partial_{x_2} Q, ..., \partial_{x_N} Q\Big\},
		\label{nondegeneracy of the ground state}
	\end{equation}
	where $L_+$ is defined in (\ref{linearized operator L+}).
	\label{assumption: nondegeneracy of the ground state}
\end{aum}

Next, we clarify the restrictions on parameters $(\gamma,p,N)$. First, the inter-range case means that
\begin{equation}
	s_c \in (0,1), \; p \ge 2 \text{ and } 0 < \gamma < N \quad \Rightarrow \quad 2< Np-N-\gamma <2p, \; p \ge 2 \text{ and } \gamma \in (0,N).
		\label{parameter: range 1}
\end{equation}

However, out of some technical reasons, when it comes to the threshold solution with initial data satisfying $u _0 \in L_{\text{rad}}^2$ and $\mathcal{G}[u_0]>1$ in subsection \ref{subsection: radial solutions}, we have to add more restrictions on the parameters $(\gamma,p,N)$, i.e.
\begin{equation}
     Np \le 3N+\gamma \text{ and } Np \le 2N-p+6, \; N \ge 2.
	\label{parameter: range for radial case}
\end{equation}

\begin{rmk}
	The restriction (\ref{parameter: range for radial case}) has only to be imposed in the case of classification of threshold solutions with initial data satisfying $\mathcal{G}[u_0]>1$ and $u_0 \in L_{rad}^2$, see Theorem \ref{theorem: classification of threshold solutions at zero momentum case} for details.
\end{rmk}
Next, we establish the existence of special solutions to (\ref{Hartree equation}) at the mass-energy threshold
\begin{equation}
	\mathcal{M} \mathcal{E} [u] =1.
	\label{mass-energy threshold}
\end{equation}
\begin{thm}
    Under the Assumption \ref{assumption: uniqueness of the ground state} and Assumption \ref{assumption: nondegeneracy of the ground state}, there exist two radial solutions $Q^+(t,x)$ and $Q^-(t,x)$ in $H^1(\mathbb{R}^N)$ such that \\
	\textit{a.} $M[Q^+]=M[Q^-]=M[Q]$, $E[Q^+]= E[Q^-]=E[Q]$. $[0,+\infty)$ is in the domain of the lifespan of $Q^\pm$ and
	\[
	\| Q^\pm (t) - e^{it} Q \|_{H^1(\mathbb{R}^N)} \le C e^{-e_0 t}, \quad \forall t \ge 0,
	\]
	where $e_0$ is the unique positive eigenvalue of linearized operator $\mathcal{L}$ defined in (\ref{linearized equation}).  \\
	\textit{b.} $\| \nabla Q_0^- \|_2 < \| \nabla Q \|_2$, $Q^-$ is globally defined and scatters in negative time, \\
	\textit{c.} $\| \nabla Q_0^+ \|_2 > \| \nabla Q \|_2$, and the negative time of existence of $Q^+$ is finite.
	\label{theorem: construction of special functions}
\end{thm}

Then we classify all solutions to (\ref{Hartree equation}) at the mass-energy critical threshold with zero momentum as follows:
\begin{thm}
	Under the Assumption \ref{assumption: uniqueness of the ground state} and Assumption \ref{assumption: nondegeneracy of the ground state}, let $u$ be a solution to (\ref{Hartree equation}) satisfying (\ref{mass-energy threshold}) and $P[u]=0$. \\
	\textit{a.} If $\mathcal{G}[u_0] <1$, then either $u$ scatters or $u=Q^-$ up to symmetries. \\
	\textit{b.} If $\mathcal{G}[u_0]=1$, then $u=e^{it} Q$ up to the symmetries. \\
	\textit{c.} If $\mathcal{G}[u_0]>1$, and $\begin{cases}
		u_0 \text{ is of finite variance, i.e. } |x| u_0 \in L^2,\\
	 \text{or } u_0 \in L_{rad}^2(\mathbb{R}^N) \text{ and } (N,p,\gamma) \text{ satisfies } (\ref{parameter: range for radial case}) \text{ in addition},
	\end{cases}$ then either the lifespan of $u$ is finite or $u=Q^+$ up to the symmetries.  \\
	The symmetries cited above include all symmetries except for Galilean transformation.
	\label{theorem: classification of threshold solutions at zero momentum case}
\end{thm}

As for the threshold solutions to (\ref{Hartree equation}) with nonzero momentum, we take the Galilean transformation into account. Precisely, let $\xi_0 = -\frac{P[u]}{M[u]}$, we get a solution $v$ to (\ref{Hartree equation}) with zero momentum, which is the minimal energy solution among all Galilean transformations of $u$, and $v$ satisfies
\[
M[v]=M[u], \quad E[v] = E[u] - \frac{1}{2} \frac{P[u]^2}{M[u]}, \quad \| \nabla v_0 \|_2^2 =\| \nabla u_0 \|_2^2 - \frac{P[u_0]^2}{M[u_0]}.
\]
Applying Theorem \ref{theorem: classification of threshold solutions at zero momentum case} to $v$, 
\begin{thm}
	Under the same conditions as what in Theorem \ref{theorem: classification of threshold solutions at zero momentum case}, let $u$ be a solution to (\ref{Hartree equation}) with $P[u] \not= 0$ and
	\[
	\mathcal{M} \mathcal{E}[u_0] - \frac{N(p-1)-\gamma}{N(p-1)-\gamma -2} \mathcal{P}[u_0]^2 \le 1.
	\]  
	\textit{a.} If $\mathcal{G}[v_0]^2 = \mathcal{G}[u_0]^2 - \mathcal{P}[u_0]^2 <1$, then either $u$ scatters or $u=Q^-$ up to the symmetries. \\
	\textit{b.} If $\mathcal{G}[v_0]^2 = \mathcal{G}[u_0]^2 - \mathcal{P}[u_0]^2 =1$, then $u=e^{it} Q$ up to the symmetries. \\
	\textit{c.} If $\mathcal{G}[u_0]>1$, and  $\begin{cases}
		u_0 \text{ is of finite variance, i.e. } |x| u_0 \in L^2, \\
	   \text{or } u_0 \in L_{rad}^2(\mathbb{R}^N) \text{ and } (N,p,\gamma) \text{ satisfies } (\ref{parameter: range for radial case}) \text{ in addition},
	\end{cases}$ then either the lifespan of $u$ is finite or $u=Q^+$ up to the symmetries.\\
	The symmetries cited above include all symmetries.
	\label{classification of threshold solutions with nonzero momentum case}
\end{thm}

When dimension $N=5$, if $(\gamma,p)$ is sufficiently close to $(2,2)$, then \cite{zexing2022priori} checks the uniqueness and nondegeneracy of the ground state. In addition, $(\gamma, p)$ also satisfies the restriction (\ref{parameter: range for radial case}). Consequently we can clarify the threshold solutions when $(\gamma,p)$ is sufficiently close to $(2,2)$ when $N=5$.
\begin{cor}
	If $N=5$, $p \ge 2$ and $(\gamma,p)$ is close to $(2,2)$, then the results in Theorem \ref{theorem: classification of threshold solutions at zero momentum case} and Theorem \ref{classification of threshold solutions with nonzero momentum case} hold.
\end{cor}

There are several difficulties in dealing with threshold solutions to the generalized Hartree equation. The major one is the nonlocal nature of the nonlinear term. Unlike the power type nonlinear term $|u|^{p-1} u$, the behavior of $\left( |\cdot|^{-(N-\gamma)} * |u|^p \right) |u|^{p-2} u$ at one point is determined by the global behavior of $u$, which motivates us to impose additional restrictions on parameters $(\gamma,p,N)$ out of technical reason when dealing with $L_{rad}^2$ initial data.

Another problem is to deal with the fractional, low power of the parameter $p$. The fractional power makes us unable to expand the nonlinear term directly, always with high order remaining terms. Even worse, if the parameter $p \in (2,3)$, then the mean value theorem is invalid when dealing with the remaining terms, and we need to use the fact that $f(x) = |x|^{-\alpha}$ is H\"{o}lder continuous of $\alpha$ order when $\alpha \in (0,1)$. Moreover, when constructing the special solution, instead of relying on $H^s$ estimates as in \cite{duyckaerts2010threshold}, we choose $\| \left\langle \nabla \right\rangle \cdot \|_{S(L^2)}$ estimates. However, when we are in the case for $p \in (2,3)$, the power of difference term is too low to use the contraction method, which requires to combine $S(\dot{H}^{s_c})$ norm.

Moreover, the previous papers only mention the exponential decay ground state $Q$ and its regularity. We know little about the properties of high order derivatives of ground state $Q$, and we even do not have the a priori upper bound of $|\partial^\alpha Q|$, let alone $Q \in \mathcal{S}$, so it seems not rigorous to  simply follow the idea from \cite[Corollary 3.8]{campos2020threshold}. Instead, we try to use the classical elliptic comparison theory and iteration argument to overcome it.

Finally, it is worth mentioning that the choice of orthogonal condition $ \int \Delta Q h_1=0$ used in \cite{duyckaerts2010threshold} is not suitable for our setting, then we try to use other better orthogonal conditions and find that $\int \left( |\cdot|^{-(N-\gamma)} *Q^p \right) Q^{p-1} h_1=0$ just meets our demand, which is analogous to the orthogonal condition used in \cite{campos2020threshold} and  \cite{duyckaerts2009dynamic}.

The paper is organized as follows:

In section \ref{Preliminaries}, we introduce the Strichartz pairs widely used in this paper and present some basic preliminaries for the later discussion. 

In section \ref{The linearized equation and properties of linearized operator}, we consider the linearized equation and explore the spectral properties of the linearized operator and the coercivity of the linearized energy.

In section \ref{Modulation}, we discuss the modulation stability near the ground state solution. Here we identify the spatial and phase parameters which control the variations from $e^{it} Q$ when the entire variation is small in $H^1$ norm. 

In section \ref{Convergence to Q above the threshold} and section \ref{Convergence to Q below the threshold}, we study the solutions with initial data from Theorem \ref{theorem: classification of threshold solutions at zero momentum case} part $(a)$ and $(c)$ respectively. Our main goal is to obtain the exponential convergence to $e^{it} Q$ in positive time direction. 

In section \ref{uniqueness: section}, we improve the rate of exponential convergence. After we construct special solutions $Q^\pm$ stated in Theorem \ref{theorem: construction of special functions} by a fixed point argument, we prove the rigidity, showing the solutions discussed in section \ref{Convergence to Q above the threshold} and section \ref{Convergence to Q below the threshold} are equal to $Q^\pm$ respectively. 

In Appendix \ref{Appendix: ground state}, we are devoted to the properties of the ground state and the eigenfunctions $\mathcal{Y}_\pm$ of $\mathcal{L}$ with eigenvalues $\pm e_0$, especially the exponential decay of any oder derivatives of ground state $Q$. Finally, we give some useful estimates on the nonlinear term $R(h)$ defined in (\ref{linearized equation: nonlinear term}) in Appendix \ref{Appendix: Strichartz estimates}.

\noindent \textbf{Acknowledgements.} The author thanks his advisor Baoping Liu for suggesting this topic and for valuable advice and guidance during the discussion. The author would also like to thank Zexing Li and Shengxuan Zhou for their valuable comments and suggestions. Moreover, the author thanks Guixiang Xu and Yongfu Yang for their guidance. The author was partially supported by the NSFC$12071010$ and NSFC$11631002$.

	\section{Preliminaries}
\label{Preliminaries}
\subsection{Strichartz estimates and admissible pairs}
In this section, we will introduce the admissible Strichartz pairs and recall the corresponding Strichartz estimates.
\begin{defi}
	For $s >0$, the pair $(q,r)$ is  $\dot{H}^{s}$ admissible if 
	\begin{equation}
		\frac{2}{q}+ \frac{N}{r} =\frac{N}{2}-s, \quad 2 \le q,r \le \infty, \quad \text{ and } (q,r,N) \not= (2,\infty,2).
		\label{Strichartz pair}
	\end{equation}
If $s=0$, we say that the pair $(q,r)$ is $L^2$ admissible.
\end{defi}

In order to control the constants uniformly in Strichartz estimates, we restrict the range for the pair $(q,r)$,
\begin{equation}
	\begin{cases}
		\left( \frac{2}{1-s} \right)^+ \le q \le \infty, \quad \frac{2N}{N-2s} \le r \le \left( \frac{2N}{N-2} \right)^-, \text{ if } N \ge 3, \\
		\left( \frac{2}{1-s} \right)^+ \le q \le \infty, \quad \frac{2}{1-s} \le r \le \left( \left( \frac{2}{1-s} \right)^+ \right)', \text{ if } N=2,\\
		\frac{4}{1-2s} \le q \le \infty, \quad \frac{2}{1-2s} \le r \le \infty, \text{ if } N=1.
	\end{cases}
		\label{restriction on Strichartz pair}
\end{equation}
Then we define $S(\dot{H}^s)$ norm by
	\begin{align*}
	\| u \|_{S \left( \dot{H}^{s_c} \right)} = \sup\left\lbrace \| u \|_{L_t^q L_x^r}: (q,r) \text{ satisfies  } (\ref{Strichartz pair}) \text{ and } (\ref{restriction on Strichartz pair}) \right\rbrace.
\end{align*}
Moreover, to define the corresponding dual Strichartz norm, we should set the following restrictions:
\begin{equation}
	\begin{cases}
		\left( \frac{2}{1+s} \right)^+ \le q \le \left( \frac{1}{s} \right)^-, \quad \left( \frac{2N}{N-2s} \right)^+ \le r \le \left( \frac{2N}{N-2} \right)^-, \text{ if } N \ge 3, \\
		\left( \frac{2}{1+s} \right)^+ \le q \le \left( \frac{1}{s} \right)^-, \quad \left( \frac{2}{1-s}\right)^+ \le r \le \left( \left( \frac{2}{1+s} \right)^+ \right)', \text{ if } N=2,\\
		\frac{2}{1+2s} \le q \le \left( \frac{1}{s} \right)^-, \quad \left( \frac{2}{1-s} \right)^+ \le r \le \infty, \text{ if } N=1,
	\end{cases}
	\label{restriction on dual Strichartz pair}
\end{equation}
then we define the dual Strichartz norm  as below:
\[
\| u \|_{S'(\dot{H}^{-s})} = \inf \left\lbrace \| u \|_{L_t^{q'} L_x^{r'}}: (q,r) \text{ satisfies } (\ref{Strichartz pair}) \text{ and } (\ref{restriction on dual Strichartz pair}) \right\rbrace.
\]

\begin{defi}
	In the later discussion in this paper, for given $N,p, \gamma$ and hence a fixed $s_c \in (0,1)$, we select specific Strichartz pairs as follows: \\
	$\bullet$ $L^2$-admissible pair: $S(L^2)=L_t^{q_1} L_x^{r_1} \cap L_t^{q_2} L_x^{r_2}$, where 
	\[
	(q_1, r_1)= \left( \frac{2p}{1+s_c(p-1)}, \frac{2Np}{N+\gamma} \right) \text{ and }
	(q_2, r_2) =\left( \frac{2p}{1-s_c} , \frac{2Np}{N+\gamma +2s_c p} \right);
	\]
	$\bullet $ $L^2$ dual admissible pair: $S'(L^2)= L_t^{q_1'} L_x^{r_1'}$, where $(q_1', r_1')= \left( \frac{2p}{2p-1-s_c(p-1)}, \frac{2Np}{2Np-N-\gamma} \right)$; \\
	$\bullet$ $\dot{H}^{s_c}$-admissible pair: $S(\dot{H}^{s_c})= L_t^{q_2}L_x^{r_1}$, where $(q_2, r_1)= \left( \frac{2p}{1-s_c}, \frac{2Np}{N+\gamma} \right)$; \\
	$\bullet$ $\dot{H}^{-s_c}$ dual admissible pair: 
	$S'(\dot{H}^{-s_c})= L_t^{q_3'} L_x^{r_1'}$, where $\left(q_3', r_1'\right)= \left( \frac{2p}{(2p-1)(1-s_c)}, \frac{2Np}{2Np-N-\gamma} \right)$. \\
	By the assumption on $(N,p, \gamma)$, all the Strichartz pairs shown above satisfy (\ref{restriction on Strichartz pair}) and (\ref{restriction on dual Strichartz pair}) respectively.
	\label{rmk: definition of Strichartz pair}
\end{defi}
As for the Strichartz pairs selected above, by Sobolev embedding, we have
\begin{lem}
		$\forall f \in S\left( I, \left\langle \nabla \right\rangle L^2 \right)$, then $ f \in S(I,\dot{H}^{s_c})$ and
	\begin{equation}
		\| f\|_{S(\dot{H}^{s_c})}= \| f\|_{L_t^{q_2} L_x^{r_1}} \lesssim  \big\| |\nabla|^{s_c} f \big\|_{L_t^{q_2} L_x^{r_2}} \lesssim \big\| \left\langle \nabla \right\rangle f \| _{L_t^{q_2} L_x^{r_2}} \lesssim \big\| \left\langle \nabla \right\rangle f \big\|_{S(L^2)}.
		\label{critical Strichartz norm can be bounded by H1 Strichartz norm}
	\end{equation}
\end{lem}

\begin{lem}[Kato-Strichartz estimates]
	\label{lemma: Strichartz estimates}
$\forall f \in S\left( I, \left\langle \nabla \right\rangle L^2 \right)$,  then
\begin{equation*}
	\Big\| \int_{s>t} e^{i(t-s) \Delta} f(s) ds \Big\|_{S(I, L^2)} \lesssim \| f \|_{S'(I,L^2)},
\end{equation*}
\begin{equation*}
	\Big\| \int_{s>t} e^{i(t-s) \Delta} f(s) ds \Big\|_{S(I, \dot{H}^{s_c})} \lesssim \| f \|_{S'(I, \dot{H}^{-s_c})}.
\end{equation*}
\end{lem}

\subsection{Gradient seperation}
In this subsection, we want to give a rough classification of solutions to (\ref{Hartree equation}) through the gradient. Precisely, by the same argument as what in \cite[Lemma $2.2$]{duyckaerts2010threshold},
\begin{lem}
	Under the Assumption \ref{assumption: uniqueness of the ground state}, we consider the solutions to (\ref{Hartree equation}) with initial data $u_0$ satisfying (\ref{mass-energy threshold}) and $P[u_0]=0$.\\
	(\textit{a}). If $\| u_0 \|_2^\theta \| \nabla u_0 \|_2 = \| Q\|_2^\theta \| \nabla Q \|_2$, then $u=e^{it} Q$ up to symmetries. \\
	(\textit{b}). If  $\| u_0 \|_2^\theta \| \nabla u_0 \|_2 < \| Q\|_2^\theta \| \nabla Q \|_2$, then $u$ is globally defined and
	\[
	\| u_0 \|_2^\theta \| \nabla u(t) \|_2 < \| Q\|_2^\theta \| \nabla Q \|_2, \quad \forall t \in \mathbb{R}.
	\]
	(\textit{c}). If $\| u_0 \|_2^\theta \| \nabla u_0 \|_2 > \| Q\|_2^\theta \| \nabla Q \|_2$, then
	\[
	\| u_0 \|_2^\theta \| \nabla u(t) \|_2 > \| Q\|_2^\theta \| \nabla Q \|_2, \quad \text{for all } t \text{ belonging to the lifespan of } u.
	\]
	\label{lemma: behavior of gradient of solution}
\end{lem}

\subsection{Qualitative rigidity} In this part, we give a criterion of the closeness of solution to (\ref{Hartree equation}) to the soliton family:
\begin{prop}
	Under the Assumption \ref{assumption: uniqueness of the ground state}, there exists a function $\varepsilon(\rho)$ defined for small $\rho>0$ such that $\lim\limits_{\rho \to 0} \varepsilon(\rho) =0 $ and $\forall u \in H^1$ such that
	\begin{align*}
		&\Big| \| u \|_2^2 -\| Q \|_2^2 \Big| + \Big| \| \nabla u\|_2^2 -\| \nabla Q \|_2^2 \Big| \\
		&\qquad \qquad  + 	\Big| \int \left( |\cdot|^{-(N-\gamma)} * |u|^p \right) |u|^{p} dx - \int \left( |\cdot|^{-(N-\gamma)}* Q^p \right) Q^p  dx \Big| 
		\le \rho,
	\end{align*}
	there exists $(\theta_0, x_0) \in \mathbb{R} \backslash 2 \pi \mathbb{Z} \times \mathbb{R}^N$ such that 
	\begin{equation}
		\| u - e^{i\theta_0} Q (\cdot -x_0) \|_{H^1} \le \varepsilon(\rho).
		\label{smallness of distance between u and soliton}
	\end{equation}
	In particular, for solution to (\ref{Hartree equation}) satisfying 
	\begin{equation}
		M[u]=M[Q], \; E[u]= E[Q],
		\label{condition: mass and energy threshold}
	\end{equation} 
	if we let
	\begin{equation}
		\delta(t)= \Big| \| \nabla u(t) \|_2^2 - \| \nabla Q \|_2^2 \Big|,
		\label{delta: definition, the distance}
	\end{equation}
	and assume $\delta(t) \le \rho$, then there exists $(\theta_0,x_0)$ such that (\ref{smallness of distance between u and soliton}) holds.
	\label{propostion: convergence of u to soliton}
\end{prop}
\begin{proof}
	If not, then there exists a sequence $\{u_n\}_{n=1}^\infty \subset H^1(\mathbb{R}^N)$ such that
	\begin{align}
		&  \Big| \| u_n \|_2^2 -\| Q \|_2^2 \Big| + \Big| \| \nabla u_n \|_2^2 -\| \nabla Q \|_2^2 \Big|  \notag \\
		& \qquad + \Big| \int \left( |\cdot|^{-(N-\gamma)} * |u_n|^p \right) |u_n|^p dx - \int \left( |\cdot|^{-(N-\gamma)}* Q^p \right) Q^p  dx \Big| \to 0,
		\label{assumption: closeness u to ground state Q}
	\end{align}
	but
	\[
	\inf_{(\theta_0, x_0) \in \mathbb{R} \times \mathbb{R}^N} \| u_n - e^{i \theta_0} Q(\cdot-x_0) \|_{H^1} \ge \varepsilon_0.
	\]
	By linear decomposition for $\{ u_n \}_{n=1}^\infty$ as used in \cite[Proposition $3.1$]{hmidi2005blowup}, after extracting a subsequence, there exist a family of linear profiles $\{U^j\}_{j=1}^\infty \subset H^1$ and $\{x_n^j\}$ such that
	
	\noindent $(1)$ For every $k \not= j$, $|x_n^k - x_n^j| \to \infty$, $n \to \infty$; 
	
	\noindent $(2)$ For every $l \ge 1$ and every $x \in \mathbb{R}^N$, $u_n(x)= \sum_{j=1}^l U^j(x-x_n^j) + r_n^l(x)$ with
	\begin{equation}
		\lim_{l \to \infty} \limsup_{n \to \infty} \| r_n^l \|_{L^s (\mathbb{R}^N)} =0, \quad \forall s \in \left( 2, 2^* \right), \text{ where } 2^* = 
		\begin{cases}
			\frac{2N}{N-2}, & \text{if } N\ge 3, \\
			+\infty, & \text{if } N=1,2.
		\end{cases}
		\label{smallness of remaining term}
	\end{equation}
	
	\noindent $(3)$. Orthogonality. For any fixed $l \in \mathbb{N}$,
	\begin{equation}
		\| u_n \|_2^2 = \sum_{j=1}^l \| U^j \|_2^2 + \| r_n^l \|_2^2 + o_n(1),
		\label{orthogonality: L2 sense}
	\end{equation}
	\begin{equation}
		\| \nabla u_n \|_2^2 = \sum_{j=1}^l \| \nabla U^j \|_2^2 + \| \nabla r_n^l \|_2^2 + o_n(1).
		\label{orthogonality: dot(H)1 sense}
	\end{equation}
	By (\ref{orthogonality: L2 sense}) and (\ref{orthogonality: dot(H)1 sense}), we have
	\begin{equation}
		\| Q\|_2^2= \lim_{n\to \infty} \| u_n \|_2 ^2 \ge \sum_{j=1}^l \| U^j \|_2^2, \quad  \| \nabla Q \|_2^2= \lim_{n \to \infty} \| \nabla u_n \|_2^2  \ge \sum_{j=1}^l \| \nabla U^j \|_2^2.
		\label{limitation of un can be bounded by bubbles}
	\end{equation}
	
	Moreover, aftering extracting a subsequence if necessary, by (\ref{smallness of remaining term}) and the fact that $|x_n^k - x_n^j| \to \infty$ as $n \to \infty$, we have
	\begin{equation*}
		\lim_{l \to \infty} \limsup_{n \to \infty} \iint \frac{|r_n^l(x)|^p |r_n^l(y)|^p}{|x-y|^{N-\gamma}} dxdy =0,
	\end{equation*}
	and hence
	\[
	\lim_{n \to \infty}\iint \frac{|u_n(x)|^p |u_n(y)|^p}{|x-y|^{N-\gamma}} dxdy = \lim_{l \to \infty}  \sum_{j=1}^l \iint \frac{ | U^j(x)|^p |U^j (y) |^p }{|x-y|^{N-\gamma}} dxdy.
	\]
	
	From Proposition \ref{proposition: existence of ground state}, note that
	\begin{equation*}
		\sum_{j=1}^l \iint \frac{|U^j(x)|^p |U^j(y)|^p}{|x-y|^{N-\gamma}} dxdy 
		\le C_{GN} \sum_{j=1}^l \| U^j \|_2^{N+\gamma-(N-2)p} \| \nabla U^j \|_2^{Np - (N+\gamma)},
	\end{equation*}
	we obtain that
	\begin{align}
		C_{GN}^{-1}
		& \le \liminf_{l \to \infty}  \frac{ \sum_{j=1}^l \| U^j \|_2^{N+\gamma-(N-2)p} \| \nabla U^j \|_2^{Np - (N+\gamma)} }{		\sum_{j=1}^l \iint \frac{|U^j(x)|^p |U^j(y)|^p}{|x-y|^{N-\gamma}} dxdy }  \notag\\
		&\le \liminf_{l \to \infty}  \frac{ \left( \sum_{j=1}^l \| U^j \|_2^2 \right)^\frac{N+\gamma-(N-2)p}{2}  \left( \sum_{j=1}^l \| \nabla U^j \|_2^2 \right)^\frac{Np - (N+\gamma)}{2} }{\sum_{j=1}^l \iint \frac{|U^j(x)|^p |U^j(y)|^p}{|x-y|^{N-\gamma}} dxdy }  \label{key point of proving there is only one profile}\\
		& \le \liminf_{l \to \infty}  \frac{  \| Q \|_2^{N+\gamma -(N-2)p} \| \nabla Q \|_2^{Np-(N+\gamma)}  } {\sum_{j=1}^l \iint \frac{|U^j(x)|^p |U^j(y)|^p}{|x-y|^{N-\gamma}} dxdy } \notag\\
		&= \frac{  \| Q\|_2^{N+\gamma -(N-2)p} \| \nabla Q \|_2^{Np-(N+\gamma)}  }{ \iint \frac{|Q(x)|^p |Q(y)|^p}{|x-y|^{N-\gamma}} dxdy } = C_{GN}^{-1} \notag,
	\end{align}
	thus the inequalities are all equalities. In particular, from (\ref{key point of proving there is only one profile}), there is only one nonzero profile. Without loss of generality, we may assume it is $U^1$, then
	\[
	u_n(x)= U^1(x-x_n^1)+ r_n^1(x)
	\]
	and therefore
	\[
	\frac{ \| U^1 \|_2^{N+\gamma-(N-2)p} \| \nabla U^1 \|_2^{Np-(N+\gamma)}}{ \iint \frac{|U^1(x)|^p |U^1(y)|^p}{|x-y|^{N-\gamma}} dxdy } =C_{GN}^{-1} = \frac{ \| Q \|_2^{N+\gamma-(N-2)p} \| \nabla Q \|_2^{Np-(N+\gamma)}}{ \iint \frac{|Q(x)|^p |Q(y)|^p}{|x-y|^{N-\gamma}} dxdy  } . 
	\]
	If we have the uniqueness of ground state, i.e. Assumption \ref{assumption: uniqueness of the ground state}, then there eixsts $(\mu_2, x_2,\lambda_2)\in  \overline{\mathbb{R}^+} \times \mathbb{R}^N \times \mathbb{C}$ such that
	\[
	U^1(x)= \lambda_2 Q (\mu_2 x + x_2),
	\]
	 where $\mu_2=1$ and $|\lambda_2|=1$ by (\ref{assumption: closeness u to ground state Q}), thus
	\[
	U^1(x)=e^{i \theta_2} Q(x+ x_2) \; \Rightarrow \; u_n(x+x_n^1)= e^{i \theta_2} Q(x+x_2) + r_n^1(x+x_n^1).
	\]
	Together with (\ref{orthogonality: L2 sense}) and (\ref{orthogonality: dot(H)1 sense}), we obtain that $\| r_n^1 \|_{H^1} = o_n(1)$, which means that
	\[
	\inf_{(\theta_0,x_0) \in \mathbb{R} \times \mathbb{R}^N} \| u_n - e^{i \theta_0} Q(\cdot-x_0) \|_{H^1} \le \| r_n^1 \|_{H^1} =o_n(1),
	\]
	a contradiction!
\end{proof}
\section{The linearized equation and properties of linearized operator}
\label{The linearized equation and properties of linearized operator}
\subsection{The linearized equation}
Consider the solution $u$ to (\ref{Hartree equation}) around $e^{it}Q$ and write $u$ as
\[
u(x,t)= e^{it} \left(Q(x)+ h(t,x) \right).
\]
If we write $h=h_1+ ih_2$, then $h=\left( \begin{matrix}
	h_1 \\ h_2
\end{matrix} \right)$ satisfies
\begin{equation}
	\partial_t h + \mathcal{L} h =i R (h), \quad \mathcal{L} = \left( \begin{matrix}
		0 & -L_- \\
		L_+ & 0
	\end{matrix} \right), 
	\label{linearized equation}
\end{equation}
where
\begin{equation}
	L_+ = -\Delta + 1 - p  \left(|\cdot|^{-(N-\gamma)} * \left( Q^{p-1} \cdot \right) \right) Q^{p-1} - (p-1)\left(|\cdot|^{-(N-\gamma)} * Q^p \right)Q^{p-2},
	\label{linearized operator L+}
\end{equation}

\begin{equation}
	L_- = -\Delta + 1  - \left(|\cdot|^{-(N-\gamma)} * Q^p \right) Q^{p-2}
	\label{linearized operator L-},
\end{equation}
and
\begin{align}
	R(h)
	&= \left(|\cdot|^{-(N-\gamma)} * |Q+h|^p \right) |Q+h|^{p-2} (Q+h) \notag \\
	&- \left( |\cdot|^{-(N-\gamma)} * Q^p \right) \left( Q^{p-1} + \frac{p-2}{2} Q^{p-2} \bar{h} + \frac{p}{2} Q^{p-2} h \right)  \notag \notag \\
	& - \left( |\cdot|^{-(N-\gamma)} * \left( \frac{p}{2} Q^{p-1} \bar{h} + \frac{p}{2} Q^{p-1} h \right) \right) Q^{p-1}
	\label{linearized equation: nonlinear term}.
\end{align}

\subsection{The spectral properties of the linearized operator}
In this subsection, we explore the sepctral properties of the linearized operator. Before we present the main result, we give the following auxiliary lemma about the non-negativity of linearized energy $\Phi$ on a subspace of $H^1(\mathbb{R}^N)$ with co-dimension $1$, where $\Phi$ is defined by
\begin{equation}
	\Phi(h)=\frac{1}{2} \int (L_+ h_1) h_1 dx + \frac{1}{2} \int (L_- h_2) h_2 dx.
	\label{linearized energy}
\end{equation}

\begin{lem}[non-negativity of the linearized energy]
	For any  function $h \in H^1$ satisfying 
	\begin{equation}
     \left\langle\left( |\cdot|^{-(N-\gamma)} * |Q|^p \right) Q^{p-1}, h_1 \right\rangle =0,
     	\label{orthogonal condition: orthogonal to Delta Q}
    \end{equation}
    	 we have $ \Phi (h) \ge 0$.
	\label{lemma: non-negativity of linearized operator}
\end{lem}
\begin{proof}
	Similar as the process in \cite[A.1]{duyckaerts2010threshold}, since $Z[u]$ attains its infimum at $Q$,
	\[
	I[u] \triangleq \frac{\| u \|_2^{N+\gamma-(N-2)p} \| \nabla u \|_2^{Np-(N+\gamma)}}{\| Q \|_2^{N+\gamma-(N-2)p} \| \nabla Q \|_2^{Np-(N+\gamma)}} - \frac{\int \left( |\cdot|^{-(N-\gamma)} * |u|^p \right) |u|^p dx}{\int \left( |\cdot|^{-(N-\gamma)} * Q^p \right) Q^p dx} \ge 0, \quad \forall u \in H^1(\mathbb{R}^N). 
	\]
	For any function $h \in H^1(\mathbb{R}^N)$, the function $\lambda \mapsto I[Q+\lambda h]$ with domain $\mathbb{R}$ attains its minimum at $\lambda =0$, which implies that
	\begin{equation}
		\frac{d^2}{d\lambda^2}\Big|_{\lambda=0} I[Q+\lambda h] \ge 0.
		\label{positivity of second derivative of I}
	\end{equation}
	Next, we expand $ \| Q + \lambda h \|_2^{N+\gamma-(N-2)p}$, 	$\| \nabla Q + \lambda \nabla h \|_2^{Np-(N+\gamma)} $ and $ \int \left( |\cdot|^{-(N-\gamma)} * |Q + \lambda h|^p  \right) |Q+\lambda h|^p dx $ with respect to $\lambda$ up to order $2$ respectively and then compute the expansion of $ I[Q+\lambda h]$ in $\lambda$ of order $2$. Then by $\int \nabla Q \cdot \nabla h_1 = - \int Q h_1$, which simply follows from (\ref{ground state: equation}) and (\ref{orthogonal condition: orthogonal to Delta Q}),
	\begin{align*}
		&\frac{Np-(N+\gamma)}{ \| \nabla Q \|_2^2} \Phi(h) - \frac{\left( Np-(N+\gamma) \right)(p-2)  }{2\| \nabla Q \|_2^2 } \int \left( |\cdot|^{-(N-\gamma}  * Q^p\right) Q^{p-2} |h_2|^2 dx \\
		 &-\frac{2p(Np-(N+\gamma))}{N+\gamma-(N-2) p} \left( \frac{\int \nabla Q \cdot \nabla h_1}{\int |\nabla Q|^2} \right)^2 \ge 0,
	\end{align*}
   thus it implies $\Phi(h) \ge 0$ by (\ref{parameter: range 1}).
\end{proof}
\begin{rmk}
	If we choose the orthogonal condition $\int \Delta Q h_1 dx =0$ as in \cite{duyckaerts2010threshold}, in order to get the non-negativity of linearized operator, some additional restriction on parameters $(N,p,\gamma)$ should be imposed. By contrast, if we choose $\int \left( |\cdot|^{-(N-\gamma)} * Q^p \right) Q^{p-1} h_1 dx =0$, we can get the optimal range, i.e. (\ref{parameter: range 1}).
\end{rmk}

Next, we obtain the spectral properties of linearized operator $\mathcal{L}$ as follows.
\begin{prop}
	Under the Assumption \ref{assumption: nondegeneracy of the ground state}, let $\sigma (\mathcal{L})$ be the spectrum of the linearized operator $\mathcal{L}$ defined on $L^2(\mathbb{R}^N) \times L^2(\mathbb{R}^N)$ and let $\sigma_{\text{ess}} \left( \mathcal{L} \right)$ be its essential spectrum. Then
	\[
	\sigma_{\text{ess}} (\mathcal{L}) =\{i \xi: \xi \in \mathbb{R}, |\xi| \ge 1\}, \quad \sigma(\mathcal{L}) \cap \mathbb{R} = \{ -e_0, 0, e_0\} \text{ for some } e_0>0.
	\]
	Furthermore, $e_0, -e_0$ are the only pair of eigenvalues of $\mathcal{L}$ on $\mathbb{R} \setminus \{ 0 \}$ with eigenfunctions $\mathcal{Y}_+ \in 
	\mathcal{S}$ and $\mathcal{Y}_- \in \mathcal{S}$ satisfying $\mathcal{Y}_+ = \overline{\mathcal{Y}}_-$. And if we let $\mathcal{Y}_1 =\Re \mathcal{Y}_+= \Re \mathcal{Y}_-$ and $\mathcal{Y}_2 = \Im \mathcal{Y}_+ = -\Im \mathcal{Y}_-$, then
	\[
	L_+ \mathcal{Y}_1 =e_0 \mathcal{Y}_2 \quad \text{ and } \quad L_- \mathcal{Y}_2 =-e_0 \mathcal{Y}_1.
	\]
	Moreover, as for the null space of $\mathcal{L}$, we have
	\[
	\ker \mathcal{L} = \text{span} \{ iQ, \partial_{x_1} Q,..., \partial_{x_N} Q \}.
	\]
	\label{proposition: spectral information of linearized operator L}
\end{prop}
\begin{proof}
	As for the spectral properties of $\mathcal{L}$, it suffices to consider the self-adjoint operator $\mathcal{P} = L_-^\frac{1}{2} L_+ L_-^\frac{1}{2}$ with domain $H^4(\mathbb{R}^N) \subset L^2(\mathbb{R}^N)$. Since it can be checked that $\mathcal{P}$ is a relatively compact perturbation of $(-\Delta +1)^2$, by Weyl's theorem (see \cite[Theorem $14.6$]{hislop2012introduction}), 
	\[
	\sigma_{\text{ess}} (\mathcal{P}) = [0, + \infty).
	\]
	
	Moreover, by min-max principle, following the similar argument as in \cite[Lemma $C.1$]{miao2015dynamics} or in \cite[Lemma $3.2$]{campos2020threshold}, $\mathcal{P}=L_-^\frac{1}{2} L_+ L_-^\frac{1}{2}$ has a negative eigenvalue $-e_0^2$ with associated eigenfunction $g$. Defining $\mathcal{Y}_1= L_{-}^\frac{1}{2} g, \mathcal{Y}_2 = \frac{1}{e_0} L_+ \mathcal{Y}_1$ and $\mathcal{Y}_\pm = \mathcal{Y}_1 \pm i \mathcal{Y}_2$, we obtain that $\mathcal{L} \mathcal{Y}_\pm = \pm e_0 \mathcal{Y}_\pm$.  
	
	The simplicity of eigenvalue $-e_0^2$ of $\mathcal{P}$ and the uniqueness of negative eigenvalue of $\mathcal{P}$ follow from the non-negativity of $L_+$ acting on $\Big\{ \left( |\cdot|^{-(N-\gamma)} *Q^p \right)Q^{p-1} \Big\}^\perp$, which simply follows from Lemma \ref{lemma: non-negativity of linearized operator}. As for the null space of $\mathcal{L}$, it is Assumption \ref{assumption: nondegeneracy of the ground state} and the fact that $\ker L_-= \text{span} \{ Q\}$ (see \cite[Theorem $\text{XIII. } 48$]{simon1978methods}).
	
	To prove $\mathcal{Y}_ \pm \in \mathcal{S}$, it suffices to check $\mathcal{Y}_1, \mathcal{Y}_2 \in \mathcal{S}$. First, by the bootstrap argument, we know that $\mathcal{Y}_1, \mathcal{Y}_2 \in H^\infty$, which further implies that $\mathcal{Y}_1, \mathcal{Y}_2 \in  \left( \cap_{k \ge 1} W^{k,\infty}  \right) \cap C^\infty$ by Sobolev embedding. Note that $\mathcal{Y}_1$ satisfies that
	\begin{equation}
		 \left[ (1-\Delta)^2 + e_0^2 \right] \mathcal{Y}_1 = F(x),
		\label{equation for Y_1}
	\end{equation}
    where
    \begin{align*}
   F(x)
   & = \left( |\cdot|^{-(N-\gamma)} * Q^p \right)Q^{p-2}  (-\Delta +1 ) \mathcal{Y}_1 \notag\\
   &\quad + \left( \Delta -1 \right) \left[ p \left( |\cdot|^{-(N-\gamma)} * \left( Q^{p-1} \mathcal{Y}_1 \right) \right) Q^{p-1} + \left( p-1 \right) \left( |\cdot|^{-(N-\gamma)} * Q^p \right) Q^{p-2} \mathcal{Y}_1 \right]\notag \\
   & \quad + \left( |\cdot|^{-(N-\gamma)} * Q^p \right) Q^{p-2} \cdot  \left[ p \left( |\cdot|^{-(N-\gamma)} * \left( Q^{p-1} \mathcal{Y}_1 \right) \right)  Q^{p-1} + \left( p-1 \right) \left( |\cdot|^{-(N-\gamma)} * Q^p \right) Q^{p-2} \mathcal{Y}_1 \right]
    \end{align*}
	satisfies $|F(x)| \lesssim \left\langle x \right\rangle^{-\alpha}$ for some $\alpha>0$ by the exponential decay of $\partial^\beta Q$ and the boundness of $\partial^\beta \mathcal{Y}_1$ for any $\beta \in \mathbb{Z}_{\ge 0} ^N$.
	
	Next, we calculate the integral kernel $G_-(x,y)=(-\Delta +1 - ie_0)^{-1}(x,y)$. Note that
	\begin{align*}
		\Bigg| \int_{\mathbb{R}^N} e^{2 \pi i x  \cdot \xi} \left( 4 \pi^2 |\xi|^2 +1 - ie_0 \right)^{-1} d\xi \Bigg| 
		& = \Bigg|  \int_{\mathbb{R}^N} e^{2 \pi i x  \cdot \xi}  \left( \int_0^\infty e^{-\left( 4 \pi^2 |\xi|^2 +1 - ie_0 \right) \delta} d\delta \right) d\xi \Bigg|  \\
		& \le (4 \pi)^{-\frac{N}{2}}  \int_0^\infty e^{-\frac{|x|^2}{4 \delta}} e^{-\delta}  \delta^{-\frac{N}{2}} d\delta,
	\end{align*}
	together with the estimate that
	\[
	e^{-\frac{|x|^2}{4 \delta}} e^{-\delta} \le \min \left\lbrace e^{-\delta - \frac{1}{4 \delta}}, e^{-|x|} \right\rbrace \le e^{-\frac{1}{2}\delta - \frac{1}{8 \delta}} e^{-\frac{|x|}{2}}, \text{ for } |x| \ge 1,
	\]
	we have
	\[
	|G_-(x)| \lesssim e^{-\frac{|x|}{2}}  \int_0^\infty e^{-\frac{1}{2}\delta - \frac{1}{8 \delta}}  \delta^{-\frac{N}{2}} d \delta \lesssim e^{-\frac{|x|}{2}}, \quad |x| \ge 1.
	\]
	When  $|x| \to 0$, from \cite{aronszajn1961theory} and \cite{stein2016singular},
	\[
	G_-(x) \le	(4 \pi)^{-\frac{N}{2}}  \int_0^\infty e^{-\frac{|x|^2}{4 \delta}} \delta^{-\frac{N}{2}}e^{-\delta} d\delta=
	\begin{cases}
	C |x|^{-N+2} + o(|x|^{-N+2}), & \text{ if } N \ge 3, \\
	C \log \frac{1}{|x|} + o \left( \log \frac{1}{|x|} \right), & \text{ if } N=2, \\
	C + o(1), & \text{ if } N=1.
\end{cases}
	\]
	Similarly, we can get the same estimates of the integral kernel $G_+(x,y)= (-\Delta+1+ ie_0)^{-1}(x,y)$. It then suffices to check the decay of
	\[
	\mathcal{Y}_1 = \int_{\mathbb{R}^N}  G_+ (x-z) \left( \int_{\mathbb{R}^N} G_-(z-y) F(y) dy \right) dz.
	\]
	In fact, as for $\int_{\mathbb{R}^N} G_-(z-y) F(y) dy$, we divide the integral domain into the  following three parts:  $|y-z| \le \frac{1}{2}|z|$, $|y| \ge 2 |z|$ and $|y-z| > \frac{1}{2} |z|$. After some simple estimates respectively, it can be checked that
	\[
	\int_{\mathbb{R}^N} G_-(z-y) F(y) dy \lesssim \left\langle z \right\rangle^{-\alpha}.
	\]
	Repeating the argument again, we obtain $|\mathcal{Y}_1(x)| \lesssim \left\langle x \right\rangle^{-\alpha}$. Similarly, taking $\partial^\beta$ onto both sides of (\ref{equation for Y_1}), we have the same decay of $\partial^\beta \mathcal{Y}_1, \; \forall \beta \in \mathbb{Z}_{\ge 0}^N$. By bootstrapping argument again and again, we can check $ \big\| \left\langle x \right\rangle^{\gamma} \partial^\beta \mathcal{{Y}}_1 \big\|_{L^\infty} < \infty, \; \forall \beta \in \mathbb{Z}_{\ge 0}^N, \; \forall \gamma >0$, i.e.  $\mathcal{Y}_1 \in \mathcal{S}(\mathbb{R}^N)$. The proof of $\mathcal{Y}_2 \in \mathcal{S}(\mathbb{R}^N)$ is almost the same.
\end{proof}

\begin{rmk}
	By the analogous process of proving $\mathcal{Y}_\pm \in \mathcal{S}(\mathbb{R}^N)$,
	\begin{equation}
		\left( \mathcal{L}- (k+1) e_0 \right)^{-1} \mathcal{S}(\mathbb{R}^N) \subset \mathcal{S}(\mathbb{R}^N), \quad \forall k \ge 1.
		\label{resolvent operator maps Schwartz class into Schwartz class}
	\end{equation}
\end{rmk}

\subsection{Coercivity of linearized energy}
As for the linearized energy $\Phi$ defined in (\ref{linearized energy}), we denote the bilinear symmetric form associated to $\Phi$ by $B(f,g)$, i.e.
\begin{equation}
	B(g,h)=\frac{1}{2} \int (L_+ g_1) h_1 dx + \frac{1}{2} \int (L_- g_2) h_2 dx.
	\label{bilinear summetric form}
\end{equation}
Next, we want to obtain the coercivity of $\Phi$. Before we present the result, except for (\ref{orthogonal condition: orthogonal to Delta Q}), we give some other orthogonality relations:
\begin{equation}
	\int \left( \partial_{x_1} Q \right) h_1 = ... = \int \left( \partial_{x_N} Q \right) h_1 = \int Q h_2,
	\label{orthogonal condition: orthogonal to nabla Q and iQ}
\end{equation}

\begin{equation}
	B(\mathcal{Y}_-,h)= B (\mathcal{Y}_+, h)=0 \quad \Leftrightarrow \quad \int \mathcal{Y}_1 h_2 = \int \mathcal{Y}_2 h_1=0,
	\label{orthogonal condition: defined by bilinear form B}
\end{equation}
where $\mathcal{Y}_\pm$ are defined in Proposition \ref{proposition: spectral information of linearized operator L}. We define $G_\perp $ to be the set of all $h \in H^1(\mathbb{R}^N)$ satisfying (\ref{orthogonal condition: orthogonal to Delta Q}) and (\ref{orthogonal condition: orthogonal to nabla Q and iQ}), and $G_\perp'$ to be the set of all $h \in H^1(\mathbb{R}^N)$ satisfying (\ref{orthogonal condition: orthogonal to nabla Q and iQ}) and (\ref{orthogonal condition: defined by bilinear form B}), then
\begin{prop}[coercivity of $\Phi$]
	Under the Assumption \ref{assumption: nondegeneracy of the ground state}, there exists constant $c > 0$ such that
	\begin{equation}
		\Phi(h) \ge c\|h \|_{H^1}^2, \quad \forall h \in G_\perp \cup G_\perp'.
		\label{corecivity of Phi acting on space with codimension 6}
	\end{equation}
	\label{proposition: corecivity of Phi acting on apce with codimension 6}
	\begin{proof}
		\textit{Step 1.} Coercivity of $\Phi$ on $G_\perp$. Note that $\Phi(h)=\Phi_1(h_1) + \Phi_2(h_2)$, where
		\[
		\Phi_1(h_1)= \frac{1}{2} \int \left( L_+ h_1 \right)h_1 , \qquad \Phi_2(h_2)= \frac{1}{2} \int \left( L_- h_2 \right)h_2.
		\]
		W want to prove the coercivity of $\Phi_1$ and $\Phi_2$ respectively. As for the coercivity of $\Phi_1$, we first claim that
		\begin{equation}
			\Phi_1(h_1) \ge c \| h_1 \|_2^2, \quad \forall \; h_1 \text{ satisfying } \int \left( \partial_{x_j} Q \right) h_1 = \int (|\cdot|^{-(N-\gamma)} * Q^p) Q^{p-1}h_1  =0, \forall 1 \le j \le N.
			\label{corecivity of Phi on L2 space}
		\end{equation}
		If not, we then assume that there exists a sequence of real-valued functions $\{f_n\}_n \subset H^1(\mathbb{R}^N)$ such that
		\begin{equation*}
			\lim_{n \to \infty} \Phi_1(f_n)=0, \quad \| f_n \|_2 =1 \text{ and }  \int \left( \partial_{x_j} Q \right) f_n=\int (|\cdot|^{-(N-\gamma)} *Q^p) Q^{p-1} f_n=0, \quad \forall 1 \le j \le N.
		\end{equation*}
		Note that
		\begin{align}
			\Phi_1(f_n) 
			&= \frac{1}{2} \int |\nabla f_n|^2 + \frac{1}{2} \int |f_n|^2  - \frac{p}{2}\int \left( |\cdot|^{-(N-\gamma)}*(Q^{p-1}f_n) \right)(Q^{p-1}f_n) \notag\\
			&\qquad \qquad - \frac{p-1}{2} \int \left( |\cdot|^{-(N-\gamma)} *Q^{p}  \right) Q^{p-2} |f_n|^2 = o_n(1),
			\label{smallness of Phi_1}
		\end{align}
		by Hardy-Littlewood-Sobolev inequality and exponential decay  of $\partial^\alpha Q, \forall \alpha \in \mathbb{Z}_{\ge 0}^N$ (see Lemma \ref{lemma: exponential decay of ground state and its derivatives}) we have $\int |\nabla f_n|^2 \le C \|f_n \|_2^2 < \infty$, which means that $\{ f_n \}$ is bounded in $H^1$, hence  there exists a subsequence such that
		\[
		f_n \rightharpoonup f_* \text{ in } H^1(\mathbb{R}^N) \text{ for some } f_* \in H^1(\mathbb{R}^N).
		\]
		Moreover, by compactness argument,
		\begin{equation}
			\int \left( |\cdot|^{-(N-\gamma)}* Q^p \right) Q^{p-2} |f_n|^2 \to \int \left( |\cdot|^{-(N-\gamma)}* Q^{p} \right)Q^{p-2} |f_*|^2, \; n \to \infty
			\label{weak convergence implies strong convergence of nonlinear term 1}
		\end{equation}
		and 
		\begin{equation}
			\int \left( |\cdot|^{-(N-\gamma)} * (Q^{p-1}f_n) \right) (Q^{p-1} f_n)  \to \int \left( |\cdot|^{-(N-\gamma)} * (Q^{p-1}f_*) \right) (Q^{p-1} f_*),\; n \to \infty. 
			\label{weak convergence implies strong convergence of nonlinear term 2}
		\end{equation}
		Then by semi-lower continuity of weak convergence  \cite[Proposition $3.5\; (iii)$]{brezis2010functional},
		\begin{equation*}
			0 = \liminf_{n \to \infty} \Phi_1(f_n) 	 \ge \Phi(f_*) \ge 0,
		\end{equation*}
		which implies that $\Phi_1(f_*)=0$ with
		\[
		\int \left( \partial_{x_j} Q \right) f_* = \int (|\cdot|^{-(N-\gamma)} * Q^p) Q^{p-1} f_*=0, \quad \forall 1 \le j \le N.
		\]
	    Consequently, by (\ref{smallness of Phi_1}), (\ref{weak convergence implies strong convergence of nonlinear term 1}), (\ref{weak convergence implies strong convergence of nonlinear term 2}) and the assumption that $\|f_n\|_2 =1$,
		\[
		\frac{p}{2}\int \left( |\cdot|^{-(N-\gamma)}* (Q^{p-1} f_*) \right)(Q^{p-1}f_*) + \frac{p-1}{2}\int \left( |\cdot|^{-(N-\gamma)} * Q^p \right) Q^{p-2} |f_*|^2 \ge \frac{1}{2},
		\]
		then  $0 \not= f_* \in H^1(\mathbb{R}^N)$ is a minimizer of the following variational problem
		\[
		\inf_{ f \in \mathcal{K} } \Phi_1(f)=\inf_{ f \in \mathcal{K} } \frac{1}{2} \int \left( L_+ f \right)f dx,
		\]
		where 
		\[\mathcal{K} \triangleq \left\lbrace f \in H^1(\mathbb{R}^N):  0 \not= f \text{ is real}, \; \| f \|_2=1, \;  f \perp \left( |\cdot|^{-(N-\gamma)} *Q^{p} \right) Q^{p-1}, \; f\perp \partial_{x_j} Q , \; 1 \le j \le N \right\rbrace,
		\]
		so there exists $\{\lambda_j\}_{0 \le j \le N+1}$ such that
		\[
		L_+ f_*= \lambda_0 \left(  |\cdot|^{-(N-\gamma)} *Q^{p} \right) Q^{p-1} + \sum_{j=1}^{N} \lambda_j \partial_{x_j} Q + \lambda_{N+1} f_*.
		\]
		Taking inner product with $f_*$ and $\partial_{x_j} Q$ in $L^2(\mathbb{R}^N)$ on both sides respectively, we can easily get that $\lambda_j= 0, \; \forall 1 \le j \le N+1$ and then
		\[
		L_+ f_*= \lambda_0 \left( |\cdot|^{-(N-\gamma)} * Q^p \right) Q^{p-1}.
		\]
		Note that
		\[
		L_+ Q = -(2p-2)  \left( |\cdot|^{-(N-\gamma)} *Q^{p} \right) Q^{p-1},
		\]
	    by Assumption \ref{assumption: nondegeneracy of the ground state} again,
		\[
		f_*=- \frac{\lambda_0}{2(p-1)} Q + \sum_{j=1}^N \mu_j \left( \partial_{x_j} Q \right)
		\]
		for some $\{ \mu_j \}_{1 \le j \le N}$. By orthogonal conditions and the fact that $\partial_{x_j} Q \perp Q$ for any $1\le j \le N$, $\mu_j=0, \; \forall 1 \le j \le 
		N$, then $f_*=- \frac{\lambda_0}{2(p-1)} Q$ and 
		\begin{align*}
			\Phi_1(f_*)
			 = \frac{1}{2} \int (L_+ f_*) f_* dx 
			 = -\frac{\lambda_0^2}{4(p-1)} \int \left( |\cdot|^{-(N-\gamma)}* Q^{p} \right) Q^p dx=0,
		\end{align*}
		which implies that $\lambda_0=0$ and thus $f_*=0$, a contradiction! Then the coercivity of $\Phi_1$ in the sense of $H^1(\mathbb{R}^N)$ follows from a simple interpolation of (\ref{corecivity of Phi on L2 space}) and
		\[
		\Phi_1 (h_1) \ge \int \frac{1}{2}|\nabla h_1| dx - C \| h_1 \|_{L^2(\mathbb{R}^N)}^2. 
		\]
		
		By the same argument as above, it is easy to get the coercivity of $\Phi_2$, then we have proved the coercivity of $\Phi$ on $G_\perp$. \\
		\textit{Step 2.} Coercivity of $\Phi$ on $G_\perp'$. First, we prove that
		\begin{equation}
			\Phi(h) >0, \quad \forall h \in G_\perp' \setminus \{ 0 \}.
			\label{positivity of Phi on G'(perp)}
		\end{equation}
		If not, then there exists $ f \in G_\perp' \setminus \{ 0 \}$ such that
		\[
		0 \ge \Phi(f)=B(f,f)=\frac{1}{2} \int \left( L_+ f_1 \right) f_1 + \frac{1}{2} \int \left( L_- f_2 \right) f_2.
		\]
		We define a space
		\begin{equation*}
			E \triangleq \text{span}\left\lbrace \partial_{x_1} Q,...,\partial_{x_N} Q, iQ ,\mathcal{Y}_+, f\right\rbrace.
		\end{equation*}
		Then $\forall h \in E$, which we may assume $h= \sum_{j_=1}^N \lambda_j \partial_{x_j} Q + \lambda_{N+1} iQ + \lambda_{N+2} \mathcal{Y}_+ + \lambda_{N+3} f, \; \lambda_j \in \mathbb{R}, \; \forall 1 \le j \le N+3$,
		\begin{align*}
			\Phi(h) =B(h,h) = \lambda_{N+3}^2 B(f,f) = \lambda_{N+3}^2 \Phi(f) \le 0.
		\end{align*}
		Furthermore, we claim that $\dim_{\mathbb{R}} E =N+3$. If it is true, then it contradicts to the fact that $\Phi$ is positive on subspace $G_\perp$ with codimension $N+2$, hence it remains us to check the validity of the claim. Assume that there exist $\{ \lambda_j \}_{j=1}^{N+3} \subset \mathbb{R}^{N+3}$ such that
		\[
		\sum_{j_=1}^N \lambda_j \partial_{x_j} Q + \lambda_{N+1} iQ + \lambda_{N+2} \mathcal{Y}_+ + \lambda_{N+3} f=0,
		\]
		then $\lambda_{N+2} = 0$ simply follows from the facts that
		\begin{equation*}
			0 = B\left(\sum_{j_=1}^N \lambda_j \partial_{x_j} Q + \lambda_{N+1} iQ + \lambda_{N+2} \mathcal{Y}_+ + \lambda_{N+3} f, \mathcal{Y} _-\right) = \lambda _{N+2} B(\mathcal{Y}_+, \mathcal{Y}_-)=0.
		\end{equation*}
		Since $\partial_{x_j} Q, iQ, f$ are orthogonal to each other in $L^2$, it is easy to check that that $\lambda_j=0, \forall i \not =N+2$. Then $\dim_{\mathbb{R}} E=N+3$ and we have proved (\ref{positivity of Phi on G'(perp)}). \\
		\indent As for the coercivity of $\Phi$ on $G_{\perp}'$, we assume that there exists a sequence $\{h_n \} \subset G_\perp'$ such that
		\[
		\Phi(h_n) \to 0, \text{ and } \| h_n \|_2^2=1,
		\]
		then we can also extract a subsequence $\{ h_n \} $ such that $h_n \rightharpoonup h^*$ in $H^1$ sense. By the same argument as in Step $1$, $h^*$ also satisfies $h^* \in G_\perp'$, $h^* \not = 0$ and $\Phi(h^*)=0$, which contradicts to (\ref{positivity of Phi on G'(perp)}). Consequently,
		\begin{equation}
		\Phi(h) \ge c \| h \|_2^2, \; \forall h \in G_\perp'.
		\label{coercivity of Phi on G prime}
		\end{equation}
		Then the coercivity of $\Phi$ on $G_\perp'$ immediately follows from the interpolation between (\ref{coercivity of Phi on G prime}) and 
		\[
		\Phi(h) \ge \frac{1}{2} \| h \|_{H^1}^2 - C \| h \|_{L^2}^2, \; \forall h \in H^1.
		\]
	\end{proof}
\end{prop}

\section{Modulation}
\label{Modulation}
For $u$ solution to (\ref{Hartree equation}) with
\[
M[u]= M[Q], \quad E[u] = E[Q],
\]
by Proposition \ref{propostion: convergence of u to soliton}, if $\delta(t)$ is sufficiently small, then there exists $(\tilde{\sigma}, \tilde{X})$ such that $  \big\| e^{-i \tilde{\sigma}} u(\cdot + \tilde{X}) -Q \big\| _{H^1} \le \varepsilon (\delta(t))$.  And we further have
\begin{lem}
	Under the Assumption \ref{assumption: uniqueness of the ground state}, for any solution $u$ satisfying $M[u]=M[Q], \; E[u]=E[Q]$, there exist $\delta_0>0$ such that there exists $(\sigma, X) \in 
	\mathbb{R} \backslash 2\pi \mathbb{Z} \times \mathbb{R}^N$ such that $v= e^{-i \sigma} u(\cdot+ X)$ satisfies $\| v - Q \|_{H^1} \le \varepsilon$ and 
	\begin{equation}
		\Im \int Q v =0, \quad \Re \int (\partial_{x_k} Q) v =0, \; k =1,...,N,
		\label{orthogonal condition of v}
	\end{equation}
	where $\varepsilon$ is defined in Proposition \ref{propostion: convergence of u to soliton}. Moreover, the parameters defined above are unique and the mapping $u \mapsto (\sigma, X)$ is $C^1$. 
	\label{modulation: implicit function theorem}
\end{lem}
\begin{proof}
	Note by Proposition \ref{propostion: convergence of u to soliton}, there exists $\tilde{\sigma}, \tilde{X}$ such that
	\[
	\Big\| e^{- i \widetilde{\sigma}} u(\cdot+ \widetilde{X}) - Q \Big\|_{H^1(\mathbb{R}^N)} \le {\varepsilon}(\delta), \quad \text{ if } \delta(t) < \delta_0 \ll 1.
	\]
	Without loss of generality, we assume $u$ is close to $Q$ in $H^1$, and if not, we can replace $u$ by $\tilde{u}= e^{- i \widetilde{\sigma}} u(\cdot +\widetilde{X})$. Next, we define a mapping $J: \mathbb{R} \setminus 2 \pi \mathbb{Z} \times \mathbb{R}^N \times H^1 \to \mathbb{R}^{N+1
	}$ as
	\[
	J(\sigma, X,u) = \left(  \begin{matrix}
		J_1(\sigma, X ,u)   \\
		J_2(\sigma, X, u)
	\end{matrix} \right)= \left(  \begin{matrix}
		\Im  \int e^{-i \sigma} u(x+X) Q(x) dx   \\
		\Re \int e^{-i \sigma} u(x+X) \nabla Q(x) dx
	\end{matrix} \right),
	\]
	where $J(0,0,Q)=0$ and the corresponding Jacobian matrix at $(0,0, Q)$ is invertible. By the implicit function theorem, there exists $\varepsilon_0, \eta_0 >0$ such that
	\[
	\| u -Q \|_{H^1} < \varepsilon_0 \quad \Rightarrow \quad \exists ! \; (\sigma, X) , \; |\sigma| + |X| \le \eta_0 \text{ and } J(\sigma, X ,u)=0.
	\] 
	The uniqueness of $(\sigma,X)$ and the regularity of the mapping $u \mapsto (\sigma,X)$ follow from the implicit function theorem and the regularity of solution to (\ref{Hartree equation}) in $H^1$ sense.
\end{proof}

Let $u$ be the solution to (\ref{Hartree equation}) satisfying $M[u]=M[Q]$ and $E[u]=E[Q]$, and let $D_{\delta_0}$ be the open set of all times in the lifespan of $u$ such that $\delta(t) < \delta_0$. On $D_{\delta_0}$, there exist $C^1$ functions $\sigma(t), X(t)$ such that  $e^{-i \sigma(t)} u(\cdot+X(t))$ is close to $Q$ and the orthogonal property (\ref{orthogonal condition of v}) holds. Next, we want to get more information about the behavior of such modulation parameters when it is in $D_{\delta_0}$. Here we tend to work with the parameters $X(t)$ and $\theta(t) =\sigma(t)-t$. Precisely, we rewrite
\begin{equation}
	e^{-it-i\theta(t)} u(x+ X(t),t) = (1+ \alpha(t)) Q (x)+h(t,x),\; \forall  t\in D_{\delta_0},
	\label{modulation: decomposition of u}
\end{equation}
where $\alpha(t) \in \mathbb{R}$ is a continuous function such that $\Re \left\langle h,  \left( |\cdot|^{-(N-\gamma) }* Q^p \right) Q^{p-1} \right\rangle =0$. Then by Proposition \ref{proposition: corecivity of Phi acting on apce with codimension 6}, $h \in G_{\perp}$ and thus
\begin{equation}
	\Phi(h) \simeq \| h \|_{H^1(\mathbb{R}^N)}^2.
	\label{corecivity of Phi acting on h}
\end{equation}
Furthermore, the behavior of these modulation parameters is as follows:
\begin{lem} Under the Assumption \ref{assumption: uniqueness of the ground state} and Assumption \ref{assumption: nondegeneracy of the ground state}, let $u$ be the solution to (\ref{Hartree equation}) satisfying $M[u]=M[Q], \; E[u]= E[Q]$, then taking a smaller $\delta_0$ if necessary, the following estimates hold for $t\in D_{\delta_0}$:
	\begin{equation}
		|\alpha(t)| \simeq \Big| \int Q h_1(t) dx\Big| \simeq \| h(t) \|_{H^1} \simeq \delta(t).
		\label{modulation parameters: estimate}
	\end{equation}
	\label{lemma: the estimates of modulation parameters}
\end{lem}
\begin{proof}
	Let $\widetilde{\delta}(t) \triangleq |\alpha(t)| + \delta(t) + \| h(t) \|_{H^1}$. Considering the variation of $u \mapsto M[u]$ around the ground state $Q$,
	\begin{equation*}
		M[u]
		= M[Q+\left( \alpha Q + h \right)]
		= M[Q] + 2 \Re \left\langle Q , \alpha Q+ h \right\rangle + O\left( \alpha^2 + \| h \|_{H^1}^2 \right),
	\end{equation*}
then
	\begin{equation}
		|\alpha(t)|= \frac{1}{\| Q \|_2^2} \Big| \int Q h_1 dx \Big| + O(\widetilde{\delta}^2).
		\label{modulation parameter: result of variation of mass}
	\end{equation}
	Moreover, if we consider the variation of $u \mapsto \| \nabla u \|_2^2$ near the ground state $Q$, then
	\begin{equation*}
		\| \nabla u \|_2^2
		= \Big\| \nabla \left( Q + \left( \alpha Q +h \right) \right) \Big\|_2^2 
		= \| \nabla Q \|_2^2 + 2 \alpha \| \nabla Q \|_2^2 + 2 \int \nabla Q \cdot \nabla h_1 dx +  O\left(  \alpha^2 + \| h\|_{H^1}^2 \right).
	\end{equation*}
  Together with orthogonal condition $ \int \left( |\cdot|^{-(N-\gamma)} * Q^{p} \right) Q^{p-1} h_1 dx = 0$,  we obtain that
  \begin{align}
  	\delta(u(t)) & = \Big|  \| \nabla u(t) \|_2^2 - \| \nabla Q \|_2^2 \Big| \notag\\
  	&= \Big| 2 \alpha \| \nabla Q \|_2^2 +2 \int Q h_1 dx \Big|  + O(\widetilde{\delta}^2) \notag\\
  	& =2  \left(  \| \nabla Q \|_2^2 + \| Q\|_2^2  \right) |\alpha(t)|  + O(\widetilde{\delta}^2).
  				\label{modulation parameter: result od variation of potential}
  \end{align}
  
	As for the variation of energy $E[u]$ around the ground state $Q$, if we let $v= \alpha Q +h$, 
	\begin{align*}
		E[u] &=E[Q+v]\\
		&= \frac{1}{2} \int |\nabla Q + \nabla v|^2 dx - \frac{1}{2p} \int \left( |\cdot |^{-(N-\gamma)} * |Q+v|^p \right) |Q+v|^p dx \\
		& = E[Q]  - \left\langle Q,v_1 \right\rangle  \\
		& \quad + \frac{1}{2} \left\langle -\Delta v_1 -p \left( |\cdot|^{-(N-\gamma)}*Q^{p-1} v_1 \right) Q^{p-1} -(p-1) \left( |\cdot|^{-(N-\gamma)} * Q^p \right)Q^{p-2} v_1 , v_1 \right\rangle \\
		& \quad + \frac{1}{2} \left\langle -\Delta v_2 - \left( |\cdot|^{-(N-\gamma)} * Q^p \right) Q^{p-2}v_2, v_2 \right\rangle + h.o.t. \\
		& = E[Q]+ \Phi(v)+ h.o.t.,
		\end{align*}
	where the last equality follows from the assumption $M[Q+v]= M[Q]$. Moreover, by Lemma \ref{lemma: estimate on J(z) and K(z)}, the $h.o.t.$ satisfies
	\begin{equation*}
		h.o.t.= 
		\begin{cases}
			O\left(\| v \|_{H^1}^3 \right), & \text{ if } p=2 \text{ or }  p \ge 3, \\
			O\left(\| v \|_{H^1}^{p} \right), & \text{ if } p \in (2,3),
		\end{cases}
	\end{equation*}
     and without loss of generality, we denote $h.o.t. = O\left( \| v \|_{H^1}^s \right)= O \left( \widetilde{\delta}^s \right)$ for some $s >2$.  Then by the orthogonal condition $\int \left( |\cdot|^{-(N-\gamma)} * Q^{p} \right) Q^{p-1} h_1 dx = 0$, we obtain that $B(Q,h)=0$ and
	\begin{align*}
		O(\widetilde{\delta}^s)
		=  \Phi(\alpha Q + h) 
		= \alpha^2  \Phi(Q) + \Phi(h) + 2 \alpha B(Q,h) 
		= \alpha^2  \Phi(Q) + \Phi(h),
	\end{align*}
	so together with (\ref{corecivity of Phi acting on h}) and (\ref{Pohozhaev equality}),
	\begin{equation}
		\| h \|_{H^1} \simeq \Phi(h)^{\frac{1}{2}} \simeq |\alpha|+ O(\widetilde{\delta}^\frac{s}{2}).
		\label{modulation parameter: variation of energy}
	\end{equation}
	Combining (\ref{modulation parameter: result of variation of mass}), (\ref{modulation parameter: result od variation of potential}) and (\ref{modulation parameter: variation of energy}), $\widetilde{\delta}(t) \lesssim \delta(t) + O(\widetilde{\delta^\kappa}(t))$ for some $\kappa > 1$, then by continuity argument, $\widetilde{\delta}(t) \lesssim \delta(t), \; \forall t \in D_{\delta_0},  \delta_0 \ll 1$. Then using  (\ref{modulation parameter: result of variation of mass}), (\ref{modulation parameter: result od variation of potential}) and (\ref{modulation parameter: variation of energy}) again, we immediately get (\ref{modulation parameters: estimate}).
\end{proof}

When it comes to the evolution of derivatives of modulation parameters,
\begin{lem}[Bounds on the time-derivatives]
	Under the conditions in Lemma \ref{lemma: the estimates of modulation parameters}, 
	\begin{equation}
		|\dot{\alpha}|+ |\dot{X}| + |\dot{\theta}| =O(\delta), \; \forall t \in D_{\delta_0}.
		\label{modulation parameters: estimate of their derivatives}
	\end{equation}
	\label{modulation parameters: estimate on their time-derivatives}
\end{lem}
\begin{proof}
	Let $\delta^*(t)= \delta(t)+ |\dot{\alpha}(t)| + |\dot{X}(t)| + |\dot{\theta}(t)|$. By (\ref{Hartree equation}) and (\ref{modulation: decomposition of u}),
	\begin{equation}
		i \partial_t h + \Delta h + i \dot{\alpha} Q - i \dot{X} \cdot \nabla Q - \dot{\theta} Q = O(\delta+ \delta \delta^*) \text{ in } L^2.
		\label{the equation h has to satisfies}
	\end{equation}
	Since $h$ satisfies (\ref{orthogonal condition: orthogonal to Delta Q}) and (\ref{orthogonal condition: orthogonal to nabla Q and iQ}), multiplying (\ref{the equation h has to satisfies}) by $Q$, integrating on $\mathbb{R}^N$ and then taking the real part, we get
	\[
	- \dot{\theta} \| Q \|_2^2  =  O(\delta + \delta \delta^*) \; \Rightarrow \; |\dot{\theta}| = O(\delta + \delta \delta^*).
	\]
	If we multiply (\ref{the equation h has to satisfies}) by $\partial_{x_j} Q$ and $\left( |\cdot|^{-(N-\gamma)} * Q^p \right) Q^{p-1}$ respectively, integrate on $\mathbb{R}^N$ and take the imaginary part, then
	\[
	|\dot{X}_j| = O(\delta + \delta \delta^*) \text{ and }  |\dot{\alpha}|= O(\delta + \delta \delta ^*).
	\]
	Consequently, 
	\[
	\delta^* = O(\delta + \delta \delta^*), \quad \forall t \in 
	D_{\delta_0},
	\]
	then it concludes the desired result for $\delta_0 \ll 1$.
\end{proof}
We end this section with the following lemma, which will be used in the next two sections.
\begin{lem}
	Under the Assumption \ref{assumption: uniqueness of the ground state} and Assumption \ref{assumption: nondegeneracy of the ground state}, let $u$ be the solution to (\ref{Hartree equation}) satisfying $M[u]=M[Q], E[u]=E[Q]$, and assume $u$ is defined on $[0,+\infty)$ and that there exists $c,C >0$ such that
	\begin{equation}
		\int_t^\infty \delta(s) ds \le C e^{-ct}, \forall t \ge 0,
		\label{assumption on delta: decay in exponential sense}
	\end{equation}
	then there exists $(\theta_0, x_0 ) \in \mathbb{R}\backslash 2\pi \mathbb{Z} \times \mathbb{R}^N$ such that
	\[
	\| u - e^{i\theta_0} e^{it} Q(\cdot- x_0) \|_{H^1(\mathbb{R}^N)} \le C e^{-ct}.
	\]
	\label{auxilliary lemma}
\end{lem}
\begin{proof}
	By assumption (\ref{assumption on delta: decay in exponential sense}), it can be easily to see that there must exist $\{t_n\}_{n \in \mathbb{N}}$, such that $\delta(t_n) \to 0$. With the help of convergence to $0$ of $\delta(t)$ up to a subsequence, we first need to check $\lim_{t \to \infty} \delta(t)=0$. If not, there exists $\{ t_n' \}_{n \in \mathbb{N}}$, such that
	\[
	t_n' \to\infty, \text{ and } \delta(t_n') \ge \varepsilon_1 >0
	\]
	for some $\varepsilon_1>0$. Then we can adjust the value of $\{ t_n' \}$ and extract subsequences from $\{t_n\}$ and $\{t_n'\}$ with the properties below:
	\[
	t_n < t_n',  \; \forall n \in \mathbb{N}, \quad \delta(t_n)= \varepsilon_1, \quad \delta(t) < \varepsilon_1, \; \forall t \in [t_n, t_n').
	\]
	Hence
	\[
	|\alpha(t_n)- \alpha(t_n')| \le \int_{t_n}^{t_n'} |\dot{\alpha}(t)| dt \le Ce^{-ct_n} \to 0, \quad \Rightarrow \quad \alpha(t_n') \to 0, n \to \infty,
	\]
	which leads a contradiction.
	
	With the help of (\ref{modulation parameters: estimate}) and (\ref{modulation parameters: estimate of their derivatives}), we claim that there exists $X_\infty$ and $\theta_\infty$ such that
	\begin{equation}
		\delta(t) + |\alpha(t)| + \|h(t) \|_{H^1} + |X(t)-X_\infty| + |\theta(t)- \theta_\infty| \le  C e^{-ct}, \forall \; t \ge 0.
		\label{existence of parameters X(infty) and theta(infty)}
	\end{equation}

 In fact, by Lemma \ref{modulation parameters: estimate on their time-derivatives} and Lemma \ref{auxilliary lemma} and together with (\ref{assumption on delta: decay in exponential sense}),
	\[
	\delta(t) \simeq \| h(t)\|_{H^1} \simeq | \alpha(t) |= \Big| \int_t^\infty  \dot{\alpha}(s) ds \Big| \lesssim \int_t^\infty \delta(s) ds \le C e^{-ct}, \; \forall t \ge 0,
	\]
then as a by-product,
	\[
	|\dot{X}(t)| + |\dot{\theta}(t)| \le Ce^{-ct},
	\]
	which yields that there exist $X_\infty$ and $\theta_\infty$ such that
	\begin{equation*}
		|X(t)- X_\infty| + |\theta(t)- \theta_\infty|  \lesssim  \int_t^\infty |\dot X(s)| + |\dot{\theta}(s)| ds \lesssim \int_t^\infty e^{-cs} ds \lesssim Ce^{-ct},
	\end{equation*}
	and it solves (\ref{existence of parameters X(infty) and theta(infty)}). Lemma \ref{auxilliary lemma} then immediately follows from (\ref{existence of parameters X(infty) and theta(infty)}) and the decomposition (\ref{modulation: decomposition of u}).
\end{proof}

\section{Convergence to $Q$ for $\| \nabla u_0 \|_2^\frac{1-s_c}{s_c} \| u_0 \|_2 > \| \nabla Q \|_2^\frac{1-s_c}{s_c} \| Q \|_2$}
\label{Convergence to Q above the threshold}
\begin{thm}
	Under the same conditions as what in Theorem \ref{theorem: classification of threshold solutions at zero momentum case}, let $u$ be the solution to (\ref{Hartree equation}) satisfying
	\begin{equation}
		M[u]= M[Q], \quad E[u]= E[Q], \quad  \| \nabla u_0 \|_2 > \| \nabla Q \|_2,
		\label{condition: above the threshold}
	\end{equation}
	and $u$ globally exist in positive time, and assume 
	\[
	\begin{cases}
		u_0 \text{ is of finite variance, i.e. } |x| u_0 \in L^2, \\
		\text{or } u_0 \in L_{rad}^2(\mathbb{R}^N) \text{ and } (N,p,\gamma) \text{ satisfies } (\ref{parameter: range for radial case}) \text{ in addition},
	\end{cases}\]
then there exists $\theta_0 \in \mathbb{R}\backslash 2 \pi \mathbb{Z}, x_0 \in \mathbb{R}^N, c,C>0$ such that 
	\[
	\| u - e^{it + i \theta_0} Q(\cdot -x_0) \|_{H^1(\mathbb{R}^N)} \le C e^{-ct}.
	\]
	Moreover, the lifespan of $u$ in negative time direction is finite.
	\label{theorem: convergence of u to soliton Q above the threshold}
\end{thm}

\subsection{Finite variance solutions}
\label{Finite variance solutions subsections}
In this section, we are devoted to the proof of the Theorem \ref{theorem: convergence of u to soliton Q above the threshold} when $u_0$ is of finite variance. Here we introduce virial quantity 
\begin{equation}
	y(t) \triangleq \int |x|^2 |u(x,t)|^2 dx,
	\label{virial quantity}
\end{equation}
then 
\begin{equation}
	\dot{y}
	=\Re \int_{\mathbb{R}^N} |x|^2 \dot{u} \bar{u} dx \notag  \\
	= -2 \Im \int_{\mathbb{R}^N} |x|^2 (\Delta u) \bar{u} dx \notag\\
	=4 \Im \int_{\mathbb{R}^N} x \cdot \nabla u \bar{u} dx,
	\label{first derivative of varial}
\end{equation}
and
\begin{equation}
	\ddot{y} (t) = 16(s_c(p-1)+1)E[u]- 8s_c(p-1) \| \nabla u \|_2^2.
	\label{second derivative of virial: finite variance 1}
\end{equation}
By the assumption that
\[
E[u]=E[u_0]=E[Q]
\]
and Pohozhaev identity (\ref{Pohozhaev equality}), $\ddot{y}(t)$ can be written as
\begin{align}
	\ddot{y}(t)
	&=16(s_c(p-1)+1)E[Q]- 8s_c(p-1) \| \nabla u \|_2^2 \notag\\
	&= 8s_c(p-1) \left( \| \nabla  Q \|_2^2 - \| \nabla u \|_2^2 \right) = -8 s_c(p-1) \delta(t) <0,
	\label{second derivative of virial: finite variance 2}
\end{align}
where $\delta(t)= \Big| \| \nabla u \|_2^2 - \| \nabla Q \|_2^2 \Big|$. 

Then we devide the argument into the following three steps with the previous preliminary.\\
\textit{Step 1.} Claim: 
\begin{equation}
	\dot{y}(t) >0, \; \forall t\ge 0.
	\label{derivative of virial quantity: above the threshold}
\end{equation}
Here we use the convexity argument. If not, then $\exists t_0 \ge 0$ such that $\dot{y}(t_0) \le 0$. Then by (\ref{second derivative of virial: finite variance 2}), there exists $t_1 > t_0$, such that $\forall t \ge t_1$, $\dot{y}(t) \le \dot{y}(t_1) < \dot{y}(t_0) \le 0$, hence $y(t) \to -\infty$ as $t \to \infty$, a contradiction! \\
\textit{Step 2.} Claim: If $\varphi \in C^1(\mathbb{R}^N), \; f \in H^1(\mathbb{R}^N)$ satisfying $\|f \|_2= \| Q \|_2 \text{ and } E[f]=E[Q]$, moreover, assume that $\int |\nabla \varphi |^2 |f|^2 dx < \infty$, then
\begin{equation}
	\Big| \Im \int \left( \nabla \varphi \cdot \nabla f \right) \bar{ f}\Big|^2 \le  C \delta^2(f) \int |\nabla \varphi|^2 |f|^2.
	\label{Cauchy-Schwarz type inquality}
\end{equation}
Since $Q$ is the minimizer of minimizing problem
\[
\inf_{0 \not= f \in H^1(\mathbb{R}^N)} \frac{ \| \nabla f \|_2^{Np-(N+\gamma)} \| f \|_2^{N+\gamma-(N-2) p }} { \int \left( |\cdot|^{-(N-\gamma)} * |f|^p \right) |f|^p dx},
\]
we have
\begin{align*}
	\left( \int \left( |\cdot|^{-(N-\gamma)} * |f|^p \right) |f|^p dx \right) \| \nabla Q \|_2^{Np-(N+\gamma)} \le \left( \int \left( |\cdot|^{-(N-\gamma)} * Q^p \right) Q^p dx \right) \| \nabla f \|_2^{Np-(N+\gamma)}
\end{align*}
for any $f \in  H^1(\mathbb{R}^N)$. The inequality above is also vaild for $e^{i\lambda\varphi} f, \; \forall \lambda \in \mathbb{R}$, hence
\begin{align*}
	& \quad \left( \int |\nabla \varphi|^2 |f|^2 dx \right) \lambda^2 - 2 \left( \Im \int \left( \nabla \phi \cdot \nabla f \right) \bar{f}\right) \lambda   \\
	& \qquad \qquad + \| \nabla f \|_2^2 - \frac{\| \nabla Q \|_2^2 \left( \int \left( |\cdot|^{-(N-\gamma)} *|f|^p \right) |f|^p dx \right)^\frac{2}{Np-(N+\gamma)} }{\left( \int \left( |\cdot|^{-(N-\gamma)} * Q^pdx \right)Q^p \right)^\frac{2}{Np-(N+\gamma)}}  \ge 0, \quad \forall \lambda \in \mathbb{R},
\end{align*}
which means that
\begin{equation}
	 \left( \Im \int \left( \nabla \phi \cdot \nabla f \right) \bar{f}\right)^2 
	\le \left( \int |\nabla \varphi|^2 |f|^2 dx \right) \left( \| \nabla f \|_2^2 - \frac{\| \nabla Q \|_2^2 \left( \int \left( |\cdot|^{-(N-\gamma)} *|f|^p \right) |f|^p dx \right)^\frac{2}{Np-(N+\gamma)} }{\left( \int \left( |\cdot|^{-(N-\gamma)} * Q^pdx \right)Q^p \right)^\frac{2}{Np-(N+\gamma)}} \right).
	\label{Delta < 0}
\end{equation}
By the assumption that $E[f]=E[Q]$ and Pohozhaev identity (\ref{Pohozhaev equality}),
\begin{align}
	&\quad\| \nabla f \|_2^2 - \frac{\| \nabla Q \|_2^2 \left( \int \left( |\cdot|^{-(N-\gamma)} *|f|^p \right) |f|^p dx \right)^\frac{2}{Np-(N+\gamma)} }{\left( \int \left( |\cdot|^{-(N-\gamma)} * Q^pdx \right)Q^p \right)^\frac{2}{Np-(N+\gamma)}} \notag \\
	&= \| \nabla f \|_2^2 - \| \nabla Q \|_2^2 \left(1+ \frac{p \delta(t)}{\int (|\cdot|^{-(N+\gamma)}* Q^p)Q^p dx}  \right)^\frac{2}{Np-(N+\gamma)} \notag \\
	& \le \| \nabla f \|_2^2- \|\nabla Q \|_2^2 -\frac{2p}{Np-(N+\gamma)} \frac{\| \nabla Q \|_2^2}{ \int (|\cdot|^{-(N-\gamma)} *Q^p) Q^p dx} \delta(t) +C\delta(t)^2 \notag\\
	& = \left( 1- \frac{2p}{Np-(N+\gamma)}  \frac{\| \nabla Q \|_2^2}{ \int (|\cdot|^{-3} *Q^2) Q^2 dx} \right) \delta(t) + C\delta(t)^2 =C\delta(t)^2,
	\label{Cauchy-Schwartz type inquality: auxillary result}
\end{align}
the claim then immediately follows from (\ref{Delta < 0}) and (\ref{Cauchy-Schwartz type inquality: auxillary result}). \\
\textit{Step 3.} Complete the proof. Let $\varphi(x)= |x|^2$, then by (\ref{first derivative of varial}), (\ref{second derivative of virial: finite variance 2}) and (\ref{Cauchy-Schwarz type inquality}),
\[
(\dot{y})^2 \le  C (\ddot{y})^2 y \quad \Rightarrow \quad \frac{\dot{y}}{\sqrt{y}} \le -C \ddot{y}.
\]
By (\ref{derivative of virial quantity: above the threshold}),
\[
\sqrt{y(t)}- \sqrt{y(0)} = \int_0^t \frac{\dot{y}(s)}{\sqrt{y(s)}} ds \le -C \left( \dot{y}(t) - \dot{y}(0) \right) \le C \dot{y}(0)
\]
and $y(t)$ is uniformly bounded for $t \ge 0$. Therefore,
\[
\dot{y} \le -c \ddot{y}, \quad \Rightarrow \quad \frac{d}{dt} \left( e^{ct} \dot{y} \right)  \le 0, \quad \Rightarrow \quad \dot{y} (t) \le C e^{-ct}, \; \forall t \ge 0,
\]
for some $C,c >0$. Then by (\ref{second derivative of virial: finite variance 2}),
\[
\int_t^\infty \delta(s) ds = -\frac{1}{8 s_c (p-1)} \int_t^\infty \ddot{y}(s) ds \lesssim  e^{-ct}, \; \forall  t\ge 0.
\]
Then we immediately complete this part by Lemma \ref{auxilliary lemma}.

\subsection{Radial solutions}
\label{subsection: radial solutions}
In this subsection, we consider the case for $u_0 \in L_{\text{rad}}^2(\mathbb{R}^N)$, and our goal is to prove $u_0$ is of finite variance. Here we consider the localized virial quantity
\begin{equation}
	y_R(t)= \int_{\mathbb{R}^N} \varphi_R(x) |u(x,t)|^2 dx = \int_{\mathbb{R}^N} R^2\varphi\left( \frac{x}{R} \right) |u(x,t)|^2 dx,
	\label{localized virial: radial case}
\end{equation}
where  $\varphi$ is a $C^\infty$ positive radial function on $\mathbb{R}^N$ such that
\[
\varphi(x) =
\begin{cases}
	|x|^2, & \text{if } |x| \le 1, \\
	0, & \text{if } |x| \ge 2,
\end{cases} \quad \varphi (r) \ge 0 , \quad \varphi''(r) \le 2,
\]
then after some careful calculation, we obtain that
\begin{equation}
	\dot{y}_R(t)= 2 \Im \int_{\mathbb{R}^N} \nabla \varphi_R(x) \cdot  \nabla u \bar{u} dx = 2R \Im \int_{\mathbb{R}^N} \nabla \varphi \left( \frac{x}{R} \right) \cdot  \nabla u \bar{u} dx,
	\label{first derivative of virial: radial case}
\end{equation}
and
\begin{equation}
	\ddot{y}_R(t)=- 8 s_c(p-1) \delta(t) +A_R(u(t)),
	\label{second derivative of virial: radial}
\end{equation}
where
\begin{align}
	A_R(u(t))& = 4 \int_{\mathbb{R}^N} \left( \varphi''\left( \frac{x}{R} \right) -2  \right) |\nabla u|^2 dx - \int_{\mathbb{R}^N}  \left( \Delta^2 \varphi_R \right)  |u|^2 dx \notag\\
	& \quad+ \left( \frac{2}{p} -2 \right)\int_{\mathbb{R}^N} \left(\Delta \varphi_R  -2 N \right) \left( |\cdot|^{-(N-\gamma)} * |u|^p \right) |u|^p dx \notag \\
	& \quad - \frac{2}{p}\iint_{\mathbb{R}^N \times \mathbb{R}^N} |u(x)|^p |u(y)|^p \nabla_y \cdot  \left( \frac{(\nabla \varphi_R (x) -2x)-(\nabla \varphi_R(y)- 2y)}{|x-y|^{N-\gamma}} \right) dxdy.
	\label{remaining term: radial case}
\end{align}
Next, we we claim that there exists $R_0 >0$, such that
\begin{equation}
\ddot{y}_R(t) \le -4s_c(p-1) \delta(t), \quad \forall R \ge R_0,
\label{estimate of second derivative of virial: radial, above the threshold}
\end{equation}
and it suffices to check
\[
|A_R\left( u(t) \right)| \le 4s_c(p-1) \delta(t), \quad \forall R \ge R_0
\]
for some $R_0 >0$. Here we divide the proof into following two cases: \\
\textit{Case 1.} $t \in D_{\delta_1}$ for some $0 < \delta_1 \le \delta_0 \ll 1$. Here we try to make full use of modulation. By (\ref{modulation: decomposition of u}) and the symmetry of $u$,
\[
e^{-i \theta} u = (1+ \alpha)e^{it} Q + e^{it} h = e^{it} Q + e^{it} \left( \alpha Q +h \right).
\]
We denote $v=\alpha Q +v$, then
\begin{equation*}
	A_R(u)
	= A_R\left( e^{- i \theta } u \right) 
	= A_R(e^{it} Q + e^{it} v) 
	= A_R(Q+v) -A_R(Q),
\end{equation*}
where the last equality follows from the fact that
\[
\dot{y}_R(t) = 2 R \Im \int e^{-it} Q \nabla \varphi \left( \frac{x}{R} \right) \cdot e^{it} \nabla Q dx=0 \text{ and } \delta(u(t))=0 \text{ if } u(t)=e^{it}Q.
\]
Furthermore, note that when $|x| \le R$, $\nabla \varphi_R(x)= R \nabla \varphi\left( \frac{x}{R} \right) = 2x$, the function $\left( \nabla \varphi_R(x) - 2x \right)- \left( \nabla \varphi_R(y) -2y \right)$ is supported on
\begin{equation}
\left\lbrace (x,y) \in \mathbb{R}^N \times \mathbb{R}^N: |x| \ge R \text{ or } |y| \ge R\right\rbrace.
\label{support of nonlinear term: radial case}
\end{equation}
In addition, by mean-value theorem,
\begin{align}
	\Bigg| \nabla_y \cdot  \left( \frac{(2x-\nabla \varphi_R (x))-(2y- \nabla \varphi_R(y))}{|x-y|^{N-\gamma}}  \right) \Bigg| 
	& \lesssim \Bigg| \frac{-2N + \Delta \varphi_R (y)}{|x-y|^{N-\gamma}} \Bigg| 
	 + \frac{ |x-y| + |\nabla \varphi_R(x)-\nabla  \varphi_R (y)|}{|x-y|^{N-\gamma +1}} \notag\\
	& \lesssim \frac{1}{|x-y|^{N-\gamma}}.
	\label{estimate of gradident onto convolution}
\end{align}
Then by Lemma \ref{lemma: the estimates of modulation parameters}, Lemma \ref{lemma: exponential decay of ground state and its derivatives} and Lemma \ref{lemma: estimate on J(z) and K(z)},
\begin{align}
	A_R(u(t))
	&= |A_R(Q+v) - A_R(Q)| \notag \\
	&\lesssim  \Bigg| \int_{|x| \ge R} |\nabla Q + \nabla v|^2 -|\nabla Q|^2 dx \Bigg| + \Bigg| \int_{|x| \ge R} |Q+v|^2 -Q^2 dx \Bigg| \notag\\
	& \quad + \Bigg| \iint_{\{ |x| \ge R \} \times \mathbb{R}^N}  \left( \frac{|Q(x)+v(x)|^p |Q(y)+v(y)|^p}{|x-y|^{N-\gamma}} - \frac{|Q(x)|^p |Q(y)|^p}{|x-y|^{N-\gamma}} \right) dxdy\Bigg| \notag\\
	& \lesssim e^{-cR} \delta(t) + \delta^2(t)+ \delta^{2p}(t),
	\label{radial case: estimate on AR(u(t)): near the ground state}
\end{align}
for some $c >0$, which shows that there exists $R_1 \gg 1$ and $0 <\delta_1 \ll 1$ such that $|A_R(u(t))| \le 4s_c(p-1) \delta(t), \; \forall R \ge R_1, \text{ and } \forall t \in D_{\delta_1}$.

\noindent \textit{Case 2.} $t\notin D_{\delta_1}$. Since we have required that $\varphi''(r) \le 2$,
\begin{equation}
\int \left( \varphi''\left(\frac{x}{R}\right) -2 \right) |\nabla u|^2 dx \le 0.
\label{radial case: estimate of the first term}
\end{equation}
Moreover, 
\begin{align}
	\Bigg| \int_{\mathbb{R}^5} \left( \Delta^2 \varphi_R \right) |u|^2 dx\Bigg|
	&=\Bigg| \frac{1}{R^2} \int \left( \Delta^2 \varphi \right) \left( \frac{x}{R} \right) u(t,x)^2 dx \Bigg| \notag\\
	& \le \frac{C}{R^2} \| u \|_2^2 = \frac{C}{R^2} \| Q \|_2^2 \le 2s_c(p-1)\delta_1 \le 2s_c(p-1)\delta(t)
	\label{radial case: estimate of the second term}
\end{align}
for $R \ge R_2 = \sqrt{\frac{C\| Q \|_2^2}{2s_c (p-1)\delta_1}}$. 

It remains us to estimate the last two terms of $A_R(u(t))$ defined in (\ref{remaining term: radial case}). By (\ref{support of nonlinear term: radial case}) and (\ref{estimate of gradident onto convolution}), they can both be bounded by $\big\| \left( |\cdot|^{-(N-\gamma)} * |u|^p \right) |u|^p \big\|_{L_x^1(|x| \ge R)}$. Our aim is to prove
\begin{equation}
\Big\| \left( |\cdot|^{-(N-\gamma)} * |u|^p \right) |u|^p \Big\|_{L_x^1(|x| \ge R)} \lesssim R^{-\beta} \| u \|_2^{2p-\alpha} \| \nabla u \|_2^\alpha
\label{the estimate of the remaining term: radial case}
\end{equation}
for some $\beta >0$ and $\alpha \in [0,2]$. If we check the validity of  (\ref{the estimate of the remaining term: radial case}), then by mass conservation of $u$ and Young's inequality,
\begin{equation}
\Big\| \left( |\cdot|^{-(N-\gamma)} * |u|^p \right) |u|^p \Big\|_{L_x^1(|x| \ge R)} \lesssim R^{-\beta} \delta(t).
\label{the estimate of the remaining term: radial case 1}
\end{equation}
\begin{rmk}
	If we have (\ref{the estimate of the remaining term: radial case 1}), not only can we get a decay of coefficient with respect to $R$, but also we get that the behavior of nonlinear term can be control by $\delta(t)$ when $t \notin D_{\delta_0}$. However, if we wish $\delta(t)$ to dominate the behavior of nonlinear term, we are motivated to restrict $\alpha \in [0,2]$ in (\ref{the estimate of the remaining term: radial case}),  which, by further calculation, is why we should add the additional restrictions onto the parameters $(\gamma, N, p)$ as shown in (\ref{parameter: range for radial case}).
\end{rmk}

We then continue the argument. To prove (\ref{the estimate of the remaining term: radial case}), we divide the integral above into the following two parts:
\begin{align*}
	\Big\| \left( |\cdot|^{-(N-\gamma)} * |u|^p \right) |u|^p \Big\|_{L_x^1(|x| \ge R)}
	& \le 	\Big\| \left( |\cdot|^{-(N-\gamma)} *\left(  \chi_{\left\lbrace |x| \ge \frac{1}{2}R \right\rbrace} |u|^p \right) \right) |u|^p \Big\|_{L_x^1(|x| \ge R)} \\
	&+ \Big\| \left( |\cdot|^{-(N-\gamma)} *\left(  \chi_{\left\lbrace |x| \le \frac{1}{2}R \right\rbrace} |u|^p \right) \right) |u|^p \Big\|_{L_x^1(|x| \ge R)} \\
	& \triangleq I+ II.
\end{align*}
As for $I$, by H\"{o}lder inequality and Hardy-Littlewood-Sobolev inequality,
\begin{align*}
	I & \lesssim 	\Big\|  |\cdot|^{-(N-\gamma)} *\left(  \chi_{\left\lbrace |x| \ge \frac{1}{2}R \right\rbrace} |u|^p \right)\Big\|_{L_x^\frac{2N}{N-\gamma}(|x| \ge R)}  \| u \|_{L_x^\frac{2Np}{N+\gamma} \left( |x| \ge R \right)} \notag \\
	& \lesssim \| u \|_{L_x^{\frac{2Np}{N+\gamma}} \left( |x| \ge \frac{1}{2}R \right)}^{2p} \notag \\
	& \lesssim \frac{1}{R^{\frac{(N-1)\left( N(p-1) -\gamma \right)}{N}}}  \| u \|_2^\frac{Np+N+\gamma}{N} \|\nabla u \|_2^\frac{Np-N-\gamma}{N}.
%	\label{estimate for the first term I: radial case}
\end{align*}
Note we cannot get any decay of $R$ when $N=1$, we are motivated to restrict the parameters to
\begin{equation}
	N \ge 2, \; \frac{Np-N-\gamma}{N} \in [0,2].
		\label{estimate for the first term I: radial case}
\end{equation}
As for $II$, noting that $|x-y| \sim |x|$ when $(x,y) \in \{ (x,y): |x| \ge R, |y| \le \frac{R}{2} \}$,
\begin{equation*}
	II \simeq \iint_{\left\lbrace |x| \ge R \right\rbrace \cap \left\lbrace |y| \le \frac{1}{2}R \right\rbrace} \frac{|u(x)|^p |u(y)|^p}{|x|^{N-\gamma}} dxdy  
	 = \left( \int_{|y| \le \frac{1}{2} R} |u(y)|^p dy \right) \left(  \int_{|x| \ge R} \frac{|u(x)|^p}{|x|^{N-\gamma}} dx \right).
\end{equation*}
If $s \in \Big(\frac{N}{N-\gamma}, +\infty \Big]$, then
\begin{equation*}
	II \lesssim  \big\| |x|^{-(N-\gamma)} \big\|_{L^{s}(|x| \ge R)} \big\| |u|^p \big\|_{L^{s'}(|x| \ge R)} \| u \|_{L^p\left(|x| \le \frac{R}{2}\right)}^p
	\lesssim  R^{-(N-\gamma) +\frac{N}{s}} \| u \|_{L^{s'p} (|x| \ge R)}^p \| u \|_{L^p\left(|x| \le \frac{R}{2}\right)}^p,
\end{equation*}
where
\begin{equation}
\| u \|_{L^p\left(|x| \le \frac{R}{2}\right)}^p \le || u \|_2^\frac{2N-Np+2p}{2} \| \nabla u \|_2^\frac{Np-2N}{2}
\label{estimates of the remaining term: radial case, near the origin}
\end{equation}
by H\"{o}lder inequality and Sobolev embedding. Moreover, as for $ \| u \|_{L^{s'p} (|x| \ge R)}^p$, since the integral regime is away from the origin, except for following the argument in (\ref{estimates of the remaining term: radial case, near the origin}), there is another method to handle it, which is by using interpolation between $L^2$ norm and $L^\infty$ norm together with Strauss lemma. Compared with the former one, not only can we get a decay of $R$, but also can get less power of $\| \nabla u \|_2$  followed by the latter one. Thus we follow the latter way, then
\begin{align*}
	\big\| u \big\|_{L^{s' p}(|x| \ge R)}^p
	&= \left( \int_{|x| \ge R} |u(x)|^{s' p} dx \right)^{\frac{1}{s'}} \\
	& \le  \left( \int_{|x| \ge R} |u(x)|^{2} dx \right)^{\frac{1}{s'}}  \| u \|_{L^\infty (|x| \ge R)}^{\frac{1}{s'} \left( s'p -2 \right)} \qquad \text{(by interpolation)} \\
	& \lesssim R^{-\frac{N-1}{2} \left( p - \frac{2}{s'} \right)} \| u ||_{L^2}^{\frac{p}{2} + \frac{1}{s'}} \| \nabla u \|_{L^2}^{\frac{p}{2}-\frac{1}{s'}}. \qquad \text{(by Strauss lemma)}
\end{align*}
Consequently, $II$ is bounded by
\[
II \lesssim R^{-\frac{N-1}{2} \left( p - \frac{2}{s'} \right) -(N-\gamma) +\frac{N}{s}} \| u ||_{L^2}^{\frac{3p}{2} + \frac{1}{s'}- \frac{Np}{2}+N} \| \nabla u \|_{L^2}^{\frac{p}{2}-\frac{1}{s'} +\frac{Np}{2}-N},
\]
where the power of $\| \nabla u \|_2$ satisfies
\begin{equation}
\frac{p}{2}- \frac{1}{s'} +\frac{Np}{2}- N = \frac{1}{s} + \left( \frac{p}{2} -1 \right) \left( N+1 \right)  \in \Bigg[ \left( \frac{p}{2} -1 \right) \left( N+1 \right), \frac{N-\gamma}{N} + \left( \frac{p}{2} -1 \right) \left( N+1 \right) \Bigg),
\label{the estimate of II: radial case}
\end{equation}
and the minimum is obtained at $s=\infty$. In order to get (\ref{the estimate of the remaining term: radial case}), by (\ref{estimate for the first term I: radial case}) and (\ref{the estimate of II: radial case}), we then impose
\begin{equation}
	N \ge 2, \; \frac{Np-N-\gamma}{N} \in [0,2] \text{ and } \frac{p}{2}- \frac{1}{s'} +\frac{Np}{2}- N \in [0,2].
	\label{radial case: restriction on to the parameter original version}
\end{equation}
	Since we hope to get range of $(\gamma,N,p)$ as wide as possible, we choose $s=\infty$, then at this time (\ref{radial case: restriction on to the parameter original version}) is definitely (\ref{parameter: range for radial case}).

Under the restriction (\ref{parameter: range for radial case}), combining with (\ref{radial case: estimate of the first term}), (\ref{radial case: estimate of the second term}) and (\ref{the estimate of the remaining term: radial case 1}), we can choose $R \ge R_0 \gg 1$ such that
\[
|A_R(u(t)) | \le  4 s_c (p-1) \delta(t), \quad \forall t \notin D_{\delta_1},
\]
which implies (\ref{estimate of second derivative of virial: radial, above the threshold}). As a simple by-product, $\dot{y}_R (t) >0$ and $\dot{y}_R(t)$ is monotonically decreasing, which respectively imply that
\begin{equation}
y_R(0)= \int R^2 \varphi \left( \frac{x}{R} \right) |u_0|^2 dx \le y_R(t), \forall t \ge 0, \; \forall R \ge R_0,
\label{monotonicity of localized virial quantity}
\end{equation}
and there exists $A \ge 0$ such that $\dot{y}_R(t) \downarrow A \text{ as } t \to \infty$. The latter fact ensures
\begin{equation*}
\int_0^\infty \delta(t) dt \lesssim \int_0^\infty |\ddot{y}_R (t)| dt = \Big| \int_0^\infty \ddot{y}_R(t)dt \Big|= \Big| \dot{y}_R(\infty) - \dot{y}_R(0) \Big| < \infty,
\end{equation*}
so there exists a subsequence $\{t_n\}$ such that $t_n \to \infty$ and  $\delta(t_n) \to 0$. Since $M[u]=M[Q], E[u]=E[Q]$, by Proposition \ref{propostion: convergence of u to soliton},  there exists $\theta_0$ such that
\[
\big\| u(t_n) - e^{i \theta_0} Q \big\|_{H^1(\mathbb{R}^N)} \to 0, \quad n \to \infty.
\] 
Taking $t=t_n$ in (\ref{monotonicity of localized virial quantity}) and then letting $n \to \infty$, we have 
\[
\int R^2 \varphi \left( \frac{x}{R} \right) |u_0|^2 dx \le \int R^2 \varphi \left( \frac{x}{R} \right) Q^2 dx, \quad \forall R \ge R_0,
\] 
thus by Fatou's Lemma,
\begin{align*}
	\int |x|^2 |u_0(x)|^2 dx 
	& \le \liminf_{R \to \infty} \int R^2 \varphi \left( \frac{x}{R} \right) |u_0(x)|^2 dx \\
	& \le \liminf_{R \to \infty}  \int R^2 \varphi \left(\frac{x}{R} \right) Q^2 dx < \infty,
\end{align*}
which implies the finite variance of $u_0$, then we finish the proof.

\subsection{Complete the proof of Theorem \ref{theorem: convergence of u to soliton Q above the threshold}} 
It remains for us to check that the lifespan of $u$ in negative time direction is finite. First, from the proof of (\ref{derivative of virial quantity: above the threshold}),
\begin{lem}
	Let $u$ be a solution to (\ref{Hartree equation}) satisfying (\ref{condition: above the threshold}) and $T_+(u) = + \infty$ with initial data $u_0$ of finite variance, then
	\[
	\Im \int x \cdot \nabla u(x,t) \bar{u}(x,t) dx = \frac{1}{4} \dot{y}(t) >0
	\]
	for all $t$ in the interval of existence of $u$.
	\label{lemma: positivity of the derivative of virial quantity}
\end{lem}

As for $u$ given in Theorem \ref{theorem: convergence of u to soliton Q above the threshold}, we assume $T_- (u) = -\infty$. Since the equation (\ref{Hartree equation}) satisfies the time reversal symmetry, we can see that $v(x,t)= \bar{u}(x,-t)$ is also the solution to (\ref{Hartree equation}) on $\mathbb{R}_+$ and $v$ satisfies the condition in Lemma \ref{lemma: positivity of the derivative of virial quantity}, which implies that
\[
0 < \Im \int x \cdot \nabla v(x,t) \bar{v} (x,t) dx = - \Im \int x\cdot \nabla u(x,-t) \bar{u}(x,-t) dx, \; \forall t >0,
\]
but it contradicts to the fact that
\[
\Im \int x \cdot u(x,t) \bar{u}(x,t) dx >0, \; \forall t \in \left( T_-(u), T_+(u)  \right).
\]
Hence the negative time of existence of $u$ is finite.
\section{Convergence to $Q$ for $\| \nabla u_0 \|_2^\frac{1-s_c}{s_c} \| u_0 \|_2 < \| \nabla Q \|_2^\frac{1-s_c}{s_c} \| Q \|_2$}
\label{Convergence to Q below the threshold}
\begin{thm} 
	Under the same conditions as what in Theorem \ref{theorem: classification of threshold solutions at zero momentum case}, let $u$ be the solution to (\ref{Hartree equation}) satisfying
	\begin{equation}
		M[u]=M[Q], \quad E[u]= E[Q], \quad \| \nabla u \|_2 < \| \nabla Q \|_2,
		\label{condition: below threshold}
	\end{equation}
	and assume it does not scatter in positive time. Then there exists $(\theta_0, x_0) \in \mathbb{R}\backslash 2 \pi \mathbb{Z} \times \mathbb{R}^N$, $c ,C >0$ such that
	\[
	\| u - e^{it + i \theta_0} Q(\cdot -x_0) \|_{H^1(\mathbb{R}^N)} \le Ce^{-ct}.
	\]
	\label{convergence to Q below the threshold}
\end{thm}

First of all, we present a result about concentration compactness, which is essential to the later discussion in this section.
\begin{lem}[Concentration compactness] 	\label{concentration compactness}
	Let $u$ be a solution to (\ref{Hartree equation}) satisfying the conditions in Theorem \ref{convergence to Q below the threshold}. There exists a continuous function $x(t)$ such that 
	\begin{equation}
		K \triangleq \left\lbrace u(x+x(t),t): t\in [0,+\infty) \right\rbrace  
		\label{trajetory of u}
	\end{equation}
	is pre-compact in $H^1(\mathbb{R}^N)$.
\end{lem}
\begin{proof}
	The proof is essentially the same as \cite[Lemma $6.1$]{duyckaerts2010threshold}.
\end{proof}

    Repeating the analogous argument to what in \cite[Section $6.1$]{duyckaerts2010threshold}, the continuous function $x(t)$ defined in Lemma \ref{concentration compactness} can be further modified such that not only does it retain the property in Lemma \ref{concentration compactness}, but also it satisfies
	\begin{equation}
	x(t)=X(t), \quad \forall t \in D_{\delta_0},
	\label{mass center: modified version}
	\end{equation}
   where $X(t)$ is defined in (\ref{modulation: decomposition of u}).

 Moreover, we can obtain more informations about the behavior of mass center $x(t)$ as follows:
\begin{lem}
	As $u$ given in Theorem \ref{convergence to Q below the threshold}. Let $x(t)$ defined above, then
	
	\noindent $(i)$
	\begin{equation}
		P[u]= \Im \int \bar{u} \nabla u dx=0,
		\label{vanishing of momentum}
	\end{equation}
	and thus
	\begin{equation}
		\lim\limits_{t \to +\infty} \frac{x(t)}{t}=0.
		\label{mass center: behavior at infinity}
	\end{equation}
	
	\noindent $(ii)$ there exists a constant $C >0$ such that
	\begin{equation}
		|x(t)-x(s)| \le C, \quad \forall s,t \ge 0 \text{ with } |s-t| \le 2.
		\label{mass center: formally uniform continuity}
	\end{equation}
\end{lem}
\begin{proof}
	(i) The proof is almost the same as \cite[Proposition $4.1$, Lemma $5.1$]{duyckaerts2007scattering}. \\
	(ii) The result simply follows from Lemma \ref{concentration compactness}, the continuous dependence of the flow onto the initial data and the continuity of the flow.
\end{proof}

Next, we want to study the behavior of $\delta(t)$ at infinity. First we have the vanishing of $\delta(t)$ at infinity in the mean sense.
\begin{lem}[Convergence in mean]
	Let $u$ be a solution of (\ref{Hartree equation}) satisfying (\ref{condition: below threshold}), then
	\begin{equation}
		\lim\limits_{T \to \infty} \frac{1}{T}\int_{0}^{T} \delta (t) dt =0.
		\label{convergence in mean: inequality}
	\end{equation}
	\label{convergence in mean: lemma}
\end{lem}
\begin{proof}
	Here we use the localized virial quantity (\ref{localized virial: radial case}). We have already known that $\dot{y}_R(t)$ satisfies (\ref{first derivative of virial: radial case}) and $\ddot{y}_R(t)$ satisfies (\ref{second derivative of virial: radial}) with
	\begin{align}
		A_R(u(t))& = 4 \Re \sum_{i\not= j} \int_{\mathbb{R}^N} \frac{\partial^2 \varphi_R}{\partial x_i \partial x_j} \frac{\partial u}{\partial x_i}  \frac{\partial \bar{u}}{\partial x_j} dx \notag  + \sum_{i=1}^N \int_{\mathbb{R}^N} \left( \frac{\partial^2 \varphi_R}{\partial x_i^2} -2\right) \Big| \frac{\partial u}{\partial x_i} \Big|^2 dx - \int_{\mathbb{R}^N}  \left( \Delta^2 \varphi_R \right)  |u|^2 dx \notag\\
		& \qquad+ \left( \frac{2}{p} -2 \right)  \int_{\mathbb{R}^N} \left(\Delta \varphi_R  -2 N \right) \left( |\cdot|^{-3} * |u|^2 \right) |u|^2 dx \notag \\
		&  \qquad - \frac{2}{p}   \iint_{\mathbb{R}^N \times \mathbb{R}^N} |u(x)|^p |u(y)|^p \nabla_y \cdot  \left( \frac{(\nabla \varphi_R (x) -2x)-(\nabla \varphi_R(y)- 2y)}{|x-y|^{N-\gamma}} \right) dxdy.
		\label{virial: remaining term}
	\end{align}
	Note that if $|x| \le R$,
	\[
	\frac{\partial^2 \varphi_R}{\partial x_i \partial x_j}(x)= (\Delta^2 \varphi_R)(x)= \frac{\partial^2 \varphi_R}{\partial x_i^2}(x) -2 =2N-\Delta\varphi_R(x) =0, \quad 2x- \nabla \varphi_R (x)=0,
	\]
	it means
	\begin{equation}
		|A_R(u(t))| \lesssim \int_{|x| \ge R} |\nabla u|^2 +\frac{1}{R^2}|u| ^2 + \left( |\cdot|^{-(N-\gamma)} * |u|^p \right) |u|^p dx.
		\label{virial: remaining term estimate}
	\end{equation}
	With the compactness of $K$, $\forall \varepsilon >0$, $\exists R_0(\varepsilon)>0$ sufficiently large, such that
	\begin{equation}
		\int_{|x-x(t)| \ge R_0(\varepsilon)} |\nabla u|^2 +|u| ^2 + \left( |\cdot|^{-(N-\gamma)} * |u|^p \right) |u|^p dx < \varepsilon.
		\label{smallness of remaining term away from mass center}
	\end{equation}
	\indent By (\ref{mass center: behavior at infinity}), for $ \varepsilon >0$ given above, $\exists \;  t_0(\varepsilon) >0$ such that 
	\begin{equation}
		\Big| \frac{x(t)}{t} \Big| < \varepsilon, \text{ i.e. }  |x(t) | \le \varepsilon t, \quad \forall t \ge t_0(\varepsilon).
		\label{behavior of mass center for large t}
	\end{equation}
	Note that
	\begin{equation}
		\Big|\frac{1}{T} \int_{t_0(\varepsilon)}^T ( 8 s_c(p-1) \delta (t) +A_R(u(t)) dt \Big| 
		= \Big| \frac{1}{T}\int_{t_0(\varepsilon)}^T \ddot{y}_R (t) dt\Big| 
		 \le  \frac{CR}{T},
		\label{estimate on the second derivative of virial quantity: convergence in mean}
	\end{equation}
	where the last inequality is followed by
	\[
	|\dot{y}_R(t)|= \Big| 2 R \Im \int \nabla \varphi \left( \frac{x}{R} \right) \cdot \nabla u  \bar{u} dx\Big| \le R \| \nabla \varphi \|_\infty \| \nabla u \|_2 \| u \|_2 \le CR \text{ uniformly in } t>0.
	\]
	Now we have to choose an appropriate $R$ to ensure the smallness of $A_R(u(t))$ on $[t_0(\varepsilon), T]$ and here we let $R=R_0(\varepsilon) +\varepsilon T+1$, then by (\ref{virial: remaining term estimate}), (\ref{smallness of remaining term away from mass center}) and (\ref{behavior of mass center for large t}),
	\begin{align}
		|A_R(u(t))| 
		& \lesssim \int_{|x| \ge R} |\nabla u|^2 +\frac{1}{R^2}|u| ^2 + \left( |\cdot|^{-(N-\gamma)} * |u|^p \right) |u|^p dx \notag \\
		& \le \int_{|x-x(t)| \ge R_0(\varepsilon)} |\nabla u|^2 +|u| ^2 + \left( |\cdot|^{-(N-\gamma)} * |u|^p \right) |u|^p dx <\varepsilon.
		\label{estimate on AR(u(t)): convergence in mean}
	\end{align}
	Together with (\ref{estimate on the second derivative of virial quantity: convergence in mean}) and (\ref{estimate on AR(u(t)): convergence in mean}),
	\begin{equation*}
		\frac{1}{T} \int_0^T \delta(t) dt
		=\frac{1}{T} \int_0^{t_0 (\varepsilon)} \delta(t) dt  + \frac{1}{T} \int_{t_0(\varepsilon)}^{T} \delta(t) dt 
		 \lesssim \frac{1}{T} \int_0^{t_0 (\varepsilon)} \delta(t) dt + \frac{R_0(\varepsilon)+1}{T} +\varepsilon.
	\end{equation*}
	Letting $T \to \infty$ on both sides, (\ref{convergence in mean: inequality}) immediately follows from the arbitrariness of $\varepsilon>0$.
\end{proof}

As a corollary of Lemma \ref{convergence in mean: lemma},  $\delta(t)$ vanishes at infinity up to a subsequence.
\begin{cor}
	Let $u$ be the a solution of (\ref{Hartree equation}) satisfying (\ref{condition: below threshold}), then there exists a time sequence $\{t_n\}_{n=1}^\infty$ tending to $+\infty$, such that $\delta(t_n) \to 0$.
	\label{delta tends to 0 with a time sequence}
\end{cor}

Moreover, the behavior of $\delta(t)$ can be controlled by the mass center $x(t)$ in the following sense:
\begin{lem}[Virial-type estimates on $\delta$] Under the Assumption \ref{assumption: uniqueness of the ground state} and Assumption \ref{assumption: nondegeneracy of the ground state}, there exists a constant $C >0$ such that if $0 \le \sigma < \tau$ 
	\begin{equation}
		\int_\sigma^\tau \delta(t) dt \le C \left( 1+ \sup_{\sigma\le  t \le \tau} |x(t)|\right) \left( \delta(\sigma) + \delta(\tau) \right).
		\label{virial-type estimates on delta}
	\end{equation}	
	\label{lem: virial type estimates on delta}
\end{lem}
\begin{proof}
	As for the localized virial quantity $y_R$ as defined in (\ref{localized virial: radial case}), we wish $\delta$ to control the behavior of $\ddot{y}_R$, requiring us to estimate $A_R(u(t))$. We divide the argument into the following two cases, one for $\delta(t) \le \delta_0$ and another one for $\delta(t) > \delta_0$, while the former can be dealt by modulation and the latter relies on the concentration compactness. Precisely, as for the former case, by (\ref{modulation: decomposition of u}), Lemma \ref{lemma: the estimates of modulation parameters}, Lemma \ref{lemma: exponential decay of ground state and its derivatives}, Lemma \ref{lemma: estimate on J(z) and K(z)} and the fact that for any fixed $\theta_0, X_0$, $A_R(e^{i \theta_0} e^{it} Q(\cdot + X_0)) =0$ for any $R,t$, we can repeat the  same argument as what in (\ref{radial case: estimate on AR(u(t)): near the ground state}) to get
	\begin{align*}
		|A_R(u(t))|
		&= \Big| A_R(u)- A_R \left( e^{i(t+ \theta(t))} Q(x -X(t)) \right) \Big| \\
		& \lesssim e^{-CR_0} \delta(t) + \delta(t)^2+ \delta(t)^{2p}, \quad \text{ for any } R \ge R_0 + |X(t)|.
	\end{align*}
	Then we choose $R_0>0$ sufficiently large and $\delta_0 \ll 1$, such that
	\[
	\forall R \ge R_0 + |X(t)|= R_0+ |x(t)| \quad \Rightarrow \quad  |A_R(u(t))| \le 4s_c(p-1) \delta(t),\;  \forall t \in D_{\delta_0}.
	\]
	As for the latter case, if $\delta(t) \ge \delta_0$, then by (\ref{virial: remaining term}), (\ref{virial: remaining term estimate}) and (\ref{smallness of remaining term away from mass center}),
	\begin{align*}
		|A_R(u(t))|
		& \le C \int_{|x|\ge R} \left( |\nabla u|^2 +|u|^2 + \left( |\cdot|^{-(N-\gamma)} * |u|^p \right) |u|^p \right) dx \\
		& \le C \int_{|x-x(t)| \ge R-|x(t)|}  \left( |\nabla u|^2 +|u|^2 + \left( |\cdot|^{-(N-\gamma)} * |u|^p \right) |u|^p \right) dx < C \varepsilon, \; \forall R \ge |x(t)| + R_0(\varepsilon),
	\end{align*}
    where $R_0(\varepsilon)$ is the constant chosen in (\ref{smallness of remaining term away from mass center}). Therefore, if we let $\varepsilon= \frac{4s_c(p-1) \delta_0}{C}>0$, then there exists $R_1=R(\varepsilon)>0$ such that
	\[
	\forall R \ge R_1 + |x(t)| \Rightarrow |A_R(u(t))| \le C \varepsilon =4s_c(p-1) \delta_0 \le 4s_c(p-1) \delta(t), \forall t\notin D_{\delta_0}.
	\]
	Consequently, if we let $R_2= \max \{ R_0,R_1\}$, then
	\begin{equation*} 
	    |A_R(u(t))| \le 4s_c(p-1)  \delta(t), 	\;  \forall R \ge R_2(1+|x(t)|) , \; \forall t \ge 0.
		\label{bound on remaining term of virial}
	\end{equation*}
	In particular, if we let $R=R_2 \left( 1+ \sup_{\sigma\le  t \le \tau} |x(t)| \right) $, then
	\[
	\ddot{y}_R(t) = 8s_c (p-1)\delta(t) +A_R(u(t)) \ge 4s_c(p-1) \delta(t), \quad \forall t \in [\delta, \tau],
	\]
	which indicates that
	\[
	\int_\sigma^\tau \delta(t) dt \lesssim \int_\sigma^\tau \ddot{y}_R(t) dt = \dot{y}_R(\tau) -\dot{y}_R(\sigma).
	\]
	It remains for us to estimate $\dot{y}_R(t)= 2 R \Im \int_{\mathbb{R}^N} \nabla \varphi\left( \frac{x}{R} \right) \cdot \nabla u  \bar{u}  dx$. If $\delta(t) \ge \delta_0$, then
	\begin{equation*}
		|\dot{y}_R(t)|
		 \le 2 R \| \nabla \varphi\|_\infty \| \nabla u \|_2 \| u \|_2 
		 \le C R \| Q \|_2 \left( \delta (t) + \|\nabla Q \|_2^2 \right)^\frac{1}{2} 
		 \le CR \delta(t).
	\end{equation*}
	If $ t\in D_{\delta_0}$, then by (\ref{modulation: decomposition of u}),
	\begin{align*}
		\dot{y}_R(t) 
		&=  2R \Im \int_{\mathbb{R}^N} \nabla \varphi \left( \frac{x+X(t)}{R} \right) \cdot  \nabla Q \bar{v}  dx   + 2R \Im \int_{\mathbb{R}^N} \nabla \varphi \left( \frac{x+X(t)}{R} \right) \cdot \nabla v Q dx   \\
		&\quad + 2R \Im \int_{\mathbb{R}^N} \nabla \varphi \left( \frac{x+X(t)}{R} \right) \cdot \nabla v \bar{v} dx,  \text{ where } v =\alpha Q +h,
	\end{align*}
which yields $|\dot{y}_R(t)| \le CR \left( \delta(t) + \delta(t)^2 \right) \le CR \delta(t)$ by Lemma \ref{lemma: the estimates of modulation parameters}. Consequently, combining the argument above, we immediately get
	\[
	\int_\sigma^\tau \delta(t) dt \lesssim R \left(\delta(\sigma)+ \delta(\tau)\right) = R_2\left( 1+ \sup_{\sigma\le  t \le \tau} |x(t)| \right) \left( \delta(\sigma) + \delta(\tau) \right).
	\]
\end{proof}

\begin{lem}[Control of the variations of $x(t)$] Under the Assumption \ref{assumption: uniqueness of the ground state} and Assumption \ref{assumption: nondegeneracy of the ground state}, there exists a constant $C>0$ such that
	\begin{equation}
		\forall \sigma, \tau >0 \text{ with } \sigma + 1 \le \tau, \quad |x(\tau)- x(\sigma)| \le C \int_\sigma^\tau \delta (t) dt.
		\label{mass center: the behavior can be control by delta}
	\end{equation}
	\label{control of the variations of mass center}
\end{lem}
\begin{proof}
	First we claim that there exists a constant $\delta_2 >0$ such that $\forall \tau \ge 0$,
	\begin{equation}
		\inf_{t \in [\tau, \tau+2]} \delta(t) \ge \delta_2 \; \text{ or } \sup_{t \in [\tau,\tau+2]} \delta(t) < \delta_0.
		\label{delta dichotomy: behavior on a bounded interval}
	\end{equation}
	If not, there exist two sequences $\{t_n \} , \{t_n'\}$ such that
	\[
	\delta(t_n) \to 0, \quad \delta(t_n') \ge \delta_0, \quad |t_n -t_n'| \le 2.
	\]
	After extracting a subsequence, we have $t_n' -t_n \to t_0 \in [-2,2]$. By the compactness of $K$, there exists $v_1 \in H^1$ such that $u(\cdot + x(t_n), t_n) \to v_1$ in $H^1(\mathbb{R}^N)$, then together with the fact that $\delta(t_n) \to 0$,
	\[
	M[u]=M[v_1]=M[Q],  \; E[u]= E[v_1]= E[Q], \; \| \nabla v_1 \|_{L^2}  = \|\nabla Q \|_2,
	\]
	which indicates that $v_1=e^{i \theta_0 } Q(\cdot -x_0)$ for some $(\theta_0, x_0) \in \mathbb{R}\backslash 2\pi \mathbb{Z} \times \mathbb{R}^N$ by Lemma \ref{lemma: behavior of gradient of solution}. Moreover, it is easy to see that $e^{it} e^{i \theta_0} Q(\cdot -x_0)$ is the solution to (\ref{Hartree equation}) with initial data $v_1$. By the continuous dependence of solution to (\ref{Hartree equation}) onto the initial data,
	\[
	u(\cdot +x(t_n), t_n+t_0) \to e^{i t_0} e^{i \theta_0} Q(\cdot-x_0).
	\]  
	Thus together with the continuity of the flow,
	\[
	u(\cdot + x(t_n), t_n') = u (\cdot +x(t_n), t_n + t_n' -t_n) \to  e^{it_0} e^{i \theta_0} Q(\cdot-x_0).
	\]
	Since $|t_n -t_n'| \le 2$, $|x(t_n) -x(t_n')| \le C< \infty$ uniformly by (\ref{mass center: formally uniform continuity}), so there exists $x_1$ such that $x(t_n)- x(t_n') \to x_1$ and
	\[
	u(\cdot + x(t_n'), t_n') = u(\cdot + x(t_n) +x(t_n') - x(t_n), t_n') \to e^{it_0} e^{i \theta_0} Q(\cdot - x_0 -x_1) \text{ in } H^1.
	\] 
	Therefore,
	\[
	\delta(t_n') =  \big| \| \nabla u(t_n') \|_2^2 - \| \nabla Q \|_2^2  \big| \to 0,
	\]
	which contradicts to our assumption $\delta(t_n') \ge \delta_0, \forall n \in \mathbb{N}$ and therefore we have proved (\ref{delta dichotomy: behavior on a bounded interval}).
	
	\indent Next, we claim that (\ref{mass center: the behavior can be control by delta}) holds when $\tau \le \sigma + 2$. In fact,  if $\inf_{t \in [\tau, \tau+2]} \delta(t) \ge \delta_2$, then by (\ref{mass center: formally uniform continuity}),
	\[
	\int_\sigma^\tau \delta(t) dt \ge\left( \inf_{t \in [\tau, \tau+2]} \delta(t) \right)  (\tau-\sigma)  \ge \delta_2( \tau -\sigma) \ge \delta_2 \gtrsim |x(\tau)- x(\sigma)|.
	\]
	If $\sup_{t \in [\tau, \tau+2]} \delta(t) <\delta_0$, then by Lemma \ref{modulation parameters: estimate on their time-derivatives} and Corollary \ref{mass center: modified version},
	\[
	|x(\tau) -x(\sigma)| = |X(\tau) -X(\sigma)| \le \int_\sigma^\tau |\dot{X}(t)| dt \lesssim \int_{\sigma}^\tau \delta(t)dt.
	\]
	When it comes to the general case for $\tau \ge \sigma+2$, we can divide $[\sigma, \tau]$ into intervals of length at least $1$ and at most $2$ and then stick the inequalities together to get (\ref{mass center: the behavior can be control by delta}) without the assumption $\tau \le \sigma +2$.
\end{proof}

\begin{proof}[Proof of Theorem \ref{convergence to Q below the threshold}]
	By Corollary \ref{delta tends to 0 with a time sequence}, there exists $\{t_n\}_{n \ge 1}$ such that
	\begin{equation}
	\delta(t_n) \to 0, \quad t_n \to \infty \text{ and } 1+ t_n \le t_{n+1}.
	\label{assumption on tn}
	\end{equation}
	Then for any fixed $N >0$, $\forall n \ge N+1$,
	\[
	|x(t)- x(t_N)| \le C \int_{t_N}^t \delta(s)ds \le C \int_{t_N}^{t_n} \delta(t) dt \le C \left( \delta(t_n) + \delta(t_N) \right)  \left( 1+ \sup_{t \in [t_N,t_n]} |x(t)|\right), \; \forall t \in [t_N, t_n],
	\]
	which implies that
	\[
	\sup_{t \in [t_N,t_n]} |x(t)| \le |x(t_N)| +C \left( \delta(t_n)+ \delta(t_N) \right) \left( 1+ \sup_{t\in [t_N, t_n]} |x(t)|  \right).
	\]
	By the assumption on $\{t_n\}_{n \ge 1}$ as shown in (\ref{assumption on tn}), we can choose $N \ge 1$ sufficiently large such that $\forall n \ge N+1$, $C\left( \delta(t_n) + \delta(t_N) \right) \le \frac{1}{2}$, then
	\[
	\sup_{t \in [t_N,t_n]} |x(t)| \le C |x(t_N)| + \frac{1}{2}< \infty, \; \forall n \ge N+1.
	\]
	By the arbitrariness of $n \ge N+1$ and the continuity of $x(t)$ on $\overline{\mathbb{R}^+}$, $x(t)$ is uniformly bounded on $[0,+\infty)$ and thus by Lemma \ref{lem: virial type estimates on delta},
	\[
	\int_\sigma^{t_n} \delta(t) dt \lesssim \delta(\sigma) + \delta(t_n) 
	\]
	uniformly in $n$. Letting $n \to \infty$,
	\[
	\int_{\sigma}^\infty \delta(t) dt \lesssim \delta(\sigma),
	\]
	which ensures $\int_{\sigma}^\infty \delta(t)dt \le C e^{-ct}$ by Gronwall's inequality. Then Theorem \ref{convergence to Q below the threshold} simply follows from Lemma \ref{auxilliary lemma}.
\end{proof}

\section{Uniqueness and completing the main theorem}
\label{uniqueness: section}
\subsection{Exponential Strichartz estimates for solutions to linearized equation}
Before we proceed with the argument, we remark that the case for $p=2$ is easy to deal with because we can expand the nonlinear term directly. As for $p>2$, the question becomes tougher since we encounter fractional power, in which case we can not expand nonlinear term directly anymore, and it requires us to analyze more carefully, especially for $p \in (2,3)$, which causes the low regularity of nonlinear term. Thus in the later discussion, we mainly focus on $p>2$. As for $p=2$, we just mention the underlying difficulties in later discussion.

Let $h$ satisfy $u=e^{it}(Q+h)$, then $h$ satisfies the following linearized equation
\begin{equation}
	i \partial_t h + \Delta h -h + V(h) + R(h)=0,
	\label{linearized equation: complex version}
\end{equation}
where $V(h)$ is the linear term defined by
\begin{align}
	V(h)& =  p \left( |\cdot|^{-(N-\gamma)} * (Q^{p-1} \Re h) \right) Q^{p-1} \notag \\
	&\quad + \left( |\cdot|^{-(N-\gamma)}* Q^p \right) \left( \left( \frac{p}{2} -1 \right) Q^{p-2} \bar{h} + \frac{p}{2} Q^{p-2} h \right)
	\label{linear term of linearized equation}
\end{align}
and $R(h)$ is as in (\ref{linearized equation: nonlinear term}). We remark here that the notation $h$ in this section is a little confusing with (\ref{modulation: decomposition of u}), but this section is not involved in the decomposition (\ref{modulation: decomposition of u}), so we no longer use other notations in this section to distinguish them anymore.

For any time interval $I \subset \mathbb{R}$, together with Lemma \ref{lemma: Strichartz estimate on nonlinear term: difference version for p in (2,3)}, Lemma \ref{lemma: Strichartz estimate on nonlinear term: difference version for p ge 3} and Remark \ref{remark: we can bound the nonliear by L2 Strichartz norm with gradient},
\begin{equation}
	\| \left\langle \nabla \right\rangle R(h) \|_{S'(I,L^2)} 
	\lesssim
	\begin{cases}  
		|I|^{\alpha_1} \| \left\langle \nabla \right\rangle h \|_{S(I,L^2)}^2 + |I|^{\alpha_2} \| \left\langle \nabla \right\rangle h \|_{S(I,L^2)}^{p-1}  + \| \left\langle \nabla \right\rangle h \|_{S(I,L^2)}^{2p-1}, &  p \ge 3, \\
		|I|^\beta \| \left\langle \nabla \right\rangle h \|_{S(I,L^2)}^{p-1} + \| \left\langle \nabla \right\rangle h \|_{S(I,L^2)}^{2p-1}, &  p \in (2,3),
	\end{cases}
	\label{Strichartz estimates on nonlinear term}
\end{equation}
for some $\alpha_1, \alpha_2, \beta >0$. As for the linear term $V(h)$, by Lemma \ref{lemma: Strichartz estimate on nonlinear term}, we obtain that
\begin{equation}
	\| \left\langle \nabla \right\rangle V(h) \|_{S'(I,L^2)} 
	\lesssim | I|^{\gamma} \| \left\langle \nabla \right\rangle h \|_{S (I , L^2)} 
	\label{Strichartz estimates on linear term}
\end{equation}
for some $\gamma >0$.
Then by Lemma \ref{lemma: Strichartz estimates}, (\ref{Strichartz estimates on nonlinear term}) and (\ref{Strichartz estimates on linear term}),
\begin{align*}
	\| \left\langle \nabla \right \rangle h \|_{S(I,L^2)}
	&= \Bigg\| \left\langle \nabla \right \rangle \left[ e^{i(t'-t)(\Delta -1)} h(t) - i \int_t^{t'} e^{i(t'-s) (\Delta -1)} \left( -V(h)-R(h) \right)(s) ds \right] \Bigg\|_{S(I, L^2)} \\
	& \lesssim \| h(t) \|_{H^1} + \| \left\langle \nabla \right \rangle V(h) \|_{S'(I,L^2)}  + \| \left\langle \nabla \right \rangle R(h) \|_{S'(I ,L^2)} \\
	& \lesssim \| h(t) \|_{H^1} +  \tau_0^{\gamma}  \| \left\langle \nabla \right\rangle h \|_{S (I , L^2)} \\
	& \qquad
	+ \begin{cases}  
		\tau_0^{\alpha_1} \| \left\langle \nabla \right\rangle h \|_{S(I,L^2)}^2 + \tau_0^{\alpha_2} \| \left\langle \nabla \right\rangle h \|_{S(I,L^2)}^{p-1} +  \| \left\langle \nabla \right\rangle h \|_{S(I,L^2)}^{2p-1}, &  p \ge 3, \\
		\tau_0^\beta \| \left\langle \nabla \right\rangle h \|_{S(I,L^2)}^{p-1} + \| \left\langle \nabla \right\rangle h \|_{S(I,L^2)}^{2p-1}, &  p \in (2,3),
	\end{cases}
\end{align*}
where we let $I=(t, t+ \tau_0)$. If we assume $\| h(t) \|_{H^1(\mathbb{R}^N)} \lesssim e^{-ct}$, then, by continuity argument, we can choose $\tau_0 \ll 1$ such that
\[
\big\| \left\langle \nabla \right\rangle h(t) \big\|_{S(I,L^2)}  \lesssim \big\| h(t) \big\|_{H^1} \lesssim e^{-ct},
\]
where the constant is independent of the choice of interval $I$. Then it yields
\begin{equation*}
	\| \left\langle \nabla \right \rangle h \|_{S((t, \infty),L^2)} 
	\le \sum_{n =0} ^\infty \| \left\langle \nabla \right \rangle h \|_{S((t+n \tau_0, t+ (n+1) \tau_0),L^2)} 
	\lesssim \sum_{n=0}^\infty e^{-c(t+ n \tau_0)} \lesssim e^{-ct}.
\end{equation*}
Consequently, we have the following exponential Strichartz estimate:
\begin{lem}[Exponential Strichartz estimate]
	For $h$ the solution to (\ref{linearized equation: complex version}), if $\| h(t) \|_{H^1} \lesssim e^{-ct}, \; \forall t \ge
	T$ for some $c >0$, then
	\[
	\| \left\langle \nabla \right \rangle h \|_{S((t, \infty),L^2)} \lesssim e^{-ct}, \quad \forall  t  \ge T.
	\]
	\label{lemma: exponential Strichartz estimates}
\end{lem}
\begin{rmk}
	We can extract the argument above to conclude the result as follows: let $h$ satisfy $\partial_t h + \mathcal{L} h = R$ with $\| h(t) \|_{H^1} \lesssim e^{-ct}$ and $\| \left\langle \nabla \right \rangle R \|_{S'((t,\infty),L^2)} \lesssim e^{-c't} $ for some $c'  \ge c >0$, then $\| \left\langle \nabla \right \rangle h \|_{S((t,+ \infty),L^2)} \lesssim e^{-ct}$.
	\label{remark: exponential Strichartz estimate}
\end{rmk}

\subsection{Improving the rate of decay}
\begin{lem}
	Under the Assumption \ref{assumption: uniqueness of the ground state} and Assumption \ref{assumption: nondegeneracy of the ground state}, consider the solution $h$ to
	\begin{equation}
		\partial_t h + \mathcal{L} h= R, \quad (x,t) \in \mathbb{R}^N \times (t_0, +\infty),
		\label{linearized equation: improved case}
	\end{equation}
	with
	\begin{equation}
		\| h(t)  \|_{H^1} \lesssim e^{-c_0 t}, \quad \| \left\langle \nabla \right \rangle R \|_{S'((t,\infty),L^2)}  \lesssim e^{-c_1 t}, \text{ and } c_0 < c_1.
		\label{assumption on the decay: improved case} 
	\end{equation}
	$\bullet$ If $c_0 < c_1 \le e_0$ or $e_0 < c_0 < c_1$, then $\| h(t) \|_{H^1} + \| \left\langle \nabla \right\rangle h\|_{S((t, + \infty),L^2)} \lesssim e^{-c_1^- t}$. \\
	$\bullet$ If $c_0 \le e_0 <c_1$, then there exists $A \in \mathbb{R}$ such that
	\[
	\| h(t)-Ae^{-e_0 t} \mathcal{Y}_+  \|_{H^1(\mathbb{R}^N) } + \| \left\langle \nabla \right\rangle \left( h(t)- Ae^{-e_0 t} \mathcal{Y}_+ \right) \|_{S((t, + \infty),L^2)} \lesssim e^{-c_1^- t}.
	\]
	\label{lemma: improve the decay of solution to linearized equation}
\end{lem}

\begin{proof}
	By Remark \ref{remark: exponential Strichartz estimate}, it suffices to consider the behavior of  $\| h(t) \|_{H^1}$. First, we decompose $h$ into
	\begin{equation}
		h(t)= \alpha_+(t) \mathcal{Y}_+ + \alpha_-(t) \mathcal{Y}_- + \sum_{j=0}^N \beta_j(t) Q_j + h_\perp(t),\quad h_\perp \in G_\perp',
		\label{decomposition of v: improved case}
	\end{equation}
	where $Q_0 = \frac{iQ} {\| Q\|_2}, \; Q_j = \frac{\partial_{x_j} Q}{\| \partial_{x_j} Q \| _2}, \; 1 \le j \le N$, and out of simplicity, we can also normalize the eigenfunctions $\mathcal{Y}_\pm$ such that $B(\mathcal{Y}_+, \mathcal{Y}_-)=1$. Then
	\begin{equation} 
		\alpha_+(t)= B(h(t), \mathcal{Y}_-), \quad \alpha_-(t)=B(h(t), \mathcal{Y}_+),
		\label{alpha(pm): definition}
	\end{equation}
	\begin{equation}
		\beta_j(t)= \left\langle h(t), Q_j \right\rangle - \alpha_+(t) \left\langle \mathcal{Y}_+, Q_j \right\rangle - \alpha_-(t) \left\langle \mathcal{Y}_-, Q_j \right\rangle.
		\label{beta(j): definition}
	\end{equation}

	\noindent \textit{The estimates on} $\alpha_\pm$.
	First, by Proposition \ref{proposition: spectral information of linearized operator L}, we have the decay of $\alpha_\pm(t)$ at infinity,
	\begin{equation}
		|\alpha_\pm|
		=|B(h, \mathcal{Y}_\mp) | 
		 \lesssim \Big| \frac{1}{2} \int (L_+ h_1) \mathcal{Y}_1 dx \Big| + \Big| \frac{1}{2} \int (L_- h_2) \mathcal{Y}_2 dx  \Big|
		  \lesssim \| h(t) \|_{L_x^2}  \lesssim e^{-c_0 t}.
		\label{estimate of alpha(pm)}
	\end{equation}
	 Next, we want to improve the decay of $\alpha_\pm$. As for $\alpha_+$, noting that $\alpha_+$ satisfies the following differential equation
	\begin{equation*}
		\dot{\alpha}_+(t) 
		= \partial_t B(h,\mathcal{Y}_-) 
		= B(\partial_t h, \mathcal{Y}_-) 
		= B(-\mathcal{L} h + R , \mathcal{Y}_-)  
		= -e_0 \alpha_+ + B(R, \mathcal{Y}_-),
	\end{equation*}
	we get
	\begin{equation}
	\partial_t \left( e^{e_0 t} \alpha_+\right) = e^{e_0 t}B(R, \mathcal{Y}_-),
	\label{derivative of alpha+: equation}
	\end{equation}
	which means that
	\[
	e^{e_0 t} \alpha_+(t) = \alpha_+(0) + \int_0^t e^{e_0 s} B(R, \mathcal{Y}_-) ds.
	\]
	
	If $e_0 \ge c_1>c_0$, then by using the estimate that
	\begin{equation}
		 \int_n^{n+1} e^{e_0 s} \big|B(R, \mathcal{Y}_-) \big| ds  \lesssim e^{e_0 n} \| R \|_{S' \left( (n, +\infty),L^2 \right)} \le e^{(e_0 -c_1) n}, 
		\label{improving decay: estimate on each bounded interval of nonlinear term}
	\end{equation}
	we conclude that
	\begin{equation*}
		\Bigg| \int_0^t e^{e_0 s} B(R, \mathcal{Y}_-) ds \Bigg| \le \sum_{n=0}^{[t]}  \int_n^{n+1} e^{e_0 s} \big| B(R , \mathcal{Y}_-) \big| ds  \lesssim \sum_{n=0}^{[t]} e^{(e_0 -c_1) n} \lesssim 
		\begin{cases}
		e^{(e_0-c_1) t}, & \text{ if } e_0 >c_1, \\
		t, & \text{ if } e_0 =c_1,
		\end{cases}
	\end{equation*}
    and thus $|\alpha_+ (t)| \lesssim e^{-c_1^- t}$.
    
    If $e_0 < c_0 < c_1$, then (\ref{estimate of alpha(pm)}) implies that $e^{e_0 t} \alpha_+(t) \to 0$ as $t \to \infty$ and thus by (\ref{derivative of alpha+: equation}) we have
	\[
	e^{e_0 t} \alpha_+(t) = \int_t^\infty e^{e_0 s} B(R , \mathcal{Y}_-) ds.
	\]
	By the same argument as what in (\ref{improving decay: estimate on each bounded interval of nonlinear term}),
	\begin{align*}
		|e^{e_0 t}\alpha_+ (t) |
		\lesssim \sum_{n=[t]}^\infty \Bigg| \int_n^{n+1}  e^{e_0 s} B(R , \mathcal{Y}_-) ds \Bigg| 
		\le \sum_{n=[t]}^\infty e^{(e_0 - c_1) n } \lesssim e^{(e_0-c_1) t}, 
	\end{align*}
	which means that $|\alpha_+(t)| \lesssim e^{-c_1 t}$.
	
	If $c_0 \le e_0 < c_1$, then
	\begin{equation*}
		A= \alpha_+(0) + \int_0^\infty e^{e_0 s} B(R, \mathcal{Y}_-) ds \in \mathbb{R}
	\end{equation*}
	by
	\begin{equation*}
		\Bigg| \int_0^\infty e^{e_0 s} B(R, \mathcal{Y}_-) ds \Bigg|
		\le \sum_{n=0}^\infty \Bigg| \int_n^{n+1} e^{e_0 s} B(R, \mathcal{Y}_-) ds  \Bigg| \lesssim \sum_{n=0}^\infty e^{(e_0 - c_1) n } < \infty.
	\end{equation*}
   As for $A$ given above, we then obtain that
	\begin{equation*}
		|e^{e_0 t}\alpha_+(t)- A|
		 = \Bigg| \int_t^\infty e^{e_0 s} B(R, \mathcal{Y}_-) ds \Bigg| 
		 \lesssim \sum_{n=[t]}^\infty e^{(e_0-c_1) n } \lesssim e^{(e_0 -c_1) t}.
	\end{equation*}
	Hence $\alpha_+(t)$ satisfies
	\begin{equation}
		\begin{cases}
			|\alpha_+ (t)| \lesssim e^{-c_1^- t}, & \text{if } c_0 < c_1 \le e_0 \text{ or } e_0 < c_0 < c_1, \\
			|\alpha_+(t) - A e^{-e_0 t} | \lesssim e^{-c_1 t}, & \text{if } c_0 \le e_0 < c_1.
		\end{cases}
		\label{the estimates on alpha+}
	\end{equation}
	
	Similarly, as for $\alpha_-$,
	\[
	\partial_t \left( e^{-e_0 t} \alpha_- \right) = e^{-e_0 t} B(R , \mathcal{Y}_+) \quad \Rightarrow \quad e^{-e_0 t} \alpha_-(t) = \int_t^\infty e^{-e_0 s} B(R,  \mathcal{Y}_+) ds.
	\]
	By the same argument as what for $\alpha_+$ when $e_0 < c_0 < c_1$,
	\begin{equation}
		|\alpha_-(t)| \lesssim e^{e_0 t} \sum_{n=[t]}^\infty \Bigg| \int_{n}^{n+1} e^{-e_0 s} B(R, \mathcal{Y}_+)(s) ds \Bigg| \lesssim e^{e_0 t} \sum_{n=[t]}^\infty e^{-(e_0 + c_1) n} \lesssim e^{-c_1 t}.
		\label{improve the decay: the estimate on alpha-}
	\end{equation}
	
	\noindent \textit{The estimates on the remaining terms and completing the proof of the case for $A=0$.} First, since $B(\mathcal{L} f, g)= - B(f, \mathcal{L}g)$, we have $B(\mathcal{L} h, h)=0$, which further implies that
	\[
	\partial_t \Phi(h) = 2 B(\partial_t h , h) = 2B(-\mathcal{L} h +R, h) = 2B(R, h).
	\]
	Moreover, by the assumption $\| h(t) \|_{H^1} \lesssim e^{-c_0 t}$ and Remark \ref{remark: exponential Strichartz estimate}, we obtain that
	\begin{align*}
		\int_{n}^{n+1} B(h, R) ds
		&= \frac{1}{2} \int_n^{n+1} \int_{\mathbb{R}^N} (L_+ h_1)  R_1 dxds + \frac{1}{2} \int_n^{n+1} \int_{\mathbb{R}^N}(L_- h_2) R_2 dx ds \\
		& \lesssim \| \left\langle \nabla \right\rangle h \|_{S((n,+\infty),L^2)}  \| \left\langle \nabla \right\rangle R \|_{S'((n,+\infty),L^2)}  \\
		& \lesssim e^{-(c_0+ c_1)n},
	\end{align*}
	which then implies that
	\begin{equation*}
		|\Phi(h)| = \Bigg| \int_t^\infty 2B(R, h) ds \Bigg| \lesssim e^{-(c_0 +c_1)t}.
	\end{equation*}
	Furthermore, by decomposition of $h$ in (\ref{decomposition of v: improved case}),
	\[
	\Phi(h)=B(h,h)= 2 \alpha_+ \alpha_- + B(h_\perp, h_\perp),
	\]
	by (\ref{corecivity of Phi acting on h}) and (\ref{estimate of alpha(pm)}),
	\[
	\| h_ \perp(t) \|_{H_x^1}^2 \simeq B(h_\perp, h_ \perp) = \Phi(h) - 2 \alpha_+ \alpha_- \lesssim e^{-(c_0 + c_1) t} + e^{-2c_1^- t} \lesssim e^{-(c_0 + c_1) t}
	\]
	for large $t$, i.e.
	\begin{equation} 
		\| h_ \perp(t) \|_{H^1} \lesssim e^{-\frac{c_0 + c_1}{2} t}, \quad \forall t \ge  t_0 \text{ for some } t_0 >0.
		\label{the estimate on v(perp)}
	\end{equation}
	Besides, by (\ref{beta(j): definition}),
	\begin{align*}
		\dot{\beta}_j(t)
		&= \left\langle  \dot{h}(t), Q_j \right\rangle - \dot\alpha_+(t) \left\langle \mathcal{Y}_+, Q_j \right\rangle - \dot\alpha_-(t) \left\langle \mathcal{Y}_-, Q_j \right\rangle \\
		& = \left\langle - \mathcal{L} h +R - \dot{\alpha}_+ \mathcal{Y}_+ - \dot{\alpha}_- \mathcal{Y}_-, Q_j \right\rangle \\
		& =(-\dot{\alpha}_+ - e_0 \alpha_+) \left\langle \mathcal{Y}_+, Q_j \right\rangle + (- \dot{\alpha}_- + e_0 \alpha_-) \left\langle \mathcal{Y_-}, Q_j \right\rangle - \left\langle \mathcal{L} h_\perp, Q_j \right\rangle + \left\langle R, Q_j \right\rangle,
	\end{align*}
    then together with the $\beta_j (t)\to 0$ as  $t \to \infty$,
    \begin{align}
    	|\beta_j(t)|
    	&= \Big| \int_t^\infty \dot{\beta}_j (s) ds \Big| \notag \\
    	& \lesssim \int_t^\infty \big|B(R, \mathcal{Y}_ \pm) \big| + \| h_\perp(s) \|_{H^1} + \big| \left\langle R, Q_j \right\rangle \big|  ds \notag\\
    	& \lesssim \sum_{n=0}^\infty \| \left\langle \nabla \right\rangle R \|_{S'((t+n, t+n+1), L^2)} + \| h_\perp \|_{L_s^\infty H_x^1(t+n,t+n+1)} \notag\\
    	& \lesssim \sum_{n=0}^\infty \left( e^{-c_1(t+n)} + e^{-\frac{c_0+c_1}{2}(t+n)} \right) \lesssim e^{-\frac{c_0 +c_1}{2} t}, \quad \forall t \ge t_0.
    	\label{the estimate on beta(j)}
    \end{align}
	Combining (\ref{decomposition of v: improved case}), (\ref{the estimates on alpha+}), (\ref{improve the decay: the estimate on alpha-}), (\ref{the estimate on v(perp)}) and (\ref{the estimate on beta(j)}), we obtain that
	\[
	\| h(t) \|_{H^1} \lesssim e^{-\frac{c_0+c_1}{2} t}, \quad \forall t \ge t_0.
	\]
	Taking $c_0 \to \frac{c_0+c_1}{2}$ and $c_1 \to c_1$ and repeating the argument again and again, we then improve the decay to $	\| h(t) \|_{H^1} \lesssim e^{-c_1^- t}, \quad \forall t \ge t_0.$
	
	\noindent \textit{Complete the proof.} It remains for us to check the case for $A \not = 0$, i.e. when $c_0 \le e_0 < c_1$. In fact, we notice that the function
	\[
	\widetilde{h}(t)=h(t)- A e^{-e_0 t} \mathcal{Y}_+
	\]
	satisfies the equation
	\[
	\partial_t \widetilde{h} + \mathcal{L} \widetilde{h} = R
	\]
	with $\| \widetilde{h}(t) \|_{H^1(\mathbb{R}^N)} \lesssim e^{-c_0 t},$ and $|\widetilde{\alpha}_+(t)| = B(\widetilde{h}(t), \mathcal{Y}_-) \lesssim e^{-c_1 t}$ by (\ref{the estimates on alpha+}). Repeating the previous argument  on $\widetilde{h}$, we immediately get the desired conclusion.
\end{proof}

\subsection{Existence of special solutions}
\begin{prop}
	Under the Assumption \ref{assumption: nondegeneracy of the ground state}. Let $A \in \mathbb{R}$, then there exists a sequence $(\mathcal{Z}_j^A)_{j \ge 1} \subset \mathcal{S} $ such that $\mathcal{V}_k^A \triangleq \sum_{j=1}^k e^{-j e_0 t} \mathcal{Z}_j^A$ satisfies
	\begin{equation}
		\partial_t \mathcal{V}_k^A + \mathcal{L} \mathcal{V}_k^A = i R(\mathcal{V}_k^A) + O(e^{-(k+1) e_0 t}) \text{ in } \mathcal{S}(\mathbb{R}^N).
		\label{approximation solutions}
	\end{equation}
	\label{proposition: approximation solution}
\end{prop}

\begin{rmk}
	Let $U_k^A= e^{it} (Q + \mathcal{V}_k^A)$, then by the proposition above, $U_k^A$ satisfies
	\begin{equation}
		i \partial_t U_k^A + \Delta U_k^A + \left( |\cdot|^{-(N-\gamma)} * \big|U_k^A \big|^p  \right) \big|U_k^A \big|^{p-2} U_k^A = O(e^{-(k+1) e_0 t}) \text{ in } \mathcal{S} (\mathbb{R}^N).
		\label{the construction of Q(k,A)}
	\end{equation}
	\label{remark: approximation solution}
\end{rmk}
\begin{proof}[Proof of Proposition \ref{proposition: approximation solution}.] As for $ p>2$, we follow the proof of \cite[Proposition $4.1$]{campos2020threshold} and  \cite[Lemma $6.1$]{duyckaerts2009dynamic} .  As for $J(z),K(z)$ defined in (\ref{J(z): definition}) and (\ref{K(z): definition}), they are both real-analytic on the disk $\Big\{ z \in \mathbb{C}\big| |z| <1 \Big\}$ and satisfy
	\[
	J(z)= \sum_{j_1 + j_2 \ge 2} a_{j_1 j_2}  z^{j_1} \bar{z}^{j_2}, \quad K(z) = \sum_{k_1 + k_2 \ge 2} b_{k_1 k_2} z^{k_1} \bar{z}^{k_2}
	\]
	with normal convergence of the series and all its derivatives for $|z| \le \frac{1}{2}$. Thus if $|\omega| \triangleq |Q^{-1} h| \le \frac{1}{2}$, we obtain that
	\begin{align}
		R(h)
		& = \left( |\cdot|^{-(N-\gamma)} * \left( Q^p \sum_{j_1+ j_2 \ge 2} a_{j_1 j_2} \omega^{j_1} \bar{\omega}^{j_2} \right)  \right) Q^{p-1} \left( \sum_{k_1+ k_2 \ge 2} b_{k_1 k_2} \omega^{k_1} \bar{\omega}^{k_2} + \left( 1+ \frac{p}{2} \omega + \frac{p-2}{2} \bar{\omega} \right) \right) \notag \\
		& \quad + \left( |\cdot|^{-(N-\gamma)} * \left( Q^p \left(1+ \frac{p}{2} \omega + \frac{p}{2} \bar{\omega} \right) \right) \right) Q^{p-1} \left( \sum_{k_1 + k_2 \ge 2} b_{k_1 k_2} \omega^{k_1} \bar{\omega}^{k_2} \right) \notag \\
		& \quad + \left( |\cdot|^{-(N-\gamma)} * \left( Q^p \left( \frac{p}{2} \omega + \frac{p}{2} \bar{\omega} \right) \right) \right) Q^{p-1} \left( \frac{p}{2} \omega+ \frac{p-2}{2} \bar{\omega} \right) \notag \\
		& = \sum_{j_1+ j_2 + k_1+ k_2 \ge 2} c_{j_1 j_2 k_1 k_2} \left( |\cdot|^{-(N-\gamma)} * \left( Q^p \omega^{j_1} \bar{\omega}^{j_2} \right) \right) Q^{p-1} \omega^{k_1} \bar{\omega}^{k_2},
		\label{expansion of nonlinear term}
	\end{align}	
where $R(h)$ is defined in (\ref{linearized equation: nonlinear term}). Denote $\varepsilon_k= \partial_t \mathcal{V}_k^A + \mathcal{L} \mathcal{V}_k^A  - iR(\mathcal{V}_k^A )$. As for $k=1$, let $\mathcal{Z}_1^A= A \mathcal{Y}_+$, then $\mathcal{V}_1^A = e^{-e_0 t} \mathcal{Z}_1^A$ and $\omega = Q^{-1} \mathcal{V}_1^A= e^{-e_0 t} Q^{-1} \mathcal{Y}_+$. Since the decay rate of $\mathcal{Y}_+$ is faster than $Q$ (see Lemma \ref{lemma: exponential decay of ground state and its derivatives}),  $R(\mathcal{V}_1^A)$ can be expanded as in (\ref{expansion of nonlinear term}) for large $t$, then
	\begin{equation*}
		|\varepsilon_1| 
		=  \Big| \partial_t \mathcal{V}_1^A + \mathcal{L} \mathcal{V}_1^A  - iR(\mathcal{V}_1^A ) \Big| 
		= \big| iR(\mathcal{V}_1^A) \big| 
		\lesssim e^{-2e_0t}, \quad \text{ for } t \ge t_0 \gg 1.
	\end{equation*}
	Similarly, by Lemma \ref{lemma: exponential decay of ground state and its derivatives} again, it can be also easy to check that
	\begin{equation*}
		|x|^\alpha | \partial^\beta \varepsilon_1| \lesssim e^{-2e_0 t}, \quad \forall \alpha, \beta \in \mathbb{Z}^N, \; \forall  t \ge t_0(\alpha, \beta)
	\end{equation*}
    and 
    \begin{equation*}
    	\| Q^{-1} e^{\eta |x|} \partial^\alpha \mathcal{Z}_1^A \|_{L^\infty} < \infty, \quad \forall \alpha \in \mathbb{Z}_{\ge 0}^\mathbb{N} \text{ for some } 0< \eta \ll 1. 
    \end{equation*}
	
	As for general case, we use induction. In fact, we assume that $\mathcal{Z}_1^A, \mathcal{Z}_2^A, ..., \mathcal{Z}_k^A \in \mathcal{S}(\mathbb{R}^N)$ have been constructed and satisfy
	\begin{equation}
	\| Q^{-1} e^{\eta |x|} \partial^\alpha \mathcal{Z}_j \|_{L^\infty} < \infty, \; \forall j \le k, \; \forall \alpha \in \mathbb{Z}_{\ge 0} ^N,
	\label{constrction of special function: induction assumption}
	\end{equation}
	then $\varepsilon_k = \partial_t \mathcal{V}_k^A + \mathcal{L} \mathcal{V}_k^A - iR(\mathcal{V}_k^A)= O(e^{-(k+1) e_0 t})$ in $\mathcal{S}(\mathbb{R}^N)$ and $|Q^{-1} \mathcal{V}_k^A | \le \frac{1}{2}$ for large $t$. Hence by (\ref{expansion of nonlinear term}) again,
	\begin{equation*}
		-i R(\mathcal{V}_k^A)= \sum_{j=1}^{k+1} e^{-je_0 t} F_j + O(e^{-(k+2) e_0 t}) \text{ in } \mathcal{S}(\mathbb{R}^N), \; \forall t > t_0 \gg 1,
	\end{equation*}
	 for some $F_j \in \mathcal{S}(\mathbb{R}^N), \forall 1 \le j \le k+1$. Furthermore, by induction assumption (\ref{constrction of special function: induction assumption}), $F_j$ also satisfies
	 \begin{equation}
	\| Q^{-1} e^{\eta |x|} \partial^\beta F_j  \|_{L^\infty} < \infty, \; \forall \beta \in \mathbb{Z}_{\ge 0}^N, \quad \forall 1 \le j \le k+1.
	\label{construction of special function: properties on Fj}
\end{equation}
	 Thus
	\begin{equation*}
		\varepsilon_k = \sum_{j=1}^{k} e^{-je_0 t} \left( F_j + \left( \mathcal{L} - j e_0 \right) \mathcal{Z}_j^A \right)+ e^{-(k+1) e_0 t} F_{k+1} +O(e^{-(k+2) e_0 t}) = O\left(e^{-(k+1) e_0 t} \right),
	\end{equation*}
	which then implies that
	\begin{equation*}
		\varepsilon_k = e^{-(k+1) e_0t}F_{k+1} + O(e^{-(k+2) e_0 t}).
	\end{equation*}
	Let $\mathcal{Z}_{k+1}^A \triangleq - \left( \mathcal{L} - (k+1) e_0 \right)^{-1}  F_{k+1}$, then by (\ref{resolvent operator maps Schwartz class into Schwartz class}), (\ref{construction of special function: properties on Fj}) and Lemma \ref{lemma: exponential decay of ground state and its derivatives},  $Z_{k+1}^A \in \mathcal{S}(\mathbb{R}^N)$ and $Z_{j}^A$ satisfies the induction assumption (\ref{constrction of special function: induction assumption}) for $j \le k+1$. Consequently,
	\begin{align*}
		\varepsilon_{k+1}
		& = - i  \left( R(\mathcal{V}_{k+1}^A) - R(\mathcal{V}_k^A) \right)  + \varepsilon_k  + e^{-(k+1) e_0 t} \left( \mathcal{L} -(k+1) e_0 \right) \mathcal{Z}_{k+1}^A \\
		& = -i \left( R(\mathcal{V}_{k+1}^A ) - R(\mathcal{V}_k^A) \right) + O\left( e^{-(k+2)  e_0 t}  \right).
	\end{align*}
    Using (\ref{expansion of nonlinear term}) again for large $t$,
	\begin{align*}
		R(\mathcal{V}_{k+1}^A ) - R(\mathcal{V}_k^A)
		& = \sum_{j_1+ j_2 + k_1+ k_2 \ge 2} c_{j_1 j_2 k_1 k_2} \left( |\cdot|^{-(N-\gamma)} * \left( Q^p \mathcal{V}_{k+1}^{j_1} \overline{\mathcal{V}_{k+1}^A}^{j_2} \right) \right) Q^{p-1} \mathcal{V}_{k+1}^{k_1} \overline{\mathcal{V}_{k+1}^A}^{k_2} \\
		& - \sum_{j_1+ j_2 + k_1+ k_2 \ge 2} c_{j_1 j_2 k_1 k_2} \left( |\cdot|^{-(N-\gamma)} * \left( Q^p \mathcal{V}_{k}^{j_1} \overline{\mathcal{V}_{k}^A}^{j_2} \right) \right) Q^{p-1} \mathcal{V}_{k}^{k_1} \overline{\mathcal{V}_{k}^A}^{k_2}  \\
		& = O\left(e^{-(k+2) e_0 t} \right),
	\end{align*}
	 we then conclude that $\varepsilon_{k+1}= O(e^{-(k+2) e_0t})$ in $\mathcal{S}(\mathbb{R}^N)$, and the construction is completed for $p >2$.
	
	As for $p=2$, since we can't get the parallel information about $\mathcal{Y}_\pm$ as shown in Lemma \ref{lemma: exponential decay of ground state and its derivatives}$(ii)$ which is suitable for $p>2$, it is not appropriate to simply follow the argument above to construct approximation solutions when $p=2$. However, the even power $p=2$ enables us to expand $|f+g|^2$ directly and we needn't to use the expansion (\ref{expansion of nonlinear term}) anymore, which heavily relies on a priori assumption $| \omega | \triangleq |Q^{-1} h | \le \frac{1}{2}$. Therefore, it is easier to handle the case for $p=2$ and the proof is almost essentially the same as  \cite[Proposition $3.1$]{miao2015dynamics} and \cite[ Proposition $3.4$]{duyckaerts2010threshold}.
\end{proof}

Next, we construct the special solution with threshold mass-energy by a fixed point argument.
\begin{prop}[Construction of special solutions near the approximation solution] Under the Assumption \ref{assumption: nondegeneracy of the ground state}, let $A \in \mathbb{R}$, there exists $k_0 >0$ and $t_0 >0$ such that $\forall k \ge k_0$, there exists a solution $U^A$ to (\ref{Hartree equation}) such that
	\begin{equation} 
		\| \left\langle \nabla \right\rangle(U^A - U_{l(k)}^A) \|_{S((t, + \infty),L^2)} \le e^{-(k+ \frac{1}{2}) e_0 t},\quad \forall t \ge t_0,
		\label{constrction of Q(A)}
	\end{equation}
	where $l(k)= 
	\begin{cases}
		\max \Big\{k, \left( k + \frac{1}{2} \right)\frac{1}{p-2}  \Big\}, & p \not= 2, \\
		k , & p=2.
	\end{cases} $\\
	Furthermore, $U^A$ is the unique solution to (\ref{Hartree equation}) satisfying (\ref{constrction of Q(A)}) for large $t$. Finally, $U^A$ is independent of $k$ and satisfies
	\begin{equation}
		\| U^A(t) - e^{it} Q - A e^{(i-e_0) t} \mathcal{Y}_+ \|_{H^1} \lesssim  e^{-2 e_0 t}.
		\label{Q(A): the property 1}
	\end{equation}
	\label{proposition: construction of Q(A)}
\end{prop}
\begin{proof}
	Let $h$ be the solution to the linearized equation
	\begin{equation*}
		\partial_t h + \mathcal{L} h = i R(h),
	\end{equation*}
	which is equivalent to the equation
	\begin{equation*}
		i \partial_t h + \Delta h - h + S(h)=0, \text{ where } S(h)= V(h) +R(h) \text{ is defined in } (\ref{linearized equation: complex version}).
	\end{equation*}
	By (\ref{approximation solutions}), the function $w \triangleq h-\mathcal{V}_{l(k)}^A$ satisfies the equation
	\[
	i \partial_t w + \Delta w - w + S(\mathcal{V}_{l(k)}^A + w) - S(\mathcal{V}_{l(k)}^A)  = -\varepsilon_{l(k)} = O(e^{-(l(k)+1) e_0 t}),
	\]
	and $w$ is given by the equation
	\[
	w(t)=M(w)(t),
	\]
	where
	\begin{equation}
		M(w)(t) = - i \int_t^{+ \infty} e^{i(t-s) (\Delta -1)} \left[ S(\mathcal{V}_{l(k)}^A + w) - S(\mathcal{V}_{l(k)}^A) + \varepsilon_{l(k)} \right](s) ds.
		\label{definition of mapping in contraction}
	\end{equation}
	
	Next, we begin to construct the special function for  $p \in (2,3)$. As shown in Lemma \ref{lemma: Strichartz estimate on nonlinear term: difference version for p in (2,3)}, if we only use $S(I,\left\langle \nabla \right\rangle L^2)$ norm, we face a term $\| \left\langle \nabla \right\rangle (f-g) \|_{S'(L^2)}^{p-2}$ with power $p-2 < 1$, which makes contraction unable to work, and we turn to  $S( I, \dot{H}^{s_c} )$ norm for help. Precisely speaking, for $k \ge 1$ and $t_k$ fixed later, we define the working space by
	\begin{align*}
		& B^{k} \triangleq \left\lbrace  w \in X^k: \| w \|_{X^k} \le 1 \right\rbrace, \\
		& 	X^k \triangleq \Bigg\{  w:  \left\langle \nabla \right\rangle w   \in S([t_k, +\infty), L^2), w \in S([t_k, + \infty), \dot{H}^{s_c} ) \text{ and }
		\| w \|_{X^k} < \infty, \\ 
		& \qquad \qquad   \text{ where }\| w \|_{X^k} = \sup_{t \ge t_k} \|  w  \|_{S([t, + \infty), \dot{H}^{s_c}) } e^{(k+\frac{1}{2}) \frac{1}{p-2} e_0 t}+ \sup_{t \ge t_k} \| \left\langle \nabla \right\rangle w \|_{S( [t, + \infty), L^2)} e^{(k+\frac{1}{2}) e_0 t} \Bigg\} 
	\end{align*}
	equipped with the metric
	\[
	d_k(u,v)= \sup_{t \ge t_k} \|  u-v  \|_{S([t, + \infty), \dot{H}^{s_c}) } e^{(k+\frac{1}{2}) \frac{1}{p-2} e_0 t}.
	\]
	By the uniqueness of weak limitation, it is easy to see that $(B^k, d_k)$ is a complete metric space. We will show that $M: B^k \to B^k$ and $M$ is a contraction. For any fixed $w \in B^k$, by Strichartz estimate, we can see that
	\begin{align*}
		\| \left\langle \nabla \right\rangle M(w) \|_{S([t,+ \infty)), L^2)} 
		& \lesssim  \big\| \left\langle \nabla \right\rangle \left( S(\mathcal{V}_{l(k)}^A+w) - S(\mathcal{V}_{l(k)}^A) \right) \big\|_{S'([t,+ \infty)), L^2)}  + \| \left\langle \nabla \right\rangle \varepsilon_{l(k)} \|_{S' ([t,+ \infty)), L^2)}  \\
		& \lesssim   \big\| \left\langle \nabla \right\rangle \left( R(\mathcal{V}_{l(k)}^A+w) - R(\mathcal{V}_{l(k)}^A) \right) \big\|_{S'([t,+ \infty)), L^2)}  + \big\| \left\langle \nabla \right\rangle  V(w) \big\|_{S'([t,+ \infty)), L^2)} \\ 
		& \quad +  \| \left\langle \nabla \right\rangle \varepsilon_{l(k)} \|_{S'([t,+ \infty)), L^2)}  .
	\end{align*}
	As for the estimate of the nonlinear term, by (\ref{the estimates on nonliear term: H1 sense}) and Remark \ref{remark: we can bound the nonliear by L2 Strichartz norm with gradient},
	\begin{equation*}
		\Big \| \left\langle \nabla  \right\rangle (R(\mathcal{V}_{l(k)}^A + w ) - R(\mathcal{V}_{l(k)}^A ) ) \Big\|_{S'(I_n, L^2)} \lesssim e^{-(k + \frac{3}{2}) e_0 n}  \text{ for } I_n=[n,n+1],
	\end{equation*}
    where the constant $C>0$ is independent of $n \in \mathbb{Z}^+$. Sticking the intervals together, we obtain that
	\begin{align}
		& \quad\big\| \left\langle \nabla \right\rangle \left( R(\mathcal{V}_{l(k)}^A+w) - R(\mathcal{V}_{l(k)}^A) \right) \big\|_{S'([t, +\infty), L^2)}  \notag\\
		& \lesssim \sum_{n=0}^\infty \big\| \left\langle \nabla \right\rangle \left( R(\mathcal{V}_{l(k)}^A+w) - R(\mathcal{V}_{l(k)}^A) \right) \big\|_{S'([t+n, t+n+1], L^2)}  \notag\\
		& \lesssim \sum_{n=0}^\infty e^{-\left( k +\frac{3}{2}\right) e_0 (t+n)} \lesssim e^{-\left( k + \frac{3}{2} \right) e_0 t}.
		\label{contraction: nonlinear term}
	\end{align}
    \begin{rmk}
    Before we continue proceeding with the proof, we emphasize that when handling $\big \| \left\langle \nabla  \right\rangle (R(\mathcal{V}_{l(k)}^A + w ) - R(\mathcal{V}_{l(k)}^A ) ) \big\|_{S'([t,+\infty), L^2)}$, we face the term $\| w \|_{S(I_n, \dot{H}^{s_c})}^{p-2}$ again, whose low power is likely to bring us another obstacle. Precisely, if the decay of $\| w \|_{S([t,+\infty), \dot{H}^{s_c})}$ is just $e^{-(k+\frac{1}{2}) e_0 t}$, then it is hard to get enough dacay of $\big \| \left\langle \nabla  \right\rangle (R(\mathcal{V}_{l(k)}^A + w ) - R(\mathcal{V}_{l(k)}^A ) ) \big\|_{S'([t,+\infty), L^2)}$, hence $\| w \|_{S([t,+\infty),\dot{H}^{s_c})}$ should have faster decay, and here we wish its decay rate to be  $e^{-\left(k+\frac{1}{2}\right) \frac{1}{p-2} e_0 t}$ (see definition of working space $X^k$).
    \end{rmk}
    
	As for the estimate of the linear term $\| \left\langle \nabla \right\rangle V(w) \|_{S'([t, +\infty), L^2)}  $, by (\ref{Strichartz estimates on linear term}),
	\begin{equation*}
		\| \left\langle \nabla \right\rangle V(w) \|_{S'([t, t+ \tau_0 ), L^2)} 
		\le C \tau_0^\gamma \| \left\langle \nabla \right\rangle w \|_{S([t,t+ \tau_0), L^2 )} \le C \tau_0^\gamma e^{-(k+\frac{1}{2}) e_0 t},
	\end{equation*}
	where the constant $C>0$ is independent of the choice of $\tau_0 >0$, similarly, it yields
	\begin{align}
		\| \left\langle \nabla \right\rangle V(w) \|_{{S'([t, +\infty), L^2)} } 
		& \lesssim \sum_{n=0}^\infty  \| \left\langle \nabla \right\rangle V(w) \|_{S'([t + n \tau_0 , t+ n \tau_0 + \tau_0 ), L^2)}  \notag \\
		& \le C \sum_{n=0}^\infty  \tau_0^\gamma e^{-(k+\frac{1}{2}) e_0 (t+ n \tau_0)} \notag \\
		& = \frac{C \tau_0^\gamma }{1-e^{-(k+\frac{1}{2}) e_0 \tau_0}} e^{-(k+\frac{1}{2}) e_0 t}.
		\label{contraction: linear term}
	\end{align}
	Furthermore, as for the remaining term, by Proposition \ref{proposition: approximation solution},
	\begin{equation}
		\| \left\langle \nabla \right\rangle \varepsilon_{l(k)} \|_{S(t, +\infty)} \lesssim_k e^{-(l(k)+1) e_0 t}.
		\label{contraction: remaining term}
	\end{equation}
	Combining with (\ref{contraction: nonlinear term}), (\ref{contraction: linear term}) and (\ref{contraction: remaining term}),
	\begin{equation*}
		e^{(k+\frac{1}{2}) e_0 t} \| \left\langle \nabla \right\rangle M(w) \|_{S((t,+\infty),L^2)}
		\le C e^{-e_0 t} + \frac{C \tau_0^\gamma }{1-e^{-(k+\frac{1}{2}) e_0 \tau_0}} + C_k e^{-\frac{1}{2} e_0 t}.
	\end{equation*}
	Choosing $\tau_0 \ll 1$ and then letting $k \ge 1$ and $t_k>0$ sufficiently large, we can get
	\begin{equation}
		e^{(k+\frac{1}{2}) e_0 t} \| \left\langle \nabla \right\rangle M(w) \|_{S((t,+\infty),L^2)} \le \frac{1}{4}, \quad \forall w \in B^k, \; \forall t \ge t_k.
		\label{contraction 1}
	\end{equation}
	Similarly, by (\ref{the estimates on nonliear term: critical space}), we can repeat the  previous argument to check that
	\begin{align}
		\Big\| R(\mathcal{V}_{l(k)}^A + w)  -R(\mathcal{V}_{l(k)}^A)  \Big\|_{S'((t,+\infty), \dot{H}^{-s_c})} 
		&\lesssim e^{-(k + \frac{1}{2}) \frac{1}{p-2} e_0 t} e^{-e_0 t},
		\label{contraction: nonlinear term critical case}
	\end{align}
	
	\begin{align}
		\| V(w) \|_{{S'([t, +\infty), \dot{H}^{-s_c})} } 
		& \lesssim \frac{C \tau_0^\mu }{1-e^{-(k+\frac{1}{2}) \frac{1}{p-2} e_0 \tau_0}}  e^{-(k+\frac{1}{2}) \frac{1}{p-2} e_0 t}
		\label{contraction: linear term critical case}
	\end{align}
	and 
	\begin{equation}
		\| \left\langle \nabla \right\rangle \varepsilon_{l(k)} \|_{S([t, +\infty), \dot{H}^{-s_c})} \lesssim_k e^{-(l(k)+1) e_0 t},
		\label{contraction: remaining term 2 critical case}
	\end{equation}
	thus
	\begin{equation*}
		e^{(k+\frac{1}{2}) \frac{1}{p-2} e_0 t} \|  M(w) \|_{S'((t,+\infty) , \dot{H}^{-s_c})} \le C e^{-e_0 t}  +  \frac{C \tau_0^\mu }{1-e^{-(k+\frac{1}{2}) e_0 \frac{1}{p-2} \tau_0}}  + C_k e^{-\frac{1}{2} e_0 t}.
	\end{equation*}
	Similarly, after choosing $\tau_0 \ll 1$ and letting $k \ge 1$ and $t_k>0$ sufficiently large,
	\begin{equation}
		e^{(k+\frac{1}{2}) \frac{1}{p-2} e_0 t} \|  M(w) \|_{S((t,+\infty) , \dot{H}^{s_c})} \le \frac{1}{4}, \quad \forall w \in B^k, \; \forall t \ge t_k.
		\label{contraction 2}
	\end{equation}
	Therefore, together with (\ref{contraction 1}) and (\ref{contraction 2}), $\| M(w) \|_{X^k} \le \frac{1}{2} < 1, \; \forall  h \in B^k$ for $k \gg 1$ and $t_k \gg 1$.
	
	With the same argument as above, $M$ is also a contraction in $(B^k, d_k)$. Then we get a unique solution $w$ on $(t_k,+\infty)$, and $U^A \triangleq e^{it} \left(Q +  \mathcal{V}_{l(k)}^A+ w \right)$ satisfies
	\[
	\| \left\langle \nabla \right\rangle (U^A -U_{l(k)}^A) \|_{S([t,+\infty), L^2)} \le e^{-(k+\frac{1}{2}) e_0 t}, \quad \forall t \ge t_k
	\]
	and 
	\[
	\| U^A - U_{l(k)}^A \|_{S'((t,+\infty), \dot{H}^{s_c})} \le e^{-(k+ \frac{1}{2}) \frac{1}{p-2} e_0 t}, \quad  t \ge t_k.
	\]
	Moreover, by local well-posedness of (\ref{Hartree equation}), the uniqueness given by contraction and the embedding relationship
	\[
	X^{k'} \Bigg|_{(t_{k'}, + \infty)} \subset X^k \Bigg|_{(t_{k'}, + \infty)}, \quad \forall k ' \ge k,
	\]
	there exists $k_0 \gg 1$ and $t_0 \gg 1$ such that the solution $U^A$ constructed above are coincident with each other for any $k \ge k_0$. By Duhamel formula (\ref{definition of mapping in contraction}) and Strichartz estimate
	\begin{equation*} 
		\Big\| U^A - e^{it} Q - e^{it} \mathcal{V}_{l(k)}^A(t) \Big\|_{H^1} 
		= \| w(t) \|_{H^1} \lesssim e^{-(k+\frac{1}{2}) e_0 t}, \; \forall t \ge t_0.
	\end{equation*}
	Then together with the fact that
	\[
	\Big\| \mathcal{V}_{l(k)}^A(t) -  A e^{-e_0 t} \mathcal{Y}_+  \Big\|_{H^1} \lesssim e^{-2e_0 t},
	\]
	which follows from the construction of $\mathcal{V}_{l(k)}^A$ in Proposition \ref{proposition: approximation solution}, we obtain (\ref{Q(A): the property 1}) immediately. 
	
	As for $p \ge 3$ or $p=2$, since we have more regularity for $|\cdot|^{p}$ in this case, we do not need $S(I,\dot{H}^{s_c})$ anymore, and we can introduce a simpler working space as follows:
	\begin{align*}
		& B^{k} \triangleq \left\lbrace w \in X^k: \| w \|_{X^k} \le 1 \right\rbrace, \\
		& 	X^k \triangleq \Big\{  w:  \left\langle \nabla \right\rangle w  \in S([t_k, +\infty), L^2 ): 
		\| w \|_{X^k} = \sup_{t \ge t_k} \| \left\langle \nabla \right\rangle w  \|_{S( [t, + \infty), L^2)} e^{(k+\frac{1}{2}) e_0 t}  < \infty \Big\},
	\end{align*} 
	equipped with the metric
	\[
	d_k(u,v)=  \| u-v \|_{X^k}.
	\]
	With almost the same argument as above, we can also get the desired special functions.
\end{proof}

\begin{rmk} \label{rmk: construction of special functions}
	By Proposition \ref{proposition: construction of Q(A)} shown above, we can construct special solutions $Q^\pm $ on $\overline{\mathbb{R}^+}$ in Theorem \ref{theorem: construction of special functions}. Precisely speaking, if we let $A= \pm 1$ in Proposition \ref{proposition: construction of Q(A)}, then there exist solutions $U^{\pm 1}$ on $[t_0, + \infty)$ such that
	\begin{equation*}
		\| U^ {\pm 1} -e^{it} Q  \mp e^{it} e^{-e_0 t} \mathcal{Y}_+ \|_{H^1} \le C e^{-2e_0t}, \quad \forall t \ge t_0,
	\end{equation*}
    and $U^\pm$ satisfy
    \begin{equation*}
    	\| \nabla U^{ \pm 1} (t) \|_2^2 = \| \nabla Q \|_2^2 \pm 2 e^{-e_0 t} \int \nabla Q \cdot \nabla \mathcal{Y}_1 + O(e^{-2e_0 t}), \quad t \to \infty.
    \end{equation*}
   Without loss of generality, we assume $\int \nabla Q \cdot \nabla \mathcal{Y}_1 >0$, then
   \[
   \| \nabla U^{+1} (t) \|_2 > \| \nabla Q \|_2 , \text{ and } \| \nabla U^{-1} (t) \|_2 < \nabla Q \|_2, \quad \forall t \ge t_0 \text{ for } t_0 \gg 1.
   \]
   Letting
	\begin{equation}
		Q^{\pm}(x,t) \triangleq e^{-it_0 } U^{\pm 1} (x,t+t_0), \; t \ge 0,
		\label{definition of special function}
	\end{equation}
	we get two solutions satisfying
	\begin{equation*}
		M[Q^\pm]=M[Q], \; E[Q^\pm] = E[Q], \; \| \nabla Q^+(0) \|_2 > \| \nabla Q \|_2, \; \| \nabla Q^-(0) \|_2 < \| \nabla Q \|_2,
	\end{equation*}
	and
	\begin{equation}
	\| Q^{\pm} -e^{it} Q \|_{H^1} \le C e^{-e_0 t}, \quad t \ge 0.
	\label{convergence of Qpm in the positive time direction}
	\end{equation}
\end{rmk}

\begin{proof}[Complete the proof of Theorem \ref{theorem: construction of special functions}.] As for $Q^+$ constructed in Remark \ref{rmk: construction of special functions}, by Theorem \ref{theorem: convergence of u to soliton Q above the threshold} and (\ref{convergence of Qpm in the positive time direction}), $Q^+$ satisfies the properties in Theorem \ref{theorem: construction of special functions}. As for $Q^-$, it remains to check that $Q^-$ scatters in negative time direction. By contradiction, we assume that $Q^-$ does not scatter in negative time, then $t \to \overline{Q^-}(x, -t)$ globally exists in positive time but does not scatter in positive time. Repeating the argument of Theorem \ref{convergence to Q below the threshold} onto $\overline{Q^-}(x, -t)$, there exists a parameter $x(t)$ for $t \in \mathbb{R}$ such that
	\[
	\widetilde{K} = \{ Q^-(\cdot + x(t), t),  \; t \in \mathbb{R} \}
	\]
	has a compact closure in $H^1$, and $x(t)$ is bounded and $\delta(t) \to 0, \; t \to \pm \infty$. Then a simple adjustment of Lemma \ref{lem: virial type estimates on delta} yields
	\begin{equation*}
		\int_\sigma^\tau \delta(t) dt \le C \left[ 1+ \sup_{\sigma \le t \le \tau} |x(t)| \right] \left(  \delta(\sigma) + \delta (\tau)\right) \le C \left( \delta(\sigma) + \delta(\tau) \right).
	\end{equation*}
	Letting $\sigma \to -\infty$ and $\tau \to +\infty$, we get that $\int_{\mathbb{R}} \delta(t) dt=0$, thus, $\delta(t) =0 $ for all $t$, which contradicts to the fact that $\| \nabla Q^-(0) \|_2 < \| \nabla Q \|_2$. 
\end{proof}

\subsection{Rigidity}
Before we complete the Theorem \ref{theorem: classification of threshold solutions at zero momentum case}, we still have to deal with a problem involved in uniqueness.
\begin{prop}[Rigidity]
	Under the Assumtion \ref{assumption: uniqueness of the ground state} and Assumption \ref{assumption: nondegeneracy of the ground state}, let $u$ be the solution to (\ref{Hartree equation}) on $[t_0, +\infty)$ such that $M[u]=M[Q], E[u] =E[Q]$ and
	\begin{equation}
		\| u -e^{it} Q \|_{H^1} \le C e^{-ct} , \quad \forall t \ge t_0
	\end{equation}
	for some $c,C>0$. Then there exists $A \in \mathbb{R}$ such that $u=U^A$, where $U^A$ is defined in Proposition \ref{proposition: construction of Q(A)}.
	\label{proposition with help for uniqueness}
\end{prop}
\begin{proof}
	Letting $u=e^{it} (Q+h)$, $h$ satisfies
	\[
	\partial_t h + \mathcal{L} h =R(h)
	\]
	and $\| h \|_{H^1(\mathbb{R}^N)} \le Ce ^{-ct}, \; \forall t \ge t_0$.  By exponential Strichartz estimate shown in Lemma \ref{lemma: exponential Strichartz estimates},
	\begin{equation*}
		\| h(t) \|_{H^1} + \| \left\langle \nabla \right\rangle h \|_{S((t,+\infty),L^2)} \le Ce^{-ct}, \; \forall t \ge t_0.
	\end{equation*}
	\textit{Step 1.} Improve the decay to
	\begin{equation}
		\| h(t) \|_{H^1} + \| \left\langle \nabla \right\rangle h \|_{S((t,+\infty),L^2)} \le Ce^{-e_0^-t}, \; \forall t \ge t_0.
		\label{exponential decay of v: sharp rate}
	\end{equation}
	By (\ref{Strichartz estimates on nonlinear term}),
	\[
	\| \left\langle \nabla \right\rangle R(h) \|_{S'((t,+\infty),L^2)} 
	\lesssim \sum_{n=0}^\infty \| \left\langle \nabla \right\rangle R(h) \|_{S'((t+n , t+ n+1),L^2)} \lesssim
	\begin{cases}
		e^{-(p-1) c t},  & p \in (2,3), \\
		e^{-2ct}, & p =2 \text{ or } p \ge 3.
	\end{cases}
	\]
	Then $h$ and $R$ satisfy the conditions of Lemma \ref{lemma: improve the decay of solution to linearized equation} with 
	\[c_0 =c, \quad c_1=\tau c, \text{ where } \tau \triangleq \begin{cases}
		p-1>1, & p \in (2,3), \\
		2, & p=2 \text{ or } p \ge 3.
	\end{cases}
	\]
	If $ \tau c > e_0$, then (\ref{exponential decay of v: sharp rate}) is completed. If not, we get that
	\begin{equation*}
		\| h(t) \|_{H^1} + \| \left\langle \nabla \right\rangle h \|_{S((t,+\infty),L^2)} \le Ce^{-\tau c^- t}, \; \forall t \ge t_0,
	\end{equation*}
	and then (\ref{exponential decay of v: sharp rate}) follows from an iteration argument. And as a by-product, by Lemma \ref{lemma: improve the decay of solution to linearized equation} again, there exists $A \in \mathbb{R}$ such that
	\begin{equation*}
		\Big\| h(t) - A^{-e_0 t} \mathcal{Y_+} \Big\|_{H^1} + \Big\| \left\langle \nabla \right\rangle \left( h(t)- Ae^{-e_0 t} \mathcal{Y}_+ \right) \Big\|_{S((t, + \infty),L^2)} \lesssim e^{-\tau e_0^- t},\; \forall t \ge t_0.
	\end{equation*}
	\textit{Step 2.} As for $U^A$ constructed in Proposition \ref{proposition: construction of Q(A)}, where $A$ is as in \textit{Step 1}, if $U^A \triangleq e^{it} (Q + h^A)$, then we claim
	\begin{equation}
		\| h^A(t) - h(t) \|_{H^1} +  \Big\| \left\langle \nabla \right\rangle (h^A-h) \Big\|_{S((t,+ \infty),L^2)} \lesssim e^{-\gamma t}, \; \forall t \ge t_0, \; \forall \gamma > 0 .
		\label{exponential decreasing of hA-h with arbitrary order}
	\end{equation}
	In fact, by Remark \ref{remark: exponential Strichartz estimate} and (\ref{Q(A): the property 1}), $h^A$ satisfies
	\begin{equation*}
		\Big\| h^A(t) - A^{-e_0 t} \mathcal{Y_+} \Big\|_{H^1} + \Big\| \left\langle \nabla \right\rangle \left( h^A(t)- Ae^{-e_0 t} \mathcal{Y}_+ \right) \Big\|_{S((t, + \infty),L^2)} \lesssim e^{-2 e_0^- t},\; \forall t \ge t_0,
	\end{equation*}
   then $h^A - h$ satisfies the equation
	\begin{equation*}
		\partial_t \left( h^A - h \right) + \mathcal{L} \left( h^A- h \right) =R(h^A) - R(h),
	\end{equation*}
	and
	\begin{equation*}
		\| h^A - h \|_{H^1} +  \Big\| \left\langle \nabla \right\rangle (h^A-h) \Big\|_{S((t,+ \infty),L^2)} \lesssim e^{- \tau e_0^- t}, \; \forall t \ge t_0.
	\end{equation*}
	If $p \in (2,3)$, then Lemma \ref{lemma: Strichartz estimate on nonlinear term: difference version for p in (2,3)} and Remark \ref{remark: we can bound the nonliear by L2 Strichartz norm with gradient},
	\begin{equation*}
		\| \left\langle \nabla \right\rangle (R(h^A) - R(h)) \|_{S'((n,n+1) ,L^2)}   \lesssim e^{-\lambda e_0^- n} \text{ for } \lambda = (p-2)(p-1) +1 > p-1 =\tau >1,
	\end{equation*}
	which means
	\begin{align*}
		\| \left\langle \nabla \right\rangle (R(h^A)- R(h)) \|_{S'((t,+ \infty),L^2)}
		& \lesssim \sum_{n=0}^\infty  \| \left\langle \nabla \right\rangle (R(h^A)- R(h)) \|_{S'((t+n,t+ n+1),L^2)} \\
		& \lesssim \sum_{n=0}^\infty e^{- \lambda e_0^-(t+n)}  \lesssim e^{-\lambda e_0^- t}, \quad \forall t \ge t_0,
	\end{align*}
	Using Lemma \ref{lemma: improve the decay of solution to linearized equation} with $c_0 = \tau e_0^- >e_0$ and $c_1 = \lambda e_0^- $, and repeating the process over and over again, we immediately get (\ref{exponential decreasing of hA-h with arbitrary order}).
	
	When it comes to $p=2$ or $p >3$, (\ref{exponential decreasing of hA-h with arbitrary order}) also follows from the analogous argument.
	
	\noindent \textit{Step 3.} For any $k \in \mathbb{Z}^+$, bacause $\gamma$ is arbitrarily chosen, letting $\gamma= (k_0+1) e_0$ in (\ref{exponential decreasing of hA-h with arbitrary order}),  we get
	\[
	\Big\| \left\langle \nabla \right\rangle \left( u - e^{it} Q - e^{it} \mathcal{V}_{l(k_0)} \right) \Big\|_{S((t,+\infty), L^2)} \le e^{-(k_0+\frac{1}{2}) e_0 t} \text{ for large } t,
	\]
	then $u=U^A$ follows from the uniqueness shown in Proposition \ref{proposition: construction of Q(A)}.
\end{proof}
\subsection{Complete the proof of Theorem \ref{theorem: classification of threshold solutions at zero momentum case}}
Finally, we are devoted to the proof of Theorem \ref{theorem: classification of threshold solutions at zero momentum case}. 

As for $U^A$ constructed in Proposition \ref{proposition: construction of Q(A)}, if $A\not =0$, then $U^A= Q^+$ (if $A > 0 $) or $Q^-$ (if $A < 0$) up to some symmetries. In fact, by the construction of $Q^\pm$,
\begin{equation*}
	\Big\|  Q^\pm (t) -  e^{it} Q \mp e^{-e_0 t_0} e^{(i-e_0) t} \mathcal{Y}_+ \| _{H^1} = O(e^{-2e_0 t}), \quad t \ge 0.  
\end{equation*}
thus $\forall \; t_1 \in \mathbb{R}$,
\begin{equation*}
	e^{-it_1} Q^\pm (t+t_1, x)
	= e^{it}Q \pm e^{-e_0(t_0 +t_1) } e^{-e_0 t} e^{it} \mathcal{Y}_+ + O(e^{-2e_0 t}) \text{ in } H^1.
\end{equation*}
When $A>0$, let $t_1=-t_0- e_0^{-1} \log A$, then
\begin{equation*}
	e^{-it_1} Q^+ (t+t_1, x)= e^{it} Q + A e^{it} e^{-e_0 t} \mathcal{Y}_+ +  O(e^{-2 e_0 t}) \text{ in } H^1.
\end{equation*}
Consequently, by the rigidity shown in Proposition \ref{proposition with help for uniqueness}, there exists $\widetilde{A} \in \mathbb{R}$ such that $ e^{-it_1} Q^+ (t+t_1, x) = U^{\widetilde{A}}$, and by (\ref{Q(A): the property 1}),
\begin{equation*}
	e^{it} Q + A e^{it} e^{-e_0 t} \mathcal{Y}_+ +  O(e^{-2 e_0 t})
	=e^{-it_1} Q^+ (t+t_1, x)
	=e^{it} Q + \widetilde{A} e^{it} e^{-e_0 t} \mathcal{Y}_+ +  O(e^{-2 e_0 t}) 
	\text{ in } H^1.
\end{equation*}
Comparing the first term with the last term, we obtain that $A=\widetilde{A}$. Consequently, $U^A(t,x)= U^{\widetilde{A}} (t,x)= e^{-it_1} Q^+ (t+ t_1, x)$. The case for $A<0$ also follows from the same argument.

Let $u$ be a solution to (\ref{Hartree equation}) satisfying the conditions in Theorem \ref{theorem: classification of threshold solutions at zero momentum case}, then $M[u]^\frac{1-s_c}{s_c}E[u]=M[Q]^\frac{1-s_c}{s_c}E[Q]$. Rescaling $u$ if necessary, we may assume
\begin{equation*}
	M[u]=M[Q], \quad E[u]=E[Q].
\end{equation*}
The classification of threshold solutions to (\ref{Hartree equation}) then immediately follows from Lemma \ref{lemma: behavior of gradient of solution}, Theorem \ref{theorem: convergence of u to soliton Q above the threshold}, Theorem \ref{convergence to Q below the threshold}, Proposition \ref{proposition with help for uniqueness} and the argument above.

\appendix
\section{Ground state}
\label{Appendix: ground state}
As for the existence of ground state, we can find it by minimizing a Weinstein-type functional. More precisely, for $p \ge 2$ and $\gamma \in (0,N)$, by Hardy-Littlewood-Sobolev inequality, H\"{o}lder inequality and Sobolev embedding, one concludes that
\begin{equation}
	\int_{\mathbb{R}^N} \left( |\cdot|^{-(N-\gamma)} * |u|^p \right) |u|^p dx \le C \| \nabla u \|_{L^2(\mathbb{R}^N)}^{Np-(N+\gamma)} \| u \|_{L^2(\mathbb{R}^N)}^{N+\gamma -(N-2)p}.
	\label{Hardy-Littlewood-Sobolev inequality for 5D Hartree equation}
\end{equation}
We want to determine the sharp constant for (\ref{Hardy-Littlewood-Sobolev inequality for 5D Hartree equation}), which motivates us to minimize the following Weinstein-type functional
\begin{equation}
	Z[u]=  \frac{\| u \|_2^{(N+\gamma)-(N-2)p} \| \nabla u \|_2^{Np-(N+\gamma)}}{\int \left( |\cdot|^{-(N-\gamma)} * |u|^p \right) |u|^p dx}.
	\label{Weinstein-type functional: definition}
\end{equation}
By rearrangement inequalities \cite[Theorem $3.4$, Theorem $3.7$]{lieb2001analysis} and  Strauss Lemma, the infimum of (\ref{Weinstein-type functional: definition}) can be attained at a radial and strictly positive function. The readers can refer to Lemma $4.1$ in \cite{arora2019global} for details. And the result is stated as below:
\begin{prop}[\cite{arora2019global}. existence of ground state $Q$]
	The minimizer of the problem 
	\begin{equation}
		C_{GN}^{-1} \triangleq	\inf_{0 \not= u \in H^1(\mathbb{R}^N)} Z[u]
		\label{minimizing problem}
	\end{equation}
	can be attained at a radial, positive function $Q$. Moreover, after scaling if necessary, $Q$ satisfies the equation (\ref{ground state: equation}).
	\label{proposition: existence of ground state}
\end{prop}

\begin{lem}[Pohozhaev identity]
	As for $Q$ given above, we have
	\begin{equation}
		\int \left( |\cdot|^{-(N-\gamma)} * Q^p   \right) Q^p dx = \frac{2p}{\gamma+2p -N(p-1)} \| Q\|_2^2= \frac{2p}{N(p-1)-\gamma} \| \nabla Q \|_2^2.
		\label{Pohozhaev equality}
	\end{equation}
\end{lem}
\begin{proof}
	Multiplying $x \cdot \nabla Q$ with both sides of
	\[
	-\Delta Q + Q -\left( |\cdot|^{-(N-\gamma)}*Q^p \right)Q^{p-1} =0,
	\]
	integrating on $\mathbb{R}^N$, by some careful calculation, we have
	\[
	-\frac{N-2}{2} \| \nabla Q \|_2^2 -\frac{N}{2} \| Q\|_2^2 +\frac{N+\gamma}{2p} \int \left( |\cdot|^{-(N-\gamma)}* Q^p \right)Q^{p} dx=0.
	\]
	Moreover, if we multiply $Q$ with both sides and integrate on $\mathbb{R}^N$, then
	\[
	\| \nabla Q \|_2^2 + \| Q \|_2^2 - \int \left( |\cdot|^{-(N-\gamma)}* Q^p \right)Q^p dx =0.
	\]
	Hence (\ref{Pohozhaev equality}) immediately follows from the two equations above.
\end{proof}

As for the behavior of ground state at infinity, V.Moroz and J. Van Schaftingen \cite{moroz2013groundstates}  shows the decay rate at infinity as follows:
\begin{prop}[\cite{moroz2013groundstates}. Exponential decay of ground state]
	Let $N \in \mathbb{N}_*, \gamma \in (0,N)$ and  $p \in (1, \infty)$. Assume that $\frac{N-2}{N+\gamma} < \frac{1}{p} < \frac{N}{N+\gamma}$. Let $Q \in W^{1,2} (\mathbb{R}^N)$ be a nonnegative ground state of
	\[
	-\Delta Q + Q = \left( |\cdot|^{-(N-\gamma)} * |Q|^p \right) |Q|^{p-2} Q, \quad \text{in } \mathbb{R}^N.
	\]
	If $p>2$,
	\begin{equation*}
		\lim_{|x| \to \infty} Q(x) |x|^{\frac{N-1}{2}} e^{|x|} \in (0,\infty);
	\end{equation*}
	If $p=2$,
	\begin{equation*}
		\lim_{|x| \to \infty} Q(x) |x|^{-\frac{N-1}{2}} \exp\int_\nu^{|x|}\sqrt{1- \frac{\nu^{N-\gamma}}{s^{N-\gamma}}} ds \in (0,\infty).
	\end{equation*}
	\label{proposition: exponential decay of ground state}
\end{prop}

Next, we study the decay of any order of derivative of ground state $Q$, and we need the following simple auxiliary lemma. Using the formula $\partial_r = \frac{x}{|x|} \cdot \nabla$ and induction method, we obtain the following estimate:
\begin{lem}
	If $f\in W^{k,1}(\mathbb{R}^N)$ and $g \in L^1(\mathbb{R}^N)$ are two radial functions, then $(f*g) (x)$ are also radial, and
	\begin{equation}
		\Big|\partial_r^k \left( (f*g)(x)  \right) \Big|  \lesssim  \sum_{j=1}^k \frac{\left( |D^j f| * |g| \right) (x)}{|x|^{k-j}}, \quad \forall x \in \mathbb{R}^N.
		\label{derivative of convolution in radial direction}
	\end{equation}
\end{lem}

\begin{lem}
	As for $Q$ given above. If $p > 2$, then for any muti-index $\alpha \in \mathbb{Z}_{\ge 0}^N$, the following estimates hold: \\
	$(i)$ $ \Big\| Q^{-1} \partial^\alpha Q \Big\|_{L^\infty} < \infty$, \\
	$(ii)$ $ \Big\| Q^{-1} e^{\eta |x|}  \partial^\alpha \mathcal{Y}_\pm \Big\|_{L^\infty} <\infty$ for some $0 < \eta \ll 1$,\\
	$(iii)$ $\Big\| Q^{-1} e^{\eta |x|} \partial^\alpha \left[ \left( \mathcal{L} -\lambda \right)^{-1} f \right] \Big\|_{L^\infty} < + \infty$ for every $\lambda \in \mathbb{R} \setminus \sigma(\mathcal{L})$ and every $ f \in \mathcal{S}(\mathbb{R}^N)$ such that $\Big\| Q^{-1} e^{\eta |x|} \partial^\beta f \Big\|_{L^\infty} <+ \infty$ for some $0 < \eta < \Re \sqrt{1+ \lambda i}$ and any $\beta \in \mathbb{Z}_{\ge 0}^N$.
	\label{lemma: exponential decay of ground state and its derivatives}
\end{lem}
\begin{proof}
	As for $(i)$, we do not know whether $Q$ is in Schwartz class from the published papers, hence we can not use the method in \cite[Corollary $3.8$]{campos2020threshold} directly, which heavily relies on the a priori upper bound of $\partial^\alpha Q$. Instead, we use the  comparison theorem together with the fact that $Q$ is radially symmetric. Taking gradient onto both sides of (\ref{ground state: equation}), since $Q$ is radial, we obtain that
	\[
	\left( -\Delta + 1 + \frac{N-1}{r^2} - \left( p-1 \right)  \left( |\cdot|^{N-\gamma} * Q^p \right)Q^{p-2} \right) \left( -\partial_r Q \right) = - \partial_r \left[ |\cdot|^{-(N-\gamma)} * Q^p  \right] Q^{p-1}.
	\]
	Note that $Q$ is a positive, radial decreasing function on $\mathbb{R}^N$, by an estimate on $\partial_r \left[ |\cdot|^{-(N-\gamma)} * Q^p \right]$ (see \cite[Proposition $2.2$]{zexing2022priori}),  $|\cdot|^{-(N-\gamma)} * Q^p$ is strictly radially decreasing, which means that
	\[
	\left( -\Delta + 1 + \frac{N-1}{r^2} - \left( p-1 \right)  \left( |\cdot|^{N-\gamma} * Q^p \right)Q^{p-2} \right) \left( -\partial_r Q \right) \ge 0.
	\]
	Observing that
	\begin{align*}
		\left(-\Delta +1 + \frac{N-1}{r^2} - (p-1) \left( |\cdot|^{-(N-\gamma)} * Q^p \right) Q^{p-2} \right) (r^{-N} e^{-r})  \le 0, \quad \forall r \ge R_0 \gg 1,
	\end{align*}
	by the classical elliptic comparison theorem, we find that
	\begin{equation*}
		-(\partial_r Q) (r)\ge \frac{r^{-N} e^{-r}}{R^{-N} e^{-R}} (-\partial_r Q)(R), \quad \forall r \ge R \ge R_0,
	\end{equation*}
	which implies
	\begin{equation*}
		Q(R)= \int_R^\infty (-\partial_r Q)(s) ds \ge \int_R^\infty  \frac{s^{-N} e^{-s}}{R^{-N} e^{-R}} (-\partial_r Q ) (R) ds \gtrsim (-\partial_r Q)(R), \quad \forall R \ge R_0,
	\end{equation*}
	and it concludes $(i)$ for $|\alpha| =1$. As for $|\alpha|=2$, it immediately follows from the equation
	\begin{equation}
		-\partial_r^2  Q - \frac{N-1}{r} \partial_r Q  + Q - \left( |\cdot|^{-(N-\gamma)} * Q^p \right) Q^{p-1} =0.
		\label{ground state: equation, radial case}
	\end{equation}
	And (\ref{ground state: equation, radial case}) further implies that the higher order derivatives of $Q$ can be controlled by the lower ones.  By this intuition, as for more general case, we assume that $| \partial_r^l Q(r) | \lesssim Q(r)$ holds for $ l \le k+1$, and we will prove the validity of the estimate for $l=k+2$. In fact, taking $\partial_r^k$ onto both sides of (\ref{ground state: equation, radial case}), 
	\begin{equation}
		-\partial_r^{k+2} Q - \partial_r^k \left( \frac{N-1}{r} \partial_r Q \right) + \partial_r^k Q - \partial_r^k \left[ \left( |\cdot|^{-(N-\gamma)} * Q^p \right) Q^{p-1} \right] =0,
		\label{ground state: k-th derivative radial case}
	\end{equation}
	which, by induction assumption and (\ref{derivative of convolution in radial direction}), implies that
	\[
	\Bigg| \partial_r^k \left( \frac{1}{r} \partial_r Q   \right) \Bigg| \lesssim \sum_{j=0}^k  r^{-(k-j+1)} | \partial_r^{j+1} Q | \lesssim Q(r), \quad \forall \; r  \ge 1
	\]
	and
	\begin{align*}
		\Big| \partial_r^k \left[ \left( |\cdot|^{-(N-\gamma)} * Q^p \right) Q^{p-1} \right] \Big| 
		& = \Big| \sum_{j=0}^k c_j \partial_r^j \left( |\cdot|^{-(N-\gamma)} * Q^p \right) \partial_r^{k-j} \left( Q^{p-1} \right) \Big| \\
		& \lesssim \sum_{j=0}^k \sum_{i=1}^j \frac{|\cdot|^{-(N-\gamma)} * |D^i (Q^p)|}{|x|^{j-i}}  \Big| \partial_r^{k-j} \left( Q^{p-1} \right) \Big| \lesssim Q(r), \quad \forall \; r \ge 1,
	\end{align*}
	hence $| \partial_r^{k+2} Q (r)| < Q(r) , \; \forall \; r \ge 1$. Then together with the fact that $Q \in C^\infty$ (see \cite[Theorem $3$]{moroz2013groundstates}) and $Q$ is strictly positive, we obtain that $|\partial_r^{k+2} Q| \lesssim Q$ on $\mathbb{R}^N$. Thus we have completed the proof of $(i)$. 
	
	As for $(ii)$ and $(iii)$, the proof is almost the same as \cite[Corollary $3.8$]{campos2020threshold}. 
\end{proof}

\begin{rmk}
	As for the decay of any order derivative of the ground state $Q$ for $p=2$, after repeating the same argument above, we can also get the desired conclusion, i.e.  $\| Q^{-1} \partial^\alpha Q \|_{L^\infty} < \infty$, $\forall \alpha \in \mathbb{Z}_{\ge 0}^N$. However, since $ \left( |\cdot|^{-(N-\gamma)} * Q^2 \right) $ only decays with rate $|x|^{-(N-\gamma)}$ at infinity, we cannot follow the argument in \cite{campos2020threshold} to extend $(ii)$ and $(iii)$ to the case for $p=2$ temporarily.
 \end{rmk}

\section{The estimates on nonlinear term $R(h)$}
\label{Appendix: Strichartz estimates}
\begin{lem}
	If $p > 2$, the function
	\begin{equation}
		J(z)= |1+z|^p - \left( 1+ \frac{p}{2} z + \frac{p}{2} \bar{z} \right), \quad z \in \mathbb{C}
		\label{J(z): definition}
	\end{equation}
	satisfies the following properties: \\
	$(i)$ $J(z) \in C^1(\mathbb{C})$; \\
	$(ii)$ $J_z(0)= J_{\bar{z}} (0)= 0$; \\
	$(iii)$ $|J(z_1)- J(z_2) | \lesssim \left( |z_1|+ |z_2| + |z_1|^{p-1} + |z_2|^{p-1} \right) |z_1 -z_2|, \; \forall z_1, z_2 \in \mathbb{C}$; \\
	$(iv)$ $|J_z(z_1) -J_z(z_2)|, |J_{\bar{z}} (z_1) -J_{\bar{z}} (z_2) | \lesssim \left( 1+ |z_1|^{p-2} + |z_2|^{p-2} \right)|z_1-z_2|, \; \forall z_1, z_2 \in \mathbb{C}$.
	
	\noindent Similarly, the function
	\begin{equation}
		K(z)= |1+z|^{p-2} (1+z) - \left( 1+ \frac{p}{2} z + \frac{p-2}{2} \bar{z} \right), \quad z \in \mathbb{C}
		\label{K(z): definition}
	\end{equation}
	satisfies the following properties: \\
	$(i)$ $K(z) \in C^1(\mathbb{C})$; \\
	$(ii)$ $K_z(0)= K_{\bar{z}} (0)= 0$; \\
	$(iii)$ $|K(z_1)- K(z_2)| \lesssim \begin{cases}
		\left( |z_1| +|z_2| + |z_1|^{p-2} + |z_2|^{p-2} \right) |z_1- z_2|, & \text{ if } p \ge 3, \\
		\left( |z_1|^{p-2} + |z_2|^{p-2} \right) |z_1-z_2|, & \text{ if } p \in (2,3),
	\end{cases} $ \\
	$(iv)$ $|K_z(z_1)-K_z(z_2)|, |K_{\bar{z}} (z_1)-K_{\bar{z}} (z_2)| \lesssim \begin{cases}
		\left( 1+ |z_1|^{p-3} + |z_2|^{p-3} \right)|z_1 -z_2|, & \text{ if } p \ge 3, \\
		|z_1-z_2|^{p-2}, & \text{ if } p \in (2,3).
	\end{cases}$
	\label{lemma: estimate on J(z) and K(z)}
\end{lem}
\begin{proof}
The results are easily to be checked with the help of
	\begin{equation*}
		J(z_1) - J(z_2)
		= \int_0^1 J_z (s z_1 + (1-s)z_2)  \left(z_1 -z_2\right) + J_{\bar{z}}\left( sz_1 + (1-s)z_2 \right) \left( \overline{z_1 - z_2}\right) ds.
	\end{equation*}
    We omit the details here and just emphasize that when we consider the estimates on $K(z)$ and its derivative for $p \in (2,3)$, the mean value theorem is invalid and we use an important fact that the function $z \mapsto |z|^\alpha$ is H\"{o}lder continuous with order $\alpha$ for $\alpha \in (0,1)$.
\end{proof}

\begin{lem}
	For any $u_i \in S\left(I , \left\langle \nabla \right\rangle L^2 \right), 1 \le i \le 5$, we have
	\begin{align}
		& \quad \| \left( |\cdot|^{-(N-\gamma)} * (|u_1|^{p-2} u_2 u_3) \right) |u_4|^{p-2} |u_5| \|_{S'(L^2)} \notag\\
		&\lesssim \min_{j=2,3,5} \left\lbrace \| u_1 \|_{S(\dot{H}^{s_c} )}^{p-2}  \| u_4 \|_{S(\dot{H}^{s_c} )}^{p-2}  \| u_j \|_{S(L^2)} \prod_{i ,k \not \in \{  1,4,j \}} \| u_i \|_{S(\dot{H}^{s_c} )} \| u_k\|_{S(\dot{H}^{s_c} )} \right\rbrace 
		\label{estimates of nonliear term: dual L2 space}
	\end{align}
	and
	\begin{align}
		& \quad \| \left( |\cdot|^{-(N-\gamma)} * (|u_1|^{p-2} u_2 u_3) \right) |u_4|^{p-2} |u_5| \|_{S'(\dot{H}^{-s_c})} \notag\\
		&\lesssim  \| u_1 \|_{S(\dot{H}^{s_c} )}^{p-2}  \| u_2 \|_{S(\dot{H}^{s_c})} \| u_3 \|_{S(\dot{H}^{s_c} )}  \| u_4 \|_{S(\dot{H}^{s_c} )}^{p-2}  \| u_5 \|_{S(\dot{H}^{s_c} )}.
		\label{estimates of nonlinear term: dual Hsc space}
	\end{align}
	\label{lemma: Strichartz estimate on nonlinear term}
\end{lem}
\begin{proof}
	As a matter of fact, by H\"{o}lder inequality and Hardy-Littlewood-Sobolev inequality and Definition \ref{rmk: definition of Strichartz pair}, it is easily to check that
	\begin{align*}
		& \quad \| \left( |\cdot|^{-(N-\gamma)} * (|u_1|^{p-2} u_2 u_3) \right) |u_4|^{p-2} |u_5| \|_{L_t^{q_1'} L_x^{r_1'}} \notag\\
		&\lesssim \min_{j=2,3,5} \left\lbrace \| u_1 \|_{L_t^{q_2}L_x^{r_1}}^{p-2}  \| u_4 \|_{L_t^{q_2} L_x^{r_1}}^{p-2}  \| u_j \|_{L_t^{q_1} L_x^{r_1}} \prod_{i ,k \not \in \{  1,4,j \}} \| u_i \|_{L_t^{q_2}L_x^{r_1}} \| u_k\|_{L_t^{q_2}L_x^{r_1}} \right\rbrace,
	\end{align*}
	which is (\ref{estimates of nonliear term: dual L2 space}). Similarly, we can also get  (\ref{estimates of nonlinear term: dual Hsc space}).
\end{proof}

With the help of Lemma \ref{lemma: estimate on J(z) and K(z)} and Lemma \ref{lemma: Strichartz estimate on nonlinear term}, we are going to estimate the Strichartz norm of $ \left\langle \nabla  \right\rangle \left( R(f)- R(g) \right)$.
\begin{lem} Let $I$ be a bounded interval on $\mathbb{R}$. If  $p \in (2,3)$, then for any $f, g \in S(I, \left\langle \nabla \right\rangle L^2)$,
	\begin{align}
		& \quad \Big \| \left\langle \nabla  \right\rangle (R(f) - R(g) ) \Big\|_{S'(I, L^2)} \notag \\
		& \lesssim \| \left\langle \nabla \right\rangle (f-g) \|_{S(I,L^2)} \left(  \| f \|_{S(I, \dot{H} ^{s_c})}^{2p-2} +  \| g \|_{S(I, \dot{H} ^{s_c})}^{2p-2} +  \| Q \|_{S(I, \dot{H} ^{s_c})}^{2p-3} \left(  \| f \|_{S(I, \dot{H} ^{s_c})} +  \| g \|_{S(I, \dot{H} ^{s_c})} \right)\right) \notag   \\
		& +  \| f-g \|_{S(I, \dot{H} ^{s_c})} \left( \| \left\langle \nabla \right\rangle f \|_{S(I,L^2)} + \| \left\langle \nabla \right\rangle g \|_{S(I,L^2)}  \right)  \left(  \| f \|_{S(I, \dot{H} ^{s_c})}^{2p-3} +  \| g \|_{S(I, \dot{H} ^{s_c})}^{2p-3} +  \| Q \|_{S(I, \dot{H} ^{s_c})}^{2p-3} \right) \notag  \\
		& +  \| f-g \|_{S(I, \dot{H} ^{s_c})}^{p-2}  \left( \| \left\langle \nabla \right\rangle f \|_{S(I,L^2)} + \| \left\langle \nabla \right\rangle g \|_{S(I,L^2)}  \right)  \left(  \| f \|_{S(I, \dot{H} ^{s_c})}^{p} +  \| g \|_{S(I, \dot{H} ^{s_c})}^{p} +  \| Q \|_{S(I, \dot{H} ^{s_c})}^{p}  \right)
		\label{the estimates on nonliear term: H1 sense}
	\end{align}
	and
	\begin{align}
		& \quad \| R(f)- R(g) \|_{S'(I , H^{-s_c})} \notag \\
		& \lesssim \| f -g \|_{S(I, \dot{H} ^{s_c})} \left(  \| f \|_{S(I, \dot{H} ^{s_c})}^{2p-2} +  \| g \|_{S(I, \dot{H} ^{s_c})}^{2p-2} +  \| Q \|_{S(I, \dot{H} ^{s_c})}^{2p-3} \left(  \| f \|_{S(I, \dot{H} ^{s_c})}  +  \| g \|_{S(I, \dot{H} ^{s_c})}    \right) \right).
		\label{the estimates on nonliear term: critical space}
	\end{align}
	\label{lemma: Strichartz estimate on nonlinear term: difference version for p in (2,3)}
\end{lem}
\begin{proof}
	As a matter of fact, by (\ref{linearized equation: nonlinear term}), we obtain that
	\begin{align*}
		R(f) -R(g)
		& = \left( |\cdot|^{-(N-\gamma)} * \left( Q^p \left( J(Q^{-1} f) - J(Q^{-1} g) \right) \right) \right) |Q+f|^{p-2} (Q+f)  \\
		& + \left( |\cdot|^{-(N-\gamma)} * \left( Q^p J(Q^{-1} g) \right) \right) \left( |Q+ f| ^{p-2} (Q+  f) - |Q + g|^{p-2} (Q+ g) \right) \\
		& + \left( |\cdot|^{-(N-\gamma)} * \left( \frac{p}{2} Q^{p-1} \left(f-g\right) + \frac{p}{2} Q^{p-1} \left( \overline{f-g} \right) \right) \right) Q^{p-1} K(Q^{-1} f) \\
		& + \left( |\cdot|^{-(N-\gamma)} * \left( Q^p + \frac{p}{2} Q^{p-1} g + \frac{p}{2} Q^{p-1} \bar{g} \right) \right) Q^{p-1} \left( K(Q^{-1} f) - K(Q^{-1} g) \right) \\
		& + \left( |\cdot|^{-(N-\gamma)} * \left( \frac{p}{2} Q^{p-1} (f-g) + \frac{p}{2} Q^{p-1} (\overline{f-g}) \right) \right) \left( \frac{p}{2} Q^{p-2} f + \frac{p-2}{2} Q^{p-2} \bar{f} \right) \\
		& + \left( |\cdot|^{-(N-\gamma)} *\left(  \frac{p}{2} Q^{p-1} g + \frac{p}{2} Q^{p-1} \bar{g} \right) \right) \left( \frac{p}{2} Q^{p-2}\left( f-g \right) + \frac{p-2}{2} Q^{p-2} ( \overline{f-g})  \right),
	\end{align*}
	where $J,K$ are two special functions defined in Lemma \ref{lemma: estimate on J(z) and K(z)}, then we immediately get (\ref{the estimates on nonliear term: critical space}) and
	\begin{align*}
		& \quad  \| R(f)- R(g) \|_{S'(I, L^2)} \notag \\
		& \lesssim \| f-g \|_{S(I,L^2)} \left(  \| f \|_{S(I, \dot{H} ^{s_c})}^{2p-2} +  \| g \|_{S(I, \dot{H} ^{s_c})}^{2p-2} +  \| Q \|_{S(I, \dot{H} ^{s_c})}^{2p-3} \left(  \| f \|_{S(I, \dot{H} ^{s_c})} +  \| g \|_{S(I, \dot{H} ^{s_c})} \right)\right).
	\end{align*} 
	When we take gradient onto  $R(f)- R(g)$, bearing in mind that any term with a gradient can only be bounded by its $S(I, \left\langle \nabla \right\rangle  L^2)$ norm, all the terms can be controlled by the first two terms on the right hand side of (\ref{the estimates on nonliear term: H1 sense}) except for the case when the gradient is left onto
	\[
	|Q+f|^{p-2} (Q+f) - |Q+g|^{p-2} (Q+g) \text{ and } Q^{p-1} \left( K(Q^{-1} f) - K(Q^{-1} g) \right),
	\]
	requiring us to add the third term on the right hand side of (\ref{the estimates on nonliear term: H1 sense}) to dominate their behaviors.
\end{proof}
When $p \ge 3$, similarly we also obtain that
\begin{lem} 
	If  $p \ge 3$, then for any $f, g \in S(I, \left\langle \nabla \right\rangle L^2)$, we have
	\begin{align}
		& \quad \Big \| \left\langle \nabla  \right\rangle (R(f) - R(g) ) \Big\|_{S'(I, L^2)} \notag \\
		& \lesssim \| \left\langle \nabla \right\rangle (f-g) \|_{S(I,L^2)} \left(  \| f \|_{S(I, \dot{H} ^{s_c})}^{2p-2} +  \| g \|_{S(I, \dot{H} ^{s_c})}^{2p-2} +  \| Q \|_{S(I, \dot{H} ^{s_c})}^{2p-3} \left(  \| f \|_{S(I, \dot{H} ^{s_c})} +  \| g \|_{S(I, \dot{H} ^{s_c})} \right)\right) \notag   \\
		& +  \| f-g \|_{S(I, \dot{H} ^{s_c})} \left( \| \left\langle \nabla \right\rangle f \|_{S(I,L^2)} + \| \left\langle \nabla \right\rangle g \|_{S(I,L^2)}  \right)  \left(  \| f \|_{S(I, \dot{H} ^{s_c})}^{2p-3} +  \| g \|_{S(I, \dot{H} ^{s_c})}^{2p-3} +  \| Q \|_{S(I, \dot{H} ^{s_c})}^{2p-3} \right).
		\label{the estimates on nonliear term: H1 sense for p ge 3}
	\end{align}
	\label{lemma: Strichartz estimate on nonlinear term: difference version for p ge 3}
\end{lem}

\begin{rmk}
	As widely used in Section \ref{uniqueness: section}, by (\ref{critical Strichartz norm can be bounded by H1 Strichartz norm}), we can replace all the terms involving $S(I, \dot{H}^{s_c})$ norm on the right hand side of (\ref{the estimates on nonliear term: H1 sense}),(\ref{the estimates on nonliear term: critical space}) and (\ref{the estimates on nonliear term: H1 sense for p ge 3}) with $\| \left\langle \nabla \right\rangle \cdot \|_{S(L^2)}$.
	\label{remark: we can bound the nonliear by L2 Strichartz norm with gradient}
\end{rmk}

\bibliographystyle{siam}
\bibliography{Bib-1}

\end{document}